\documentclass[a4paper,3pt]{amsart}
\usepackage[foot]{amsaddr}

\usepackage{graphicx,amsfonts,amsmath}
\usepackage[usenames,dvipsnames]{color} 
\usepackage{freefemjs}
\usepackage[titletoc]{appendix}

\usepackage[a4paper, total={6.5in, 8in}]{geometry}
\newtheorem{thm}{Theorem}[section]

\newtheorem{lem}[thm]{Lemma}

\newtheorem{prop}[thm]{Proposition}
\theoremstyle{definition}
\newtheorem{defn}[thm]{Definition}
\theoremstyle{remark}
\newtheorem{rem}[thm]{Remark}
\numberwithin{equation}{section}
\usepackage[dvipsnames]{xcolor}
\newcommand{\blue}[1]{{\color{black%blue
}#1}}
\newcommand{\red}[1]{{\color{black%red
}#1}}
\usepackage{freefemjs}
\usepackage{bm}

\usepackage{float}
\usepackage{subfigure}
\usepackage{amssymb,amsthm,mathtools}
\usepackage[colorlinks,allcolors=black]{hyperref}
\usepackage[nameinlink,noabbrev]{cleveref}

\crefname{assumption}{assumption}{assumptions}
\usepackage{changes}

\newcommand{\green}[1]{{\color{black}#1}}%{{\color{magenta}#1}}
\newcommand{\magent}[1]{{\color{black}#1}}%{{\color{magenta}#1}}%\newcommand{\redhot}[1
\usepackage{systeme}
\usepackage{graphicx}
\usepackage[]{units}
\usepackage{color}
\usepackage{mathrsfs}
\usepackage{bm}

\usepackage{float}
\usepackage{amsmath}
\usepackage{cleveref}
\usepackage{isomath}
\usepackage{amsthm}

\usepackage{amssymb}
\usepackage{pifont}
\usepackage{grffile}
\usepackage{multicol,multirow}
\usepackage{tikz}
\usetikzlibrary[patterns]
\usepackage{comment}
\usepackage{array}
\usepackage[lined]{algorithm2e}
\usepackage{pgfplots}
\usepgfplotslibrary{fillbetween}

\usepackage{caption}
\usepackage[shortlabels]{enumitem}
\usepackage{fancyvrb}
\usepackage{float}
\usepackage{booktabs}
\usepackage[abs]{overpic}
\crefname{equation}{}{}
\crefname{enumi}{}{}
\crefname{figure}{Figure}{Figure}
\crefname{subsection}{subsection}{subsections}
\crefname{lemma}{Lemma}{Lemma}
\crefname{proposition}{Proposition}{Proposition}
\numberwithin{equation}{section}
\numberwithin{figure}{section}
\definecolor{vargreen}{rgb}{0.0, 0.5, 0.0}

\usepackage{mwe} 
\usepackage{amssymb}

\usepackage[titletoc]{appendix}
\begin{document}
\title[Lotka Volterra PDE systems]{Boundary and distributed optimal control for a  population dynamics PDE model  with %higher order
 discontinuous  in time Galerkin %-stepping 
FEM schemes%diffusive Lotka - Volterra PDE system
%Discontinuous in time Galerkin finite element methods for a diffusive Lotka - Volterra distributed and boundary optimal control PDE system
}%
\author{Efthymios N. Karatzas}%
\address{School of Mathematics, Aristotle University of Thessaloniki, Thessaloniki 15780, Greece}%
\email{ekaratza@math.auth.gr}%

\thanks{}%

%\date{}%
%\dedicatory{}%
%\commby{}%
% ----------------------------------------------------------------
\begin{abstract}
We consider fully discrete finite element approximations for a semilinear optimal control system of partial differential equations in two cases: for
distributed and Robin boundary control. %applications. 
The ecological predator-prey optimal control model is approximated by conforming finite element methods mimicking the spatial part, while a discontinuous Galerkin method is used for the time discretization. We investigate the sensitivity of the solution distance from the target function, in cases with smooth and rough initial data. We employ low, and higher-order polynomials in time and space whenever proper regularity is present. The approximation schemes considered are with and without control constraints, driving efficiently the system to desired states realized using %gradient and 
non-linear gradient methods.
\end{abstract}
\maketitle
% ----------------------------------------------------------------
\section{Introduction}\label{sec:intro}
We examine an optimal control problem {\red{arising}} from chemical kinetics, biology, and ecology in a two-species formulation.
Boundary or distributed  control function ($g_i$) \magent{is applied to} a system ($y_i$) driving it to a desired target (${y_{i,d}}$),  $i=1,2$. We assume the control function space $S$ employed in a set $\Omega \subset \mathbb R ^2$ that denotes an open bounded and convex domain with Lipschitz boundary $\Gamma$.  %Considering the quadratic functional ({\ref{functional_full}),
  $S=\Omega$ refers to distributed and  $S=\Gamma$ to boundary control. The scope of this work is to study the distance between the solution  and the desired target function  and to investigate how the parameters affect the solutions under the minimization of the quadratic functional
  \begin{eqnarray}\label{functional_full}
   \magent{J(y_1,y_2,g_1,g_2):=}
    \left\{ \frac{1}
{2}\int_0^T {\left\| y_1-{y_{1,d}}\right\|_{L^2 (\Omega )}^2  dt
}
+\frac{\gamma _1}{2}\int_0^T {\left\| g_1 \right\|_{L^2 ({\bf{S}}%\Omega
)}^2 dt
}
\right.
\\   +\left.\frac{1}{2}
\int_0^T {\left\| {{y_2-{y_{2,d}}}}\right\|_{L^2 (\Omega )}^2) dt
}
+\frac{\gamma _2}{2}\int_0^T {\left\| g_2 \right\|_{L^2 ({\bf{S}}%\Omega
)}^2dt
}\nonumber
\right\},
\end{eqnarray}
subject to the Lotka-Volterra system with {\textit{distributed control}}
\begin{eqnarray}
%\nonumber
%\hskip-15pt
& &{\magent{{{y_{1t}}}}}  = \epsilon_1 \Delta y_1 +  (a-b y_2)y_1 +g_1+f_1 {\text{ in }}(0,T] \times \Omega, \quad
 y_1 = 0 {\text{ on }}(0,T] \times \Gamma, \quad
\label{LoD1}\\%\nonumber
%\hskip-15pt
& &{\magent{{{y_{2t}}}}}  = \epsilon_2 \Delta y_2 + (cy_1-d)y_2 +g_2+f_2 {\text{ in }}(0,T] \times \Omega, \quad
  y_2= 0 {\text{ on }}(0,T] \times \Gamma, %\\\nonumber
%& &\quad\qquad\quad\qquad\quad\qquad y_1(0,x)=  {y_{1,0}} %{\text{ in }} \Omega
%\quad y_2(0,x)=  {y_{2,0}} {\text{ in }} \Omega .
\label{LoD2}
\end{eqnarray}
or the
Lotka-Volterra system with Robin \textit{boundary control}
\begin{eqnarray}
%\nonumber
%\hskip-15pt
& &{\magent{{{y_{1t}}}}}  = \epsilon_1 \Delta y_1 +  (a-b y_2)y_1 +f_1 {\text{ in }}(0,T] \times \Omega, \quad
 {\magent{y_1}}+\frac{\epsilon_1}{\lambda_1}\frac{\partial y_1 }{\partial {\mathbf n}}= g_1 {\text{ on }}(0,T] \times \Gamma, \quad
\label{LoR1}\\%\nonumber
%\hskip-15pt
& &{\magent{{{y_{2t}}}}}  = \epsilon_2 \Delta y_2 + (cy_1-d)y_2 +f_2 {\text{ in }}(0,T] \times \Omega, \quad
 y_2+\frac{\epsilon_2}{\lambda_2}\frac{\partial y_2 }{\partial {\mathbf n}}= g_2 {\text{ on }}(0,T] \times \Gamma,
 \label{LoR2}%\\\nonumber & &\quad\qquad\quad\qquad\quad\qquad y_1(0,x)=  {y_{1,0}} %{\text{ in }} \Omega \quad y_2(0,x)=  {y_{2,0}} {\text{ in }} \Omega ,
\end{eqnarray} 
enforcing initial data
\[y_1(0,x) =  {y_{1,0}} %{\text{ in }} \Omega
\quad y_2(0,x)=  {y_{2,0}} {\text{ in }} \Omega. \]
Within the next paragraphs, we introduce literature and important past works as well as details related to the dynamics of the considered system. We also present the involved parameters and their physical meaning explaining how they affect the solution. 
\subsection{The physical model and related results}
{There is a wide variety of Lotka-Volterra system versions,  %e.g.  
the predator-prey/defense, the competitive/colonization and/or the
cooperative/resource exchange systems. In general,} the Lotka-Volterra predator-prey model mimics %simulates 
bacteria populations, chemical reactions, and other similar models. {In this manuscript,} we focus on {predator-prey/defense systems in} two space dimensions $x_1$, $x_2$ with two species and concentrations  $y_1$, $y_2$, (\ref{LoD1})-(\ref{LoR2}). The  $y_1$ variable describes the prey's concentration, whereas the $y_2$ variable represents the predator's. The parameters i) $a$ represents the growth rate of $y_1$, ii) $b$ represents the rate of $y_2$ is killing  $y_1$, iii) $c$ represents the growth rate of $y_2$ by chances of killing $y_1$, and iv) $d$ represents the death rate of $y_2$, \cite{Ha97}. The forcing terms $f_1$, $f_2$ and the parameters $\lambda_1$, $\lambda_2$, $\epsilon_1$, $\epsilon_2$, are given data while $\gamma_1$, $\gamma_2$ denote penalty parameters that limit the size of the control and will properly be chosen. 

Several results regarding the analysis of related systems have been recorded %presented 
in the literature. We start with the pioneer works of Holling and Volterra for simple types of predation and parasitism and the response of predators to prey density with its role in mimicry
and population regulation, \cite{Hol59, Hol63,Volterra78}%\TODO{Hol59, Hol63, Volterra78}
. In \cite{GaTre08} %\TODO{GaTre08}
fully discrete estimates, a fully discrete error bound, and rates of convergence for
systems modeling predator–prey interactions, where the local growth of prey is logistic
and the predator displays the Holling type II functional response associated with real kinetics %with interesting problems coming from biology 
is presented. %investigated. 
In %the recent work 
\cite{ChryKo22} %\TODO{{ChryKo22}} 
high-order schemes for a modified predator-prey system proving main stability estimates, under minimal regularity assumptions on the given data, are considered obtaining a-priori error estimate. In \cite{DieGaTre17} %\TODO{DieGaTre17} 
predator–prey Holling type dynamics using implicit-symplectic schemes
are analyzed proving optimal a priori error estimates. % is an improvement on previous theoretical results using standard implicit–explicit schemes are considered.
We also report the books \cite{Ha97}, \cite{Mu03}%\TODO{Ha97,Mu03}
, and the references therein, presenting ideas starting from the classical work of \cite{Volterra78} %\TODO{{Volterra78}} 
and thereafter.

Since Lotka-Volterra systems are considered reaction-diffusion systems models, we report related works resembling the predator-prey system. Such is the discontinuous Galerkin (dG) methods for mass transfer through semi-permeable membranes \cite{CaGeJe2022}%\TODO{CaGeJe2022-SIAM}
, Brusselator systems \cite{ChryKaKo19}, 
%\TODO{ChryKaKo19} 
and predator–prey models with the addition of cross-diffusion blows-up on surfaces employing an implicit-explicit (IMEX) method \cite{FriMadSguVen17}. 
%\TODO{FriMadSguVen17}
 For $\lambda$–$\omega$ type reaction-diffusion systems  the interested reader could see  \cite{GaBlo05}, %\TODO{GaBlo05}, 
for the forced Fisher equation that models the dynamics of gene selection/migration for a diploid population with two available alleles in a multidimensional habitat and in the presence of an artificially introduced genotype \cite{GuHoZhu06}%\TODO{GuHoZhu06}
, error estimates for the semidiscrete Galerkin approximations of the FitzHugh-Nagumo equations \cite{Ja92}%\TODO{Ja92}
, and epitaxial fourth order growth model with %exponential time differencing schemes and 
implicit-explicit time discretization is handled in %Na mpei isos stas voltera gia alalgi min einai idio me tochrysafino fitzugh)
 \cite{JuLiQiaoZha18,KaRo21}%\TODO{JuLiQiaoZha18-KaRo2021}
. The performance of several % of the best known 
linear multi-step IMEX schemes for reaction-diffusion problems in pattern 
formation are discussed in \cite{Ru95}, %\TODO{Ru95 reaction implicit explicit 1995}, 
and finite volume element approximation of an inhomogeneous Brusselator model with cross-diffusion are analyzed in \cite{LinRuizTian14}%\TODO{LinRuizTian14}
.

For continuous or dG in space and/or dG in time parabolic  
PDE discretizations %works %for continuous and/or discontinuous  in space/% and discontinuous in time 
%for  problems 
we %briefly 
quote %^are %we refer to
the works \cite{AkrMakr2022,AkrMakr2014,DelfHagTro81,EstepLarsson93,ErickssonJohnson95%,ArKaKa22
}, %\TODO{AkrMakr2022, AkrMakr2014-dG+cG in time DelfHagTro81 EstepLarsson93- continuous in space and discontinuous in time, ErickssonJohnson95, see also \TODO{ArKaKa22} for unfitted dG in space methods for stationary Stokes fluid} 
while nonlinear parabolic, implicit, and implicit-explicit multistep ﬁnite element methods can be found in \cite{AkrCrMakr1998, AkrCr2004}.
%
%Finally, more works for the dG in time methods for the parabolic type of PDEs are...

In \cite{CaVa22},
%\TODO{Carreon, and F. Valdez,} 
a metaheuristic optimization method based on the Lotka-Volterra system equations where 
the interaction that exists in nature between plants through a network of fungi, and the Lotka-Volterra system of equations model the different types of relationships, %and an optimization method for this type of interaction/symbiosis between plants and fungi by means of the Lotka-Volterra system of equations
 is examined. 
In \cite{Ib17}, %\TODO{Optimal control of the Lotka–Volterra system: turnpike property and numerical simulations} 
the Lotka-Volterra model is investigated as two coupled ordinary differential equations representing the interaction of two species, a prey one and a predator one formulating an optimal control problem by adding the effect of hunting both species as the control variable %in long time intervals, 
implementing a single shooting method to solve the optimization problem. %, taking the middle of the time interval as a starting point 
Moreover, in \cite{GaMo97}
%\TODO{GaMo97},\TODO{uniqueness} 
\magent{the uniqueness of the optimal control is obtained by assuming a large crowding effect of the species}.

Optimal control, with PDEs as constraints, works also are \cite{GaTre07}%\TODO{GaTre07}
, where results from semigroup theory and a semi-implicit Galerkin finite element method are employed, based on the kinetics satisfying a Lyapunov-type condition of a nonlinear plankton-fish optimal control reaction-diffusion system. % with %numerical solutions obtained with the aid of 
%. % with piecewise linear continuous basis functions. 
In \cite{ChKa14} %\TODO{ChKa14} 
an optimal control problem for a parabolic PDE also with nonsmooth initial data  and discontinuous in time discretization for
Robin boundary control and \cite{ChKa12, NeVe12} for distributed control associated with semilinear parabolic PDEs are presented, see also references therein. We refer to \cite{GuMa06}, %\TODO{GuMa06}, 
\cite{ChKa15} %\TODO{ChKa15}  
for fluids control, to \cite{HoRo21} %\TODO{HoRo21} %A priori error estimates for the finite element approximation of a 
for nonsmooth optimal control problems governed by coupled semilinear PDE-ODE systems and \cite{GarcHiKah19}, %\TODO{GarcHiKah19} 
 \cite{ChPla23} %\TODO{ChPla23} 
for
optimal control of time-discrete %two -
phase 
flow driven by  diffuse %-interface 
models. %, %{\em ESAIM, Control Optim. Calc. Var.}, {\bf 25(13)}, 1-31, 2019. 
More optimal control works are \cite{KaKaTra23} %\TODO{KaKaTra23} 
considering elliptic PDE optimal control systems with unfitted mesh approaches based on reduced order modeling techniques, and \cite{ArKa22} %\TODO{ArKa22} 
 concerning random geometrical morphings on a cutfem dG type optimal control framework.

\magent{
The scopes  of the present manuscript is to investigate continuous and fully discrete finite element formulations for a semilinear optimal control system of a predator-prey Lotka-Volterra type model. A conforming finite element for the spatial part has been adopted, as well as a discontinuous Galerkin approximation for the time part. Cases with and without control constraints have been considered, and for the handling of distributed or Robin boundary optimal controls, non-linear gradient methods have been employed.
%
%The predator-prey optimal control model is approximated using a discontinuous Galerkin method in time, combined with conforming finite element methods for the spatial part 
 A study is presented on the sensitivity of the solution distance from the target function in cases with smooth and rough initial data. Numerical tests are conducted to illustrate the effectiveness of the approach and its robustness with respect to the regularity of the initial data, as well as, from the  space and/or time higher order polynomial approximation point of view whenever it is applicable. 
%
%A two coupled nonlinear parabolic equations with two, $g_1$ and $g_2$,
%internal controls  and two Robin boundary controls have been considered.
%
%The full discretization in time and space has been defined , with a discontinuous Galerkin discretization
%in time and a conforming finite element method in space, characterizing the necessary and sufficient conditions for first and second derivative of the cost functions.
%
The aforementioned numerical experiments in cases of distributed and boundary controls and implemented with the FreeFem++ software, show also the effect of control constraints% has been used% also considering higher-order schemes
. In both distributed and boundary controls several results have been reported with respect to the distance from the targets, with smooth and non smooth initial conditions, and also the effects of controls on the optimal state solutions are presented.

The necessary and sufficient conditions characterized by first and second derivative of the cost functions have been introduced with a proof of the challenging case of constrained in specific intervals optimal control. Furthermore, some extended standard but necessary results have been presented. %, while, for sake of shortness, some of their proofs have been omitted.
 %Although the techniques used for the predator-prey systems in this paper are known in the literature and used for other parabolic or other equations, see also works of the author, the results in our opinion are interesting and not trivial.% and deserves to be published in the CAMWAjournal.
}

The document is organized as follows. In Section \ref{sec:intro}, the continuous system in a strong formulation is introduced while in Section \ref{sec:background}, related results and theoretical tools and preliminaries are presented for the continuous and discretized formulation.
%\TODO{notation and Preliminaries with basic discretisation tools.}
%
In Section \ref{sec:Distributed}, the continuous and discretized weak form in the dG in time framework and the optimality system is reported while distributed control is applied with control constraints. Additionally, the first-order optimality conditions are demonstrated and the second-order optimality conditions and sufficient conditions for optimality are proved. 
%\TODO{Distributed control, Dirichlet zero boundary, and control constraints.
%
%The continuous and discrete control problem/optimality system.
%
%The fully-discrete distributed optimal control problem
%
%3.2. Second order optimality conditions.
%
%3.2.2. Sufficient conditions for optimality
%}
In the following, in Section \ref{sec:Robin1} we extend the study to Robin boundary control, we introduce the optimality system and the fully discrete Robin boundary type optimal control problem.

Finally, in Section \ref{sec:experiments} several numerical experiments expose tables reporting results for the distance norm for the two species and $J$ functional values, for the distributed control case: ($\alpha$) with control constraints and constant polynomials in time and linear in space, %k = 0, l = 1.
%
%5.1.2. W
($\beta$) without control constraints and linear in time, quadratic in space, %k = 1, l = 2.
%
%5.1.3. Q
and ($\gamma$) a qualitative study and nullclines are presented.
%
%5.2. B
For the boundary Robin control case:
%5.2.1. C
%rough/
low regularity initial data employing constant in time%(k = 0), 
/linear in space polynomials, % (l = 1), 
%
%5.2.2. L
as well as linear in time/%(k = 1), 
linear in space, %the abusive case of 
linear in time/quadratic in space polynomials %(l = 1)
and
%5.3. T
the heuristic control algorithm used are demonstrated.

To our best knowledge, {there are no predator-prey related investigations} considering discontinuous Galerkin in time in which optimal control techniques are applied with low and higher order dG in time discretization, as well as with non-smooth data\magent{, and the manuscript provides a contribution in the field indicating a strategy that can be extended to more complicated situations.} 

\section{Background}\label{sec:background}
\subsection{Notation}
Following the literature, see e.g. \cite[Chapter 5]{E98}, we use the  Hilbert spaces standard notation $L^2(\Omega)$, $H^s(\Omega)$, $0<s \in \mathbb {R}$,
\magent{$H^1_0(\Omega) \equiv \{ v \in H^1(\Omega) : v|_{\Gamma} =0 \}$%, for the related norms and inner products
.}
We denote by $H^{-1}(\Omega)$ the dual of $H^1_0(\Omega)$, $H^{*}(\Omega)$ the dual of $H^1(\Omega)$
and the corresponding duality pairing by $\langle \cdot,\cdot \rangle_{H^1(\Omega)^*,H^1(\Omega)} \equiv \langle \cdot,\cdot \rangle$. We will frequently use the space $H^{1/2}(\Gamma)$, its dual denoted by $H^{-1/2}(\Gamma)$, and their duality pairing denoted by $\langle .,. \rangle_{H^{-1/2}(\Gamma), H^{1/2}(\Gamma)} \equiv \langle \cdot,\cdot \rangle_{\Gamma}$. Finally, the standard notation $(\cdot,\cdot)$, $(\cdot,\cdot)_{\Gamma}$ will be used for the $L^2(\Omega) \equiv H^0(\Omega)$ and $L^2(\Gamma)$ inner products respectively.
For any of the above Sobolev spaces, we define the space-time spaces %$L^p [0, T;X]$, $L^\infty  [0, T;X]$, $C[0, T;X]$ and $H^1[0, T;X]$
 in a standard way.% (see e.g. \cite[Chapter 5]{E98}).
%endowed with the norms
%$$ \| u \|_{L^p [0,T;X]} \equiv \big( \int_0^T \| u(t) \|_X^p dt \big)^{1/p}, \quad \| u \|_{L^\infty  [0,T,X]} \equiv \mathop {ess\sup
%}\limits_{0 \leqslant t \leqslant T} \| {u(t)}\|,
%$$
%$$
%\| u \|_{H^1 [0,T;X]} \equiv \big( \int_0^T  \| u(t)\|_X^2
%+ \| u'(t) \|_X^2  dt \big)^{1/
%2}, \quad \|u\|_{C[0,T;X]} \equiv \max_{t \in [0,T]} \|u(t)\|_{X}.
%$$

%Also w
%We denote with
%$H^{2,1}([0,T]\times\Omega)=\{ v\in L^2[0,T;L^2(\Omega)]:\frac{\partial v}{\partial x_i},\frac{\partial ^2 v}{\partial x_i\partial x_j},\frac{\partial v}{\partial t}\in L^2[0,T;\Omega],1\le i,j\le 2\}$
 %and the
%corresponding duality pairing by $\langle .,. \rangle$.

 For any
Banach space $X$, we denote by $L^p[0,T;X], L^{\infty}[0,T;X]$ the
standard time-space spaces, endowed with norms:
$$\|v\|_{L^p[0,T;X]} = \left ( \int_0^T \|v\|^p_{X} dt \right )^{\frac{1}{p}},
\quad \|v\|_{L^{\infty}[0,T;X]} = \mbox{ ess$ \sup_{t \in [0,T]}
\|v\|_{X}$}. $$ The set of all continuous functions $v : [0,T]
\rightarrow X$, is denoted by $C[0,T;X]$, with norm defined by
$\|v\|_{C[0,T;X]} = \max_{t \in [0,T]} \|v(t)\|_{X}.$ Finally, we
denote in a classic way  $H^1[0,T;X]$ and the norm,
\begin{eqnarray*}\|v\|_{H^1[0,T;X]} &=& \left ( \int_0^T \|v\|^2_{X} dt \right )^{\frac{1}{2}}
+ \left ( \int_0^T \|v_t\|^2_{X} dt \right )^{\frac{1}{2}} \leq C
< \infty.
\end{eqnarray*}
%\begin{eqnarray*}
   %   % \nonumber to remove numbering (before each equation)
 %\|v\|_{H^{2,1}[0,T;X]}&=&%\left
   %     \Big\{\int^T _{0}\int_{\Omega}\Big (|v|^2+\left|\frac{\partial y}{\partial t}\right|^2\Big )dxdt \\
     %  & & +\sum\limits_{i = 1}^2 \int^T_{0}\int_{\Omega} \left|\frac{\partial v}{\partial x_i}\right|^2dxdt +\sum\limits_{i,j = 1}^2 \int^T_{0}\int_{\Omega} \left|\frac{\partial ^2 v}{\partial x_ix_j}\right|^2dxdt
      %\right
     % \Big\}^{1/2}.
     % \end{eqnarray*}
%and by
We will frequently use the spaces 
for %the Brusselator
the aforementioned system applying distributed control %Fitzugh Nagumo system with
 and zero Dirichlet conditions $W_D(0,T): = L^2 [0,T;H^1(\Omega)] \cap L^{\infty} [0,T;L^{2} (\Omega)]$ with norm defined as $\|u\|_{W_D(0,T)}^2 \equiv \| u\|_{L^2 [0,T;H^1 (\Omega )]}^2  + \|u\|_{L^{\infty} [0,T;L^{2} (\Omega )]}^2$
while for problems employing Robin boundary control %(BC)
 $W_R(0,T): = L^2 [0,T;H^1(\Omega)] \cap L^{\infty} [0,T;L^{2} (\Omega)]\times L^{2} [0,T;L^{2} (\Gamma)]$ with norm defined as $\|u\|_{W_R(0,T)}^2 \equiv \| u\|_{L^2 [0,T;H^1 (\Omega )]}^2  + \|u\|_{L^{\infty} [0,T;L^{2} (\Omega )]}^2+ \|u\|_{L^{2} [0,T;L^{2} (\Gamma )]}^2$, see e.g. \cite{ChKa14}, \cite{ChKa12}.
\subsection{Preliminaries and discretization tools}\label{d.o.c.s.}
We consider a family of triangulations, $\{\mathcal{T}_h\}_{h>0}$ of $\Omega$, defined in the standard way,  \cite{Ci}. We associate two parameters $h_T$ and $\rho_T$ for each element $T \in
\mathcal{T}_h$, denoting the diameter of the set $T$,
and the diameter of the largest ball contained in $T$ respectively. The size of the mesh is denoted by
$h=\max_{T\in\mathcal{T}_h}h_T$. The following standard properties of the mesh will be assumed: %\\
$\alpha)$ There exist two positive constants
$\rho_{\mathcal{T}}$ and $\delta_{\mathcal{T}}$ such that
$
\frac{h_T}{\rho_T}\leq
    \rho_{\mathcal{T}}\ \mbox{ and } \ \frac{h}{h_T}\leq
    \delta_{\mathcal{T}} \ \ \forall T\in \mathcal{T}_h \mbox{ and } \forall h>0,
$ %\\
$\beta)$  Given $h$, let $\{ T_j \}_{j=1}^{N_h}$ denote the family of triangles belonging to $\mathcal{T}_h$ and having one side included on the boundary $\Gamma$. Thus, if the vertices of $T_j \cap \Gamma$ are denoted by $x_{j,\Gamma}$, $x_{j+1,\Gamma}$ then the straight line $[x_{j,\Gamma}, x_{j+1,\Gamma}] \equiv T_j \cap \Gamma$. Here, we also assume that $x_{1,\Gamma} = x_{N_h+1, \Gamma}$.

Considering the mesh $\mathcal{T}_h$ we construct the finite dimensional spaces $U_{h,{{D}}} \subset H^1_0(\Omega)$, $U_{h,{{R}}} \subset H^1(\Omega)$ employing piecewise polynomials in $\Omega$. Standard approximation theory assumptions are assumed on these spaces. In particular, for any $v \in H^{l+1}(\Omega)$, there exists an integer $\ell \geq 1$, and a constant $C>0$ independent of $h$ such that:
$%\begin{equation*} \label{eqn:3.1}
\inf_{v_h \in U_h}\|v - v_h\|_{H^s(\Omega)} \le Ch^{l+1-s}\|v\|_{H^{l+1}(\Omega)}, \ \mbox{ for } 0 \le l \le \ell \mbox{ and } s=-1,0,1,
$ %\end{equation*}
and inverse inequalities on quasi-uniform triangulations, i.e., there exist constants $C\ge 0$, such that $\|v_h\|_{H^1(\Omega)} \leq C/h \|v_h\|_{L^2(\Omega)}$, and $\|v_h\|_{L^2(\Omega)} \leq C/h \|v_h\|_{H^1(\Omega)^*}$.
%Therefore, under the above assumptions on the mesh regularity, and the choice of the approximation spaces, we obtain the following standard approximation properties. In particular, for the case where piecewise linear polynomials $(l=1)$ are being used for the construction of $U_h$, we deduce the estimates: for all $v \in H^2(\Omega)$, there exists $C>0$ (independent of $h$) such that,
%\begin{equation} \label{eqn:3.1}
%\inf_{v_h \in U_h}\|v - v_h\|_{H^1(\Omega)} \le Ch\|v\|_{H^{2}(\Omega)}, \qquad \inf_{v_h \in U_h}\|v - v_h\|_{L^2(\Omega)} \le Ch^{2}\|v\|_{H^{2}(\Omega)}.
%\end{equation}
{{Fully discrete}} approximations will be constructed on a quasi-uniform partition
$0=t^0 < t^1 < \ldots < t^N=T$ of $[0,T]$, i.e., there exists
a constant $0<\theta<1$ such that
$\min_{n=1,..,N} (t^n-t^{n-1}) \geq \theta \max_{n=1,...,N} (t^n-t^{n-1})$.
We also use the notation $\tau^n=t^n-t^{n-1}, \tau= \max_{n=1,...,N} \tau^n$
and we denote by
${\mathcal P}_k[t^{n-1},t^n; U_{{h{{,D}}}} \text{ or } U_{{h{{,R}}}}]$ the space of polynomials of degree $k$ or
less, having values in  $U_{{h{\blue{,D}}}}$ or $U_{{h{\blue{,R}}}}$. We seek approximate solutions %{\red{which}} %who
belonging to the spaces
$$
{\mathcal U}_{h,{\red{D}}} = \{y_{ih} \in L^2[0,T;H^1_0(\Omega)] :
        y_{ih}|_{(t^{n-1},t^n]} \in {\mathcal P}_k[t^{n-1},t^n;  U_{h,{{D}}}%{\red{U%^n
%_h}}
] \},
$$
$$
{\mathcal U}_{h,{\blue{R}}} = \{y_{ih} \in L^2[0,T;H^1(\Omega)] :
        y_{ih}|_{(t^{n-1},t^n]} \in {\mathcal P}_k[t^{n-1},t^n; U_{h,{\blue{R}}}] \}.
$$

By convention, the functions of  ${{\mathcal U}_{h,{\blue{D}}}}$, ${{\mathcal U}_{h,{\blue{R}}}}$ are left
continuous with right limits and hence will write $y^n_{ih} \equiv y^n_{{ih}-}$ for $y_{ih}(t^n) = y_{ih}(t^n_-)$, and $y^n_{{ih}+}$ for $y_{ih}(t^n_+)$,
while the jump at $t^n$, is denoted by $[y^n_{ih}] = y^n_{{ih}+} - y^n_{{ih}}$.
In the above definitions, we have used the following notational abbreviation, $y_{{ih},\tau} \equiv y_{ih}$,
${\mathcal U}_{h,\tau,{\red{D}}} \equiv {\mathcal U}_{h,{\red{D}}}$, ${\mathcal U}_{h,\tau,{\red{R}}} \equiv {\mathcal U}_{h,{\red{R}}}$. For the time discretization, we will use the lowest order scheme ($k=0$) which corresponds to the discontinuous Galerkin variant of the implicit Euler.  This approach gives good results also in cases with limited regularity which is acting as a barrier in terms of developing estimates of higher order. Although, we will illustrate some cases related to higher-order schemes.
%The fully-discrete approximations are constructed on a partition $0=t^0 <  t^1 < \ldots < t^N=T$ of $[0,T]$. On each time interval $(t^{n-1},t^n]$, of length $\tau_n \equiv t^n-t^{n-1}$, a subspace $U^n_h$ of $H^1_0(\Omega)$ is specified, and it is assumed that each $U^n_h$ satisfies the classical approximation theory results (see e.g. \cite{Ci}). We also assume that the time-steps are quasi-uniform, i.e., there exists $0 \leq \theta \leq 1$, such that $\min_{n=1,...,N} \tau_n \geq \theta \max_{n=1,...,N} \tau_n$. Now, we seek approximate solutions that belong to the space
%$$
%{\mathcal U}_{h,{\red{D}}} = \{y_{ih} \in L^2[0,T;H^1_0(\Omega)] :
%        y_{ih}|_{(t^{n-1},t^n]} \in {\mathcal P}_k[t^{n-1},t^n; {\red{U%^n
%_h}}] \}.
%$$
%Here ${\mathcal P}_k[t^{n-1},t^n; U{\red{%^n
%_h}}]$ denotes the space of
%polynomials of degree $k$ or less having values in $U{\red{%^n
%_h}}
%$. We
%also use the following notational abbreviation, $y_{i_{h,\tau}} \equiv y_{ih}, i=1,2$,
%${\mathcal U}_{h,\tau,{\red{D}}} \equiv {\mathcal U}_{h,{\red{D}}}$ etc.
 The
discretization of the control can be effectively achieved through
the discretization of the adjoint variable $\mu$. However, we
point out that the only regularity assumption on the discrete
control is $g_{1h},g_{2h} \in L^2[0,T;L^2(\Omega)]$ {{or $g_{1h},g_{2h} \in L^2[0,T;L^2(\Gamma)]$ for the distributed and boundary control respectively}}.
 %
 %
%
%$W(0,T)$ the solution space $ W(0,T) =
%L^2[0,T;H^1_0(\Omega)] \cap H^1[0,T;H^{-1}(\Omega)] $ with norm
%$$ \|v\|^2_{W(0,T)} = \|v\|^2_{L^2[0,T;H^1(\Omega)]} +
%\|v_t\|^2_{L^2[0,T;H^{-1}(\Omega)]}.$$
%
\section{Distributed control, Dirichlet zero boundary and control constraints.}\label{sec:Distributed}
This section is devoted to % we turn our attention to 
distributed control with Dirichlet zero boundary, %and 
the enriched problem actually involves the functional $J(y_1,y_2,g_1,g_2)$ as it is described in (\ref{functional_full}) for $S=\Omega$, 
%\begin{eqnarray}\label{1.1}
%   J(y_1,y_2,g_1,y_2) &=& \frac{1}
%{2}\int_0^T {\left\| y_1-{y_{1,d}}\right\|_{L^2 (\Omega )}^2  dt
%}
%+\frac{\gamma _1}{2}\int_0^T {\left\| g_1 \right\|_{L^2 (\Omega )}^2 dt
%}
%\nonumber\\&+&\frac{1}
%{2}\int_0^T {\left\| {{y_2-{y_{2,d}}}}\right\|_{L^2 (\Omega )}^2) dt
%}
%+\frac{\gamma _2}{2}\int_0^T {\left\| g_2 \right\|_{L^2 (\Omega )}^2dt
%}
%\end{eqnarray}
 subject {\red{to % the state constraints (\label{LoD1})- (\label{LoD2})
%\begin{eqnarray}\label{1.2Lo}
%\nonumber
%%\hskip-15pt
 % & &\frac{\partial{y_1}}{\partial t} - \epsilon_1 \Delta y_1 -  (a-by_2)y_1 =f_1+g_1{\text{ in }}(0,T] \times \Omega, \quad
  %y_1 = 0 {\text{ on }}(0,T] \times \Gamma, \quad
%\\%\nonumber
%%\hskip-15pt
 % & &\frac{\partial{y_2}}{\partial t}  -\epsilon_2 \Delta y_2 +(cy_1-d)y_2=f_2+g_2{\text{ in }}(0,T] \times \Omega, \quad
  %y_2= 0 {\text{ on }}(0,T] \times \Gamma, \\\nonumber
%& &\quad\qquad\quad\qquad\quad\qquad y_1(0,x)=  {y_{1,0}} %{\text{ in }} \Omega
%\quad y_2(0,x)=  {y_{2,0}} {\text{ in }} \Omega ,
%\end{eqnarray}
 the}} control constraints
$$
g_{ia} \leq g_i(t,x) \leq g_{ib} \text{ for a.e.} (t,x) \in (0,T) \times \Omega \text{,  where}\,\, {g_{ia}},{g_{ib}} \in {\mathbb R}, \,\,i=1,2. $$
\subsection{The continuous and discrete control problem/optimality system}{\label{subsec:3.1}}
\magent{%SXOLEIO 2.19 
Below, we state the optimality system which consists of the state equation given in the weak form, % (\ref{eqn:2.3})
the adjoint, and the optimality condition applying distributed control. Particularly, first-order necessary conditions (optimality system) of the above optimal control problems are demonstrated. %, we also refer to \cite{ChKa12} for a simpler . %, and the convergence rate are expected to be similar to the uncontrolled problem, so we omit the proves. %Also the second order necessary conditions are proved for the Fitzhugh Nagumo systems %. %but and we suggest the reader to the paragraph [Section 5] in this work to see the proof of the second order conditions in the quite  similar Fitzugh - Nagumo semilinear problem but
%with zero Dirichlet conditions.
The aforementioned systems %, and the numerical experiments that follow
 are introduced by employing a discontinuous in time Galerkin (dG) scheme and a conforming Galerkin method in space. The corresponding optimality system (first-order necessary conditions) consists of a primal (forward in time) equation system and an adjoint (backward in time) equation system which are coupled through an optimality condition, and non-linear terms  %, see e.g. \cite{ChKa12}
 as it is described in  Lemma \ref{lem:2.5}%\ref{Blem:2.5}
.
The main aim is to show that the dG approximations of the optimality system exhibit good behavior and to examine the crucial matter of the distance of the solution from the desired target.}

We begin by stating the weak formulation of the state equation. Given $f_1$, $f_2\in L^2\left [0,T;H^{-1}(\Omega)\right ]$, controls $g_1,g_2\in L^2\left[0,T;L^{2}(\Omega)\right ]$, and {{initial}}  states ${y_{1,0}}$, ${y_{2,0}} \in L^2(\Omega)$
we seek $y_1,y_2\in L^\infty[0,T;L^2(\Omega)] \cap L^2[0,T;H_0^{1}(\Omega)]$ such that for a.e. $t\in(0,T]$, and for all
$v \in H^1(\Omega)$
\begin{eqnarray}
 %\qquad
 \left\langle {{\magent{y_{1t}}} ,v} \right\rangle  + \epsilon_1(\nabla y_1,\nabla v )  -((a-by_2)y_1,v)=      \left\langle {f_1,v } \right\rangle+ \left< {g_1,v } \right > \quad\mbox{ and}\quad
  \left( {y_1(0),v } \right) &=& \left( {{y_{1,0}} ,v } \right),\nonumber\\
 %\qquad
   \left\langle {\magent{{y_{2t}}} ,v} \right\rangle  + \epsilon_2  (\nabla y_2, \nabla v )  -((cy_1-d)y_2,v)= \left\langle {f_2,v } \right\rangle+\left <g_2,v \right >  \quad\mbox{ and}\quad
  \left( {y_2(0),v } \right) &=& \left( {{y_{2,0}} ,v } \right).\nonumber%\label{eqn:2.1}
\end{eqnarray}
An equivalent \magent{well-posed} weak formulation which is more suitable for the analysis of dG schemes is to seek
$y_1,y_2\in\magent{W_D(0,T) = L^2 [0,T;H^1(\Omega)] \cap L^{\infty} [0,T;L^{2} (\Omega)]}$ such that for all $v \in L^2[0,T;H^1(\Omega)] \cap H^1[0,T;H^{-1}(\Omega)]$,
\begin{eqnarray}
&&  (y_1(T),v(T)) + \int_0^T {\left( { - \left\langle {y_1,v_t } \right\rangle  + \epsilon_1 \left( {\nabla y_1,\nabla v } \right)-\left( {(a-by_2)y_1,v } \right)
       }    \right)} dt \nonumber \\
&& \qquad   = ({y_{1,0}} ,v(0)) + \int_0^T {  {\left\langle {f_1,v }  \right\rangle  }  } dt+ \int_0^T {\left( {  g_1,v  } \right)} dt, \label{eqn:2.2}
\\
&&  (y_2(T),v(T)) + \int_0^T {\left( { - \left\langle {y_2,v_t } \right\rangle  + \epsilon_2 \left( {\nabla y_2,\nabla v } \right)
      -({(cy_1-d)y_2,v } ) }    \right)} dt \nonumber \\
&& \qquad   = ({y_{2,0}} ,v(0)) + \int_0^T { { \left <f_2,v \right >}  } dt+ \int_0^T {\left( {  g_2,v  } \right)} dt.\label{eqn:2.1}
\end{eqnarray}

The control to state mapping $G: L^2[0, T;{L^2(\Omega)\times L^2(\Omega)]} \to W_D(0, T)\times W_D(0, T)$, which associates to each control $g_1, g_2$ the corresponding state $G(g_1,g_2)$ $ = ({y}_{g_1},{y}_{g_2})$ $\equiv$ $(\magent{y_1(g_1,g_2)},\magent{y_2(g_1,g_2)})$ via (\ref{eqn:2.2})--(\ref{eqn:2.1}) is well defined, and continuous, so does the cost functional, frequently denoted to by its reduced form, $J(y_1,y_2,g_1,g_2) \equiv J(\magent{y_1(g_1,g_2)},\magent{y_2(g_1,g_2)}): L^2[0,T;{L^2(\Omega)}] \to \mathbb R$. %is also well-defined and continuous.
\begin{defn}\label{defn:2.2}
Let ${f_1, f_2}\in L^2 [0,T;H^{-1} (\Omega )]$, ${y_{1,0}, y_{2,0}} \in L^2(\Omega)$, and ${y_{1,d}, y_{2,d}} \in L^2[0,T;L^2(\Omega)]$\magent{, $g_{ia}$, $g_{ib} \in \mathbb{R}$}  be given data.
Then, the set of admissible controls denoted by ${\mathcal A}_{ad}$ for the corresponding distributed control problem with constrained controls
takes the form:
%\begin{enumerate}
%\item {\it Unconstrained Controls}:
%${\mathcal A}_{ad} \equiv L^2[0,T;L^2(\Gamma)]$.
%\item {\it Constrained Controls}:
${\mathcal A}_{ad} = \{ {g_i} \in L^2[0,T;L^2(\Omega)]: {{g_{ia}}} \leq {g_i}(t,x) \leq {{g_{ib}}}, i=1,2 \text{ for a.e. } (t,x) \in (0,T) \times \Omega \}.$
%\end{enumerate}
The pairs $(\magent{y_i(g_1,g_2)},g_i) \in W_D(0,T) \times {\mathcal A}_{ad},i=1,2$,  is said to be an optimal solution if $J(\magent{y_1(g_1,g_2)},\magent{y_2(g_1,g_2)},g_1,g_2)\le J(\magent{w_1(h_1,h_2),w_2(h_1,h_2),}h_1,h_2)$\magent{, for all $(h_1 , h_2 ) \in \mathcal{A}_{ad }\times \mathcal{A}_{ad}$, $(w _1 (h _1\magent{,h_2}), w _2 (\magent{h_1,}h _2 )) \in W_ D (0, T ) \times W_ D (0, T )$ being the solution of problem (\ref{eqn:2.2})--(\ref{eqn:2.1}) with control functions $(h _1 , h _2 )$}%\forall (w_i(h_i),h_i) \in W_D(0,T) \times {\mathcal A}_{ad}, i=1,2
.
\end{defn}%\newline
We will occasionally abbreviate the notation $y_i \equiv {y}_{g_i} \equiv {y_i}( \magent{g_1,g_2}),i=1,2$. An optimality system of equations can be derived by using standard techniques under minimal regularity assumptions on the given initial data and forcing term; see for instance \cite{ChKa12}% or \cite[Section 2]{ChGuHo06}
. Here, we adopt the notation and framework of \cite[Section 2]{CaCh12}. We first state the basic differentiability property of the cost functional.
\begin{lem} \label{lem:2.4}
The cost functional $J:L^2[0,T;{L^2(\Omega)}] \to \mathbb R$ is of class $C^{\infty}$ and for every $g_1,g_2,y_1,y_2 \in L^2[0,T;{L^2(\Omega)}]$,
\[J^{'}(g_1,g_2)(u_1,u_2) = \int_0^T \int_\Omega \Big (( {\mu_1} (\magent{g_1,g_2}) + \gamma_1 g_1) u_1 , ( {\mu_2} (\magent{g_1,g_2}) + \gamma_2 g_2) u_2\Big ) dxdt, \]
where $\mu _i(\magent{g_1,g_2}) \equiv {\mu}_{g_i} \in W_D(0,T), i=1,2,$ is the unique solution of following problem: For all $v \in L^2[0,T;H^1(\Omega)] \cap H^1[0,T;H^{-1}(\Omega)]$,
\begin{eqnarray} %\label{eqn:2.4}
&&   \int_0^T {\left( {  \left\langle {{\mu}_{g_1},v_t } \right\rangle  + \epsilon_1 \left( {\nabla{\mu}_{g_1},\nabla v } \right)
    - ((a-by_{g_2}){\mu}_{g_1},v )   }    \right) dt} \nonumber \\
&& \qquad \qquad\qquad\qquad\qquad\qquad\qquad\qquad
 = -({\mu}_{g_1}(0) ,v(0)) + \int_0^T \left( {y}_{g_1}-y_{1,d},v   \right)dt,\\
&&   \int_0^T {\left( {  \left\langle {{\mu}_{g_2},v_t } \right\rangle  + \epsilon_2\left( {\nabla{\mu}_{g_2},\nabla v } \right)
    -((c{y}_{g_1}-d){\mu}_{g_2},v )   }    \right) dt} \nonumber \\
&& \qquad \qquad \qquad\qquad\qquad\qquad\qquad\qquad
 = -({\mu}_{g_2}(0) ,v(0)) + \int_0^T \left( {y}_{g_2}-{y_{2,d}},v   \right)dt.
\end{eqnarray}
where $\mu_{g_i}(T)=0$ and $(\mu_{g_i})_t \in L^2[0,T;H^{-1}(\Omega)], i=1,2$.
\end{lem}

In the following Lemma, we state the optimality system which consists of the state equation% (given in the weak form (\ref{eqn:2.3}))
, the adjoint, and the optimality condition.
\begin{lem} \label{lem:2.5} Let $(y_{g_i},g_i) \equiv (y_i,g_i) \in W_D(0,T) \times {\mathcal A}_{ad}, i=1,2,$ denote the unique optimal pairs of Definition \ref{defn:2.2}. Then, there exists an adjoint $\mu_1,\mu_2 \in W_D(0,T)$ satisfying, $\mu_1(T) =\mu_2(T)=0$ such that for all $v  \in L^2[0,T;H^1(\Omega)] \cap H^1[0,T;H^{-1}(\Omega)]$,
 \begin{eqnarray}%\label{eqn:2.5}
&&  (y_1(T),v(T)) + \int_0^T {\left( { - \left\langle {y_1,v_t } \right\rangle  + \epsilon_1 \left( {\nabla y_1,\nabla v } \right)-\left( {(a-by_2)y_1,v } \right)
       }    \right)} dt \nonumber \\
&& \qquad\qquad\qquad\qquad\qquad    = ({y_{1,0}} ,v(0)) + \int_0^T {  {\left\langle {f_1,v } \right\rangle  }  } dt+ \int_0^T {\left( { g_1,v } \right)} dt,
\\
&&  (y_2(T),v(T)) + \int_0^T {\left( { - \left\langle {y_2,v_t } \right\rangle  + \epsilon_2 \left( {\nabla y_2,\nabla v } \right)
       }    \right)} dt - \int_0^T {\left(  (cy_1-d)y_2,v   \right)} dt\nonumber \\
&& \qquad\qquad\qquad\qquad\qquad    = ({y_{2,0}} ,v(0)) + \int_0^T {    { \left <f_2,v \right >}  } dt+ \int_0^T {\left( {  g_2,v  } \right)} dt,\\
%\end{eqnarray}
%\begin{eqnarray}
&&   \int_0^T {\left( {   \left\langle {{\mu}_{1},v_t } \right\rangle  + \epsilon_1 \left( {\nabla{\mu}_{1},\nabla v } \right)
    - ((a-by_{2}){\mu}_{1},v)    }    \right) dt} \nonumber\\
&& \qquad \qquad \qquad \qquad \qquad   = -({\mu}_{1}(0) ,v(0)) + \int_0^T \left(  {y}_{1}-{y_{1,d}},v    \right)dt, \label{eqn:2.6mu1}\\
&&   \int_0^T {\left( {  \left\langle {{\mu}_{2},v_t } \right\rangle  + \epsilon_2\left( {\nabla{\mu}_{2},\nabla v } \right)
    -((cy_1-d){\mu}_{2},v )   }    \right) dt} \nonumber \\
&& \qquad \qquad \qquad \qquad \qquad    = -({\mu}_{2}(0) ,v(0)) + \int_0^T \left( {y}_{2}-{y_{2,d}},v   \right)dt,\label{eqn:2.6mu2}
\end{eqnarray}
%with:
%\begin{equation} \label{eqn:2.8}
%\int_0^T \int_\Omega \Big (\left ( \gamma_1 g_1+   \mu_1,u \right )
%\Big)
%dx dt = 0, \quad \int_0^T \int_\Omega \left ( \gamma_2 g_2+   \mu_2 , u \right )
%dx dt = 0, \quad \forall u \in {\mathcal A}_{ad}.
%\end{equation}
%or
with control constraints:
\begin{eqnarray}\label{eqn:2.8}
\int_0^T \int_\Omega  \left ( \gamma_1 g_1+   \mu_1 \right) \left ( u_1-g_1 \right )
dx dt \geq 0,\,
\int_0^T \int_\Omega  \left ( \gamma_2 g_2+   \mu_2 \right) \left ( u_2-g_2 \right )
dx dt \geq 0, \\ \forall u_1,u_2 \in {\mathcal A}_{ad}.\nonumber
\end{eqnarray}

In addition, ${y_{i,t}}, {\mu_{i,t}} \in L^2[0,T;H^{-1}(\Omega)]$, and note that (\ref{eqn:2.8}), is equivalent to
$$g_1(t,x) = Proj_{[{g_{1a}},{g_{1b}}]} \left ( - \frac{1}{\gamma_1} \mu_1 (t,x) \right ), \, g_2(t,x) = Proj_{[{g_{2a}},{g_{2b}}]} \left ( - \frac{1}{\gamma_2} \mu_2 (t,x) \right ), $$ for a.e. $(t,x) \in (0,T]\times \Omega$, with ${\mu_{i,t }}\in L^2[0,T;H^2(\Omega)] \cap L^2[0,T;L^2(\Omega)]$,
 where $Proj_{[g_{ia}, g_{ib}]}(g_i) =$ \newline $\max\{g_{ia},\min\{g_{ib}, g_i\}\}$, $i = 1, 2$.
%\red
%\red{In addition, ${\mu_i}_t \in L^2[0,T;H^2(\Omega)] \cap L^2[0,T;L^2(\Omega)],i=1,2$.}
\end{lem}
%\begin{proof}
%\red{The derivation of the optimality system is standard (see e.g. \cite{Tr10}). For the enhanced regularity on $\mu$, we note that $y_i-{y_{i,d}} \in L^2[0,T;L^2(\Omega)]$ and apply the analogue of Theorem \ref{thm:2.1} for (\ref{eqn:3.5}), (\ref{eqn:3.5_}) to get that $\mu_i \in L^2[0,T;H^2(\Omega)] \cap H^1[0,T;L^2(\Omega)],i=1,2$.}
%\end{proof}
%red
%\end{lem}
\begin{proof}
The derivation of the optimality system is standard, see e.g. \cite{Tr10}.% For the enhanced regularity on $\mu$, we note that $y_i-{y_{i,d}} \in L^2[0,T;L^2(\Omega)]$ and apply the analogue of Theorem \ref{thm:2.1} for (\ref{eqn:2.6mu1}), (\ref{eqn:2.6mu2}) to get that $\mu_i \in L^2[0,T;H^2(\Omega)] \cap H^1[0,T;L^2(\Omega)],i=1,2$.
\end{proof}

\subsubsection{The fully-discrete {{distributed}} optimal  control problem}
The discontinuous time-stepping fully-discrete scheme for the control to state mapping $G_{\blue{h,D}}:L^2[0,T;{L}^2(\Omega)\times {L}^2(\Omega)] \to {\mathcal U}_{\blue{h,D}}\times {\mathcal U}_{\blue{h,D}}$,
which associates to each control $(g_1, g_2)$ the corresponding state $G_{\blue{h,D}}(g_1,g_2) = ({y_1}_{g_1,h},{y_2}_{g_2,h}) \equiv (y_{1h}({g_1}),y_{2h}({g_2}))$ is defined as follows: For any control data $g_1,g_2 \in L^2[0,T;L^2(\Omega)]$, for given initial data ${y_{1,0}}, {y_{2,0}} \in L^2(\Omega)$, forces $f_1,f_2 \in L^2[0,T;H^{-1}(\Omega)]$, and targets $y_{1,d}$, $y_{2,d} \in L^2[0,T;L^2(\Omega)]$ we seek $y_{1h}$, $y_{2h} \in \mathcal U_h$ such that for
$n=1,...,N$, and for all $v_h \in {\mathcal P}_k[t^{n-1},t^n;U_{\blue{h,D}}]$,
\begin{eqnarray}\label{eqn:3.2}
&&  (y_1^n,v^n) + \int_{t^{n-1}}^{t^n} {\left( { - ( {y_{1h},v_{ht} } )  + \epsilon_1 \left( {\nabla y_{1h},\nabla v_h } \right)-\left( (a-by_{2h})y_{1h},v_h   \right)
       }    \right)} dt \nonumber \\
&& \qquad   = ({y_1}^{n-1} ,v_+^{n-1}) + \int_{t^{n-1}}^{t^n} {  {\left\langle {f_1,v_h } \right\rangle  }  } dt+ \int_{t^{n-1}}^{t^n} {\left ( {  g_{1},v_h  } \right )} dt,
\\%
&&  (y_2^n,v^n) + \int_{t^{n-1}}^{t^n} {\left( { - ( {y_{2h},v_{ht} } )  + \epsilon_2  \left( {\nabla y_{2h},\nabla v_h } \right)
       }    \right)} dt - \int_{t^{n-1}}^{t^n} {\left( {  {(cy_{1h}-d)y_{2h},v_h }    } \right)} dt\nonumber \\
&& \qquad   = ({y_2}^{n-1} ,v_+^{n-1}) + \int_{t^{n-1}}^{t^n} {  { \left <f_2,v_h \right >}  } dt+ \int_{t^{n-1}}^{t^n} {\left ( {  g_{2},v_h  } \right )} dt.  \label{eqn0:3.3}
\end{eqnarray}
\begin{rem}We note that in the above definition only $g_1,g_2 \in L^2[0,T;L^2(\Omega)]$ regularity is needed to validate the weak fully-discrete formulation. As a matter of fact, the %\TODO
{stability estimate at arbitrary time-points} as well as in $L^2[0,T;H_0^1(\Omega)]$ %and $L^2[0,T;L^2(\Gamma)]$
 norm easily follows by setting $v_h = y_{1h}$, $v_h = y_{2h}$ into (\ref{eqn:3.2})--(\ref{eqn0:3.3}) while for the arbitrary time-points stability estimate, we may apply the techniques %which were developed 
as in \cite{ChKa14,ChKa12}. %[Section 2]{ChWa06}
%for general linear parabolic PDEs, (see also \cite{ChKa12}
%for stability estimate for semilinear parabolic PDEs with Dirichlet data).
\end{rem}
Analogously to the continuous case, 
\magent{we will occasionally abbreviate the notation $y_{ih} = y_{ih} (g_{ih} )$ and} we note that the control to fully-discrete state mapping $G_{{h,D}} :L^2[0,T;{{L}}^2(\Omega)\times {{L}}^2(\Omega)] \to {\mathcal U}_{{h,D}}\times {\mathcal U}_{{h,D}}$, is well defined, and continuous and let \magent{$J_h(y_{1h},y_{2h},g_{1h},g_{2h}) = \frac{1}{2}\int_0^T \int_{\Omega} (| y_{1h}-{y_{1,d}}|^2 +| y_{2h}-{y_{2,d}}|^2) dxdt  +
\frac{\gamma_1 }{2}\int_0^T \int_{\Omega} (| g_{1h}|^2) dxdt+
\frac{\gamma_2 }{2}\int_0^T \int_{\Omega} (| g_{2h}|^2) dxdt$}. 

The control problem definition now takes the form:
\begin{defn}\label{defn:3.1}
Let $\magent{f_1,f_2}\in L^2 [0,T;H^{-1} (\Omega )]$, $\magent{y_{1,0}, y_{2,0}} \in L^2(\Omega)$, $\magent{y_{1,d}, y_{2,d}} \in L^2[0,T;L^2(\Omega)]$, be given data.
 Suppose that the set of discrete admissible controls is denoted by ${\mathcal A}^d_{ad} \equiv {\mathcal U}_{{h,D}} \cap {\mathcal A}_{ad}$\magent{. T}he pairs $(y_{ih},g_{ih}) \in {\mathcal U}_{{h,D}} \times  L^2[0,T; U_{{h,D}}%^n
], i=1,2,$ satisfy (\ref{eqn:3.2})--(\ref{eqn0:3.3}) and the definition of the corresponding control problem takes the form: The pairs $(y_{ih},g_{ih})\in A^d_{ad},i=1,2$,  are said to be optimal solutions if $J_h(y_{1h},y_{2h},g_{1h},g_{2h})\le J_h(w_{1h},w_{2h},u_{1h},u_{2h})$ \magent{for all $(u _{1h} , u_{2h} ) \in \mathcal{A} ^d_{ad}\times \mathcal{A} ^d_{ad}, (w_{1h} , w_{2h} ) \in \mathcal{U} _{h,D} \times \mathcal{U} _{h,D}$ being the solution of problem (\ref{eqn:3.2})-(\ref{eqn0:3.3}) with control functions $(u_{h1} , u_{h2} )$}.
\end{defn}

The discrete optimal control problem solution existence can be proved by standard techniques while uniqueness follows from the structure of the functional, and the linearity of the equation. The basic stability estimates in terms of the optimal pair $(y_{ih},g_{ih}) \in L^2[0,T;H_0^1(\Omega)] \times L^2[0,T;L^2(\Omega)],i=1,2,$ can be easily obtained, see e.g.  \cite{ChKa12}. 

We note that the key difficulty of the discontinuous time-stepping formulation is the lack of any meaningful regularity for the time-derivative of $y_{ih}$ due to the presence of discontinuities. However,  it is also worth noting that dG in time is applicable even for higher order schemes, \cite{ChKa12}.
\subsubsection{The discrete {{distributed}} optimality system}
In the following Lemmas \ref{lem:3.3} and  \ref{lem:3.4}  we state the discrete optimality systems in two appropriate forms.
\begin{lem} \label{lem:3.3}
The cost functional $J_h:L^2[0,T;{L^2(\Omega)}] \to \mathbb R$ is well defined differentiable and for every $g_1,g_2,u_1,u_2 \in L^2[0,T;L^2(\Omega)]$,
\[J_h^{'}(g_1,g_2)(u_1,u_2) = \int_0^T \int_\Omega \Big (( {\mu_{1h}} ({g_1}) + \alpha g_1) u_1 , ( {\mu_{2h}} ({g_2}) + \alpha g_2) u_2\Big ) dxdt, \]
where $\mu _{ih}({g_i}) \equiv {\mu}_{{g_i},h} \in W_D(0,T), i=1,2,$ is the unique solution of following problem: For all $\upsilon _h \in P_k[t^{n-1},t^n;U_{\blue{h,D}}], n=1,...N$,
\begin{eqnarray} \label{eqn:3.3}
&&  -(\mu_{{g_1},h+}^n,\upsilon ^n)+ \int_{t^{n-1}}^{t^n} {\left( {  ( {{\mu}_{{g_1},h},\upsilon_{ht} } )  + \epsilon_1 \left( \nabla \upsilon_h,{\nabla{\mu}_{g_1,h} } \right)
    -  ((a-by_{g_2,h}){\mu}_{g_1,h},\upsilon_h     })    \right) dt} \nonumber \\
&& \qquad   = -({\mu}_{g_1,h+}^{n-1} ,\upsilon_+^{n-1}) + \int_{t^{n-1}}^{t^n} \left( {y}_{g_1,h}-{y_{1,d}},\upsilon_h   \right)dt,\\
&& -(\mu_{{g_2},h+}^n,\upsilon ^n)+  \int_{t^{n-1}}^{t^n} {\left( { ( {{\mu}_{g_2,h},\upsilon_{ht} })  + \epsilon_2\left( {\nabla{\mu}_{g_2,h},\nabla\upsilon_h } \right)
   -((c {y}_{g_1,h}-d){\mu}_{g_2,h},\upsilon_h )   }    \right) dt} \nonumber \\
&& \qquad   = -({\mu}_{g_2,h+}^{n-1} ,\upsilon_+^{n-1}) + \int_{t^{n-1}}^{t^n} \left( {y}_{g_2,h}-{y_{2,d}},\upsilon_h   \right)dt,
\end{eqnarray}
where $\mu_{g_i+}^N=0$ and $g_i\in \mathcal A_{ad}^d$.
\end{lem}
%Below, we state the optimality system which consists of the state equation (given in the weak form (\ref{eqn:2.3})), the adjoint and the optimality condition.
\begin{lem} \label{lem:3.4} Let $(y_{ih}({g_{ih})},g_{ih}) \equiv (y_{ih},g_{ih}) \in {\mathcal U}_{\blue{h,D}} \times {L^2[0,T;U_{\blue{h,D}}]}, i=1,2,$ denote the unique optimal pairs of \magent{Definition \ref{defn:3.1}}. Then, there exists an adjoint $\mu_{1h},\mu_{2h} \in {\mathcal U}_{\blue{h,D}}$ satisfying, $\mu_{1h+}^N =\mu_{2h+}^N=0$ such that for all $\upsilon_h  \in P_k[t^{n-1},\magent{t^n};U_{\blue{h,D}}]$, and for all $n=1,...N$
 \begin{eqnarray}\label{eqn:3.4}
 &&  (y_1^n,\upsilon^n) + \int_{t^{n-1}}^{t^n} {\left( { - ( {y_{1h},\upsilon_{ht} } )  +\epsilon_1 \left( {\nabla y_{1h},\nabla\upsilon_h } \right)-\left( ({a-by_{2h})y_{1h},\upsilon_h } \right)
       }    \right)} dt \nonumber \\
&& \qquad   = ({y_1}^{n-1} ,\upsilon_+^{n-1}) + \int_{t^{n-1}}^{t^n} { {\left\langle {f_1,\upsilon_h } \right\rangle  }  } dt+ \int_{t^{n-1}}^{t^n} {\left ( { g_{1},\upsilon_h } \right )} dt,
\\
&&  (y_2^n,\upsilon^n) + \int_{t^{n-1}}^{t^n} {\left( { - \left\langle {y_{2h},\upsilon_{ht} } \right\rangle  + \epsilon_2 \left( {\nabla y_{2h},\nabla\upsilon_h } \right)
       }  -\left( { {(cy_{2h}-d)y_{2h},\upsilon_h }  } \right)  \right)} dt \nonumber \\
&& \qquad   = ({y_2}^{n-1} ,\upsilon_+^{n-1}) + \int_{t^{n-1}}^{t^n} { {  \left <f_2,\upsilon_h \right >}  } dt+ \int_{t^{n-1}}^{t^n} { { \left ( g_{2},\upsilon_h  \right )} } dt,%\nonumber \\
\end{eqnarray}
\begin{eqnarray} \label{eqn:3.5}
&&(\mu_{1+}^n,\upsilon^n) +   \int_{t^{n-1}}^{t^n} {\left( { ( {{\mu}_{1h},\upsilon_{ht} })  + \epsilon_1 \left( {\nabla{\mu}_{1h},\nabla\upsilon_h } \right)
    - ({(a-by_{2h})\mu}_{1h},v    }    \right) dt} \nonumber \\
&& \qquad \qquad \qquad \qquad \qquad   = -({\mu}_{1+}^{n-1} ,\upsilon_+^{n-1}) + \int_{t^{n-1}}^{t^n} \left( {y}_{1h}-{y_{1,d}},v   \right)dt,\\
&&  (\mu_{2+}^n,\upsilon^n) +  \int_{t^{n-1}}^{t^n} {\left( {  ( {{\mu}_{2h},\upsilon_{ht} } )  + \epsilon_2\left( {\nabla{\mu}_{2h},\nabla\upsilon_h } \right)
    -((c{y}_{1h}-d){\mu}_{2h},\upsilon_h     }    \right) dt} \nonumber \\
&& \qquad \qquad \qquad \qquad \qquad    = -({\mu}_{2+}^{n-1} ,\upsilon_+^{n-1}) + \int_{t^{n-1}}^{t^n} \left( {y}_{2h}-{y_{2,d}},v   \right)dt,\label{eqn:3.5_1}
\end{eqnarray}
%with:
%\begin{eqnarray} \label{eqn:3.6}
%\int_0^T \int_\Omega \left ( \gamma_1 g_{1h}+  \mu_{1h},u_{h} \right)
%dx dt = 0, \quad \int_0^T \int_\Omega \left ( \gamma_2 g_{2h}+  \mu_{2h},u_{h} \right)
%dx dt = 0, \quad \forall u_{h} \in {\mathcal A}_{ad}^d. \nonumber
%\end{eqnarray}
%or
with control constraints:
\begin{eqnarray} \label{eqn:3.6}
\int_0^T \int_\Omega  \left ( \gamma_1 g_{1h}+  \mu_{1h} \right) \left ( u_{1h}-g_{1h} \right ) dx dt \ge 0,\,\,
\int_0^T \int_\Omega  \left ( \gamma_2 g_{2h}+  \mu_{2h} \right) \left ( u_{2h}-g_{2h} \right ) dx dt \ge 0,
\nonumber \\ \forall u_{1h},u_{2h} \in {\mathcal A}_{ad}^d. 
\end{eqnarray}
In addition,  (\ref{eqn:3.6}) is equivalent to
$$g_{1h}(t,x) = Proj_{[g_{1a},g_{1b}]} \left ( - \frac{1}{\gamma_1} \mu_{1h} (t,x) \right ),\, g_{2h}(t,x) = Proj_{[g_{2a},g_{2b}]} \left ( - \frac{1}{\gamma_2} \mu_{2h} (t,x) \right ),$$ for a.e. $(t,x) \in (0,T]\times \Omega$, %In addition, ${\mu_{i,t}} \in L^2[0,T;H^2(\Omega)] \cap L^2[0,T;L^2(\Omega)],i=1,2$.
%\end{lem}
%
along with ${\mu_i}_t \in L^2[0,T;H^2(\Omega)] \cap L^2[0,T;L^2(\Omega)],i=1,2$.
\end{lem}
\blue{
\begin{proof}
The derivation of the optimality system is standard, \cite{Tr10}. For an enhanced regularity on $\mu_i$, we note that $y_i-{y_{i,d}} \in L^2[0,T;L^2(\Omega)]$ one may apply 
  basic existence, uniqueness, and regularity estimate results
% the analogue of {\red{Theorem \ref{thm:2.1}}} 
 on (\ref{eqn:3.5})--(\ref{eqn:3.5_1}) to get $\mu_i \in L^2[0,T;H^2(\Omega)] \cap H^1[0,T;L^2(\Omega)],i=1,2$.
%{\red{The next theorem states the basic existence, uniqueness, and regularity result (see e.g. 
%\cite{ChGuHo06}).%!!!!
%\begin{thm}\label{thm:2.1}
%Suppose $g\in L^2[0,T;H^{-1/2+\theta}(\Gamma)] \cap
%H^\theta[0,T;H^{-1/2}(\Gamma)]$,
%$y_0\in H^\theta(\Omega)$,
%and $f\in L^2[0,T;H^{1-\theta}(\Omega)^*]$
%for some $\theta\in[0,1]$.
%Then, there exists a unique
%$y \in L^2[0,T;H^{1+\theta}(\Omega)]
%\cap H^1[0,T;H^{1-\theta}(\Omega)^*]$
%satisfying \eqref{eqn:2.1} and
%\begin{eqnarray*}  \label{eqn:2.3}
%  && \|y\|_{L^2[0,T;H^{1+\theta}(\Omega)]}
%  +\|y_t\|_{L^2[0,T;H^{1-\theta}(\Omega)^*]}
%\\
%  && \le C\Big( \|f\|_{L^2[0,T;H^{1-\theta}(\Omega)^*]}
%             + \|u_0\|_{\theta}
%             + \|g\|_{L^2[0,T;H^{-1/2+\theta}(\Gamma)]}
%             + \|g\|_{H^\theta[0,T;H^{-1/2}(\Gamma)]}
%       \Big). \nonumber
%\end{eqnarray*}
%\end{thm}}}
\end{proof}
\subsection{Second order optimality conditions}
 A significant aspect of the optimization procedure analysis is the second-order sufficient conditions. %A very important matter for %To perform the numerical analysis of the problem as well as the optimization procedure analysis is  the second-order sufficient conditions. 
The second order conditions have to be written for directions $(v_1,v_2)\in T_{{\mathcal A}_{ad}
} (\bar g_1, \bar g_2)$ such that
$J'(g_1,g_2)(v_1,v_2) = \bf{0}$, where $T_{{\mathcal A}_{ad}}(\bar g_1, \bar g_2)$  is the tangent cone at $(\bar g_1, \bar g_2)$ to ${\mathcal A}_{ad}$. 
We recall that the cone of feasible directions and tangent cones,  ${\mathcal R}_{{\mathcal A}_{ad}}$ and  $T_{{\mathcal A}_{ad}}$, are ${\mathcal R}_{{\mathcal A}_{ad}}(\bar g_i)=\{ y_i\in L^2[0,T;L^2(\Omega)];\exists \sigma_i>0;\bar g_i+\sigma y_i\in {{\mathcal A}_{ad}}\}$, and $T_{{\mathcal A}_{ad}}(\bar g_i)=\{ y_i\in L^2[0,T;L^2(\Omega)];\exists \bar g_i(\sigma_i)=\bar g_i+\sigma_i y_i+{{o}}(\sigma_i)\in {{\mathcal A}_{ad}}, \sigma_i \ge 0 \}$ respectively. 
%To characterize these directions, we introduce

Following \cite{CaCh12} and assuming $\gamma _1, \gamma _2\geq 0$, $y_{1,d},y_{2,d}\in L^2[0,T;L^2(\Omega)]$, the mapping $G:L^2[0,T;L^2(\Omega)\times L^2(\Omega)]\to H^{2,1}([0,T]\times\Omega)\cap C([0,T;H^1_0])$, associated to each control $g_1,g_2$ the corresponding state $G(g_1)=y_{g_1}$, and $G(g_2)=y_{g_2}$ solution of (\ref{eqn:2.2})--(\ref{eqn:2.1}), is well defined and continuous, then the cost functional $J:L^2[0,T;L^2(\Omega)]\to \mathbb{R}$ is also well defined and continuous, the following theorems state for the differentiability of the mappings $G$ and $J$.
\begin{thm}
The mapping $G$ is of class $C^\infty$. Moreover, then  $z_{v_i}$ and $z_{v_iv_i}$ are the unique solutions of the following equations
\begin{eqnarray}
  \left\langle {{z_{{v_1},t}} ,w} \right\rangle  +  \epsilon_1(\nabla z_{v_1},\nabla w )  - \left <  (a-b y_{g_2} )z_{v_1},w \right >&=&      \left\langle {{v_1},w } \right\rangle \qquad\qquad\mbox{ and}\quad\quad
  z_{v_1}(0) = 0,\nonumber\\
  \left\langle {{z_{{v_1}{v_1},t}} ,w} \right\rangle  +  \epsilon_1(\nabla z_{{v_1}{v_1}},\nabla w )  - \left <(a-b y_{g_2} )z_{{v_1}{v_1}},w \right >&=&      0   \,\,\,\qquad\qquad\qquad\mbox{ and}\quad
  z_{{v_1}{v_1}}(0) = 0, \nonumber\\
 % \text{ where }z_{1v}=G'(g)v_1 \text{ and }z_{2v}=G'(g)v_1
   \left\langle {{z_{v_2,t}} ,w} \right\rangle  + \epsilon_2 (\nabla z_{v_2},\nabla w ) - \left <(cy_{g_1}-d)z_{v_2},w \right >&=&      \left\langle {{v_2},w } \right\rangle \qquad\qquad\mbox{ and}\qquad
  z_{v_2}(0) = 0,\nonumber\\
  \left\langle {{z_{{v_2}{v_2},t}} ,w} \right\rangle  +\epsilon_2  (\nabla z_{{v_2}{v_2}},\nabla w )  - \left <(cy_{g_1}-d)z_{{v_2}{v_2}},w \right > &=&      0,\qquad\qquad\qquad\mbox{ and}\quad
  z_{{v_2}{v_2}}(0) = 0, \nonumber% % \text{ where }z_{1v}=G'(g)v_1 \text{ and }z_{2v}=G'(g)v_1
\end{eqnarray}
\end{thm}
for any $g_1,g_2,v_1,v_2 \in L^2[0,T;L^2(\Omega)]$, denoting $y_{g_i}=G(g_i)$,  $z_{v_i}=G'(g_i)v_i$ and $z_{v_iv_i}=G''(g_i)v_i^2$, $\forall w\in W_D(0,T)%H^1_0
$.
%\begin{proof}
%The proof is similar to [Theorem 3.1, \cite{CaCh12} (extended version)] combining the estimates as demonstrated in \cite{NeVe12}.
%\end{proof}
\begin{thm}
The object functional $J:L^2[0,T;{L^2(\Omega)}]\to \mathbb{R}$ is of class $C^\infty$ and for every $g_1,g_2,v_1,v_2\in L^2[0,T;L^2(\Omega)]$ we have
\begin{eqnarray}
J'(g_1,g_2)({v_1},{v_2})&\magent{=}&\int_0^T\int_{\Omega} \Big( (\mu_1(g_1)+\gamma_1 g_1)v_1,(\mu_2(g_2)+\gamma_2 g_2)v_2\Big)dxdt,
\\
J''(g_1,g_2)({v_1}^2,{v_2}^2)&\magent{=}&\int_0^T\int_{\Omega} \Big(z_{v_1}^2+\gamma_1 {v^2_1},z_{v_2}^2+\gamma_2 {v^2_2}\Big)dxdt.
\end{eqnarray}
\end{thm}
\begin{proof}
The proof is similar to [Theorem 3.1, \cite{CaCh12} (extended version)] combining the estimates as demonstrated in \cite{NeVe12}.
\end{proof}
%\textcolor[rgb]{1.00,0.00,0.00}{\begin{lem}
%Let $\bar g_i \in Q_{ad}$ be a local solution of (2.5)$(min J,g_i\in Q_{ad})$ in the sense of Definition 2.1.
%Then the following variational inequality holds:
%\begin{equation}
%J'( \bar g_1,\bar g_2)(p_1- \bar g_1,p_2- \bar g_2) \ge (0,0) \quad\forall p_1,p_2 \in Q_{ad}. \label{(2.11)}
%\end{equation}
%\end{lem}
%\begin{ass} Let $g_i \in Q_{ad}$ fulfils the first-order necessary optimality conditions
%\ref{(2.11)}. We assume that there exist a constant $\gamma > 0$ such that
%\begin{equation}
%J''( \bar g_1,\bar g_2)(v_1^2,v_2^2) \ge \gamma \|(v_1,v_2)\|^2 \qquad \forall v\in Q
%\end{equation}
%\end{ass}
%\begin{thm}
%Let us assume that $\bar g_1, \bar g_2 \in {\mathcal A}_{ad}$ satisfies
%\begin{equation}
%J'(\bar g_1, \bar g_2)(g_1-\bar g_1,g_2-\bar g_2) \ge (0,0) \forall g_1, g_2 \in {\mathcal A}_{ad}
%\end{equation}
%\begin{equation}
%J''(\bar g_1, \bar g_2)(v_1^2,v_2^2) > (0,0), \quad \forall v\in
%\end{equation}
%then there exist $\epsilon > 0$ and $\delta > 0$ such that
%\begin{equation}
% J(\bar g_1, \bar g_2) + \frac{\delta}{2}\|(g_1-\bar g_1,g_2-\bar g_2)\|_{L^2[0;T;L^2(\Omega)\times L^2(\Omega)]} \le J( g_1,  g_2) \forall  g_1, g_2\in {\mathcal A}_{ad} \cap B_\epsilon(\bar g_1, \bar g_2)
%\end{equation}
%where $B_\epsilon(\bar g_1, \bar g_2)$ is the ${L^2[0;T;L^2(\Omega)\times L^2(\Omega)]}$ ball of center $\bar g_1, \bar g_2$ and radius $\epsilon$.
%\end{thm}
%}
\subsubsection{Cones of Critical Directions} In the theory of second-order necessary conditions for certain classes of optimal control problems involving equality and/or inequality constraints in the control, under certain  assumptions, a certain quadratic form is non-negative on a cone of critical directions (or differentially admissible variations),  \cite{RoLi13}, \cite{CaCh12}, % \cite{CaCh13},
\cite{CaTrReyes2008}. For this reason, we use a cone of critical directions as below:
\begin{eqnarray}
% \nonumber to remove numbering (before each equation)
  {\mathcal C}_{(\bar g_1,\bar g_2)}=
%  \\&&
   \left\{  {v_i}\in L^2[0,T;L^2(\Omega)]:
J'(\bar g_1,\bar g_2)(v_1,v_2)={\bf{0}};(v_1,v_2)\in T_{{\mathcal A}_{ad}}(\bar g_1,\bar g_2)
    \right\}\nonumber %\\%\label{coneDef}
\end{eqnarray}
\begin{eqnarray}
 % {\mathcal C}_{(\bar g_1,\bar g_2)}
=
%  \\&&
   \left\{  {v_i}\in L^2[0,T;L^2(\Omega)]:
\left |
   {\begin{array}{*{20}{c}}
       v_i\ge 0, \\
       v_i\le 0,\\
       v_i=0,
    \end{array}}
  \right.
  {\begin{array}{*{20}{c}}
        \text{ if } &-\infty <g_{ia}={\bar g}_i&\\
       \text{ if } &{\bar g}_i=g_{ia}<\infty& \\
        \text{ if } &{\bar \mu}_i +\gamma_i {\bar g}_i \ne 0&
    \end{array}} %\}
    \right\},i=1,2, \text{  }%\nonumber \\
\label{coneDef}
\end{eqnarray}
%where $i=1,2$ and
and since $
%\begin{eqnarray}
%&&
J'({\bar g}_1,{\bar g}_2)({v_1},{v_2})=\int_0^T\int_{\Omega} ( (\bar \mu _1+\gamma_1 \bar g_1)v_1,(\mu _2+\gamma_2 g_2)v_2)dxdt,%\nonumber\\
%&&
$ it is
\begin{eqnarray}
&(\bar \mu _1+\gamma_1\bar g_1)v_1=0,\quad (\bar \mu _2+\gamma_2 \bar g_2)v_2=0,&  \label{cone1} 
\end{eqnarray}
for a.e  $ (t,x)\in \Omega\times [0,T]$ and  $\forall (v_1,v_2)\in {\mathcal C}_{(\bar g_1,\bar g_2)}$.
%{\mathcal C}=
Also if $(\bar {g}_1,\bar {g}_2)$ is a local solution of the continuous problem, and using that
\begin{equation} %\label{eqn:3.8}
\int_0^T \int_\Omega (\left ( \gamma_1 \bar g_1+   \bar \mu_1 \right) \left ( u_1-\bar g_1 \right ))dx dt \geq 0, \,
\int_0^T \int_\Omega (\left ( \gamma_2 \bar g_2+   \bar \mu_2 \right) \left ( u_2-\bar g_2 \right ))dx dt \geq 0
\quad \forall u_1,u_2 \in {\mathcal A}_{ad},
\end{equation}
we have
\begin{eqnarray}
%\left\{
{\begin{array}{*{20}{c}}
 \text{ if }\bar g_i(t,x)=g_{ia} \text{ then }  \alpha \bar g_i+   \bar \mu_i \ge 0,\\
       \text{ if }\bar g_i(t,x)=g_{ib} \text{ then } \alpha \bar g_i+   \bar \mu_i\le 0,\\
       \text{ if }\bar g_i(t,x)\in [g_{ia},g_{ib}] \text{ then } \alpha \bar g_i+   \bar \mu_i=0,
\end{array}}
% \right.
{\rm{ and}}
%\left\{
{\begin{array}{*{20}{c}}
 \text{ if }\alpha \bar g_i+   \bar \mu_i > 0 \text{ then }  \bar g_i(t,x)=g_{ia}, \\
 \text{ if }\alpha \bar g_i+   \bar \mu_i< 0  \text{ then }   \bar g_i(t,x)=g_{ib}.
\end{array}}
%\right.
\nonumber\\\label{d_i(x,t)}
\end{eqnarray}
%The following statements, ideas and proofs are based on the very new work \cite{CaCh13}.
%-----------------------------------------------------------------------------------------------
\begin{prop}[{Second order optimality conditions}\label{Second_Order_proof}]

If $(\bar {g}_1,\bar {g}_2)$ is a local solution of the continuous problem then $J''(\bar {g}_1,\bar {g}_2)(v^2_1,v^2_2)>0$, $\forall v_i\in {\mathcal C}_{(\bar g_1,\bar g_2)}$.
\end{prop}
%\begin{proof}
%See Appendix \ref{Second_Order_proof}.
%\end{proof}
\begin{proof}
If $(v_1,v_2)\in {\mathcal C}_{(\bar g_1,\bar g_2)}$ and for $\varepsilon < \min{\frac{g_{ib}-g_{ia}}{2}}$, $i=1,2$ we define
\[v_{i,\varepsilon=}\left\{ {\begin{array}{*{20}{c}}
{0}\\
0\\
{{\rm{Proj}}{_{[ - 1/\varepsilon ,1/\varepsilon ]}} v_i }
\end{array}} \right.\begin{array}{*{20}{c}}
{{\text{ if }} g_{ia}< \bar g_i<g_{ia}+\varepsilon},\\
{{\text{ if }} g_{ib}-\varepsilon < \bar g_i < g_{ib}},\\
{\text{ otherwise.}}
\end{array}\]
Obviously
$(v_{1\varepsilon},v_{2\varepsilon}) \in \mathcal C _{(g_1,g_2)} \forall \varepsilon > 0$, $|v_{i,\varepsilon} |\le |v_i|$ and $(v_{1\varepsilon},v_{2\varepsilon})\to (v_{1},v_{2})$ strongly in $L^2[0,T;L^2(\Omega)\times L^2(\Omega)]$ when $\varepsilon \to 0$.

Now we check that  $\bar g _i+\rho v_{i\varepsilon}\in {\mathcal A}_{ad}, \forall $ $\rho \in [0,\varepsilon ^2]$.
 If $\gamma _i \bar g_i +\bar \mu_i \ne 0$, then $v_{i,\varepsilon}=0$. So  $\bar g _i+\rho v_{i\varepsilon}=\bar g_i\in[g_{ia},g_{ib}]$. For $\gamma _i \bar g_i +\bar \mu_i = 0$:
(a) If $\bar g_i=g_{ia}$ then $v_i\ge 0$ and $v_{i,\varepsilon}\ge 0$. So $g_{ia}\le g_i +\rho v_{i,\varepsilon}$. Also  $g_i +\rho v_{i,\varepsilon}\le g_{ia}+\varepsilon ^2 \frac{1}{\varepsilon}\le (g_{ia}+g_{ib})/2<g_{ib}$. Similarly for $\bar g_i=g_{ib}$.
(b) If $g_{ia}<\bar g_i<g_{ia}+\varepsilon$, then    $\bar g _i+\rho v_{i\varepsilon}=\bar g_i\in[g_{ia},g_{ib}]$
 and the same applies if  $g_{ib}-\varepsilon<\bar g _i<g_{ib}$.
(c) If $g_{ia}+\varepsilon\le \bar g_i\le g_{ib}-\varepsilon$, then  $\bar g_i +\rho v_{i,\varepsilon}\ge g_{ia}+\varepsilon-\varepsilon ^2 \frac{1}{\varepsilon}= g_{ia}$, and  $\bar g_i +\rho v_{i,\varepsilon}\le g_{ib}-\varepsilon+\varepsilon ^2 \frac{1}{\varepsilon}= g_{ib}$. %, and we have the desired.

Since we have assumed  that $(\bar g_1,\bar g_2)$ is local minimum, it is
$$
J(\bar g_1,\bar g_2)\le J(\bar g_1+\rho v_{1\varepsilon},\bar g_2+\rho v_{2\varepsilon}),
$$
and a second order Taylor-Maclaurin expansion of $J$ at $(\bar g_1,\bar g_2)$,   for $\rho<\varepsilon^2$ small enough  there exists $\theta _\rho\in(0,\rho)$ such that
\begin{eqnarray}
% \nonumber to remove numbering (before each equation)
 0&\le& J(\bar g_1+\rho v_{1\varepsilon},\bar g_2+\rho v_{2\varepsilon})-J(\bar g_1,\bar g_2) \nonumber\\
 &&=\rho J'(\bar g_1,\bar g_2)(v_{1\varepsilon},v_{2\varepsilon})+\frac{\rho}{2}J''(\bar g_1+\theta_\rho v_{1\varepsilon},\bar g_2+\theta_\rho v_{2\varepsilon})(v_{1\varepsilon}^2,v_{2\varepsilon}^2).\label{cone2}
\end{eqnarray}
Since $(v_{1\varepsilon},v_{2\varepsilon})\in {\mathcal C}_{(\bar g_1,\bar g_2)}$, (\ref{cone1}) gives  $J'(\bar g_1,\bar g_2)(v_{1\varepsilon},v_{2\varepsilon})=0$ and (\ref{cone2}) shows $J''(\bar g_1+\theta_\rho v_{1\varepsilon},\bar g_2+\theta_\rho v_{2\varepsilon})(v_{1\varepsilon}^2,v_{2\varepsilon}^2)\ge 0$. Taking the limit $\rho \to 0 $ we have $J''(\bar g_1,\bar g_2)(v_{1\varepsilon}^2,v_{2\varepsilon}^2)\ge 0$. Taking $\varepsilon \to 0$
\magent{
\begin{eqnarray*}
% \nonumber to remove numbering (before each equation)
 \mathop {\lim }\limits_{\varepsilon \to 0} J''(\bar g_1,\bar g_2)({v_{1\varepsilon}}^2,{v_{2\varepsilon}}^2)&=&\int_0^T\int_{\Omega} \Big(z_{v_{1\varepsilon}}^2+\gamma_1 {v^2_{1\varepsilon}},z_{v_{2\varepsilon}}^2+\gamma_2 {v^2_{2\varepsilon}}\Big)dxdt \\
  &=&\int_0^T\int_{\Omega} \Big(z_{v_1}^2+\gamma_1 {v^2_1},z_{v_2}^2+\gamma_2 {v^2_2}\Big)dxdt\\
  &=&J''(g_1,g_2)({v_1}^2,{v_2}^2),
\end{eqnarray*}
}
and we have provided the desired estimate. %objective.
\end{proof}
%--------------------------------------------------------------------------------
\subsubsection{Sufficient conditions for optimality}
We quote and prove the sufficient conditions, thus, we try to be as unrestrictive as possible comparing them with the necessary second-order conditions and trying the gap between them to be small. %To perform the numerical analysis of the problem as well as the analysis of the algorithms of optimization, second order sufficient conditions are required. These sufficient conditions should be as unrestrictive as possible. One way of measuring this is to compare them with the necessary second order conditions and check if the gap is small. This is the reason why we first introduce the second order necessary conditions.
\begin{thm}[{Sufficient conditions for optimality}]\label{Sufficient_Con}
Let us assume that $\bar g_1, \bar g_2 \in {\mathcal A}_{ad}$ satisfies
\begin{eqnarray}\label{1st cond}
J'(\bar g_1, \bar g_2)(g_1-\bar g_1,g_2-\bar g_2) &\magent{\ge}& \magent{(0,0) \quad\forall }g_1, g_2 \in {\mathcal A}_{ad},
%\end{equation}
%\begin{equation}
\\
\label{2nd cond}
J''(\bar g_1, \bar g_2)(v_1^2,v_2^2) &\magent{>}& \magent{(0,0) \quad\forall }(v_1,v_2)\in {\mathcal C}_{(\bar g_1,\bar g_2)}- \{0\},
\end{eqnarray}
then there exist $\epsilon > 0$ and $\delta > 0$ such that
\begin{eqnarray}
 J(\bar g_1, \bar g_2) + \frac{\delta}{2}\|(g_1-\bar g_1,g_2-\bar g_2)\|_{L^2[0;T;L^2(\Omega)\times L^2(\Omega)]} \le J( g_1,  g_2),
 \magent{\quad\forall  g_i\in {\mathcal A}_{ad} \cap B_\epsilon(\bar g_i),}%\nonumber
\end{eqnarray}
where $B_\epsilon(\bar g_i)$ is the ${L^2[0;T;L^2(\Omega)]}$ ball of center $\bar g_i$ and radius $\varepsilon$.
\end{thm}
%--------------------------------------------------------------------------------
%\begin{proof}
%See Appendix \ref{Sufficient_Con}.
%\end{proof}
\begin{proof}
Following the methodology of \cite{CaCh12}, we show a contradiction  supposing that the theorem is false, then
\begin{eqnarray}\label{contradiction}
% \nonumber to remove numbering (before each equation)
  J(\bar g_1, \bar g_2)+\frac{1}{2n}\|(\bar g_1-g_{1n}, \bar g_2-g_{2n})\|^2_{L^2[0,T;L^2(\Omega)\times L^2(\Omega)]}> J(g_{1n}, g_{2n})
\end{eqnarray}
for sequences $g_{1n}$, $ g_{2n}$ such that
\begin{eqnarray}\label{1/n}
\|\bar g_1-g_{1n}\|\le \frac{1}{n},\text{ } \|\bar g_2-g_{2n}\|\le \frac{1}{n}.
\end{eqnarray}
We also assume that
\begin{eqnarray}\label{rho}
\rho _{n}&=&\|(\bar g_1-g_{1n},\bar g_2-g_{2n})\|_{L^2[0,T;L^2(\Omega)\times L^2(\Omega)]},   
\\
\label{vk}\magent{(v_{1n},v_{2n})}&=&\magent{\frac{1}{\rho _{n}}(\bar g_1-g_{1n},\bar g_2-g_{2n}),}
\end{eqnarray}
and that $v_{in}\rightharpoonup v_{i}$ (weakly in $L^2[0,T;L^2(\Omega)]$) as $n\to \infty$.

 Initially, we prove that $(v_{1},v_{2})\in {\mathcal C}_{(\bar g_1,\bar g_2)}$. We note that the first two relations in (\ref{coneDef})  show closeness and convexity for every element in ${L^2[0,T;{L^2(\Omega)\times L^2(\Omega)}]}$. We can see from (\ref{vk}) that the above two relations are satisfied so does its weak limit. Now we try to prove the third relation of   (\ref{coneDef}).

 We apply the mean value theorem in $J$ for $0<\theta_n <1$, and we employ the relations (\ref{1/n})-(\ref{vk}) to derive
 \begin{eqnarray*}
 % \nonumber to remove numbering (before each equation)
    J(\bar g_1, \bar g_2)+\frac{\rho^2_n}{2n}&>&J(g_{1_n}, g_{2_n})
\\&\magent{=}&\magent{J(\bar g_1+\rho_n v_{1_n},\bar g_2+\rho_n v_{2_n})}\\&\magent{=}&\magent{J(\bar g_1, \bar g_2)+\rho_n J'(\bar g_1+\theta_n\rho_n v_{1_n},\bar g_2+\rho_n v_{2_n}),}
 \end{eqnarray*}
 and
 \begin{eqnarray}\label{to 0}
 J'(\bar g_1+\theta_n\rho_n v_{1n},\bar g_2+\rho_n v_{2n})<\frac{\rho_n}{2n}\to 0.
 \end{eqnarray}
 We prove now that $J'(\bar g_1+\theta_n\rho_n v_{1n},\bar g_2+\rho_n v_{2n})\to
 J'(\bar g_1,\bar g_2)(v_{1},v_{2})$. Setting also $(g_{1{\theta_n}}, g_{2{\theta_n}})=(\bar g_1+\theta_n\rho_n v_{1n},\bar g_2 +\theta_n\rho_n v_{2n})$, from (\ref{1/n})--(\ref{vk}) we take that
 $(g_{1{\theta_n}}, g_{2{\theta_n}})\to(\bar g_1,\bar g_2)$ strongly in ${L^2[0,T;L^2(\Omega)\times L^2(\Omega)]}$, and so does %the associated state $(y_{1{\theta_n}},y_{2{\theta_n}})\to (\bar y_1,\bar y_2)$  and
the adjoint state $(\mu_{1{\theta_n}},\mu_{2{\theta_n}})\to (\bar \mu_1,\bar \mu_2)$ in  $H^{2,1}([0,T]\times\Omega)\cap C([0,T;H^1_0])$, then
 \begin{eqnarray}
J'(\bar g_{1}+\theta_n\rho_n v_{1n},\bar g_2+\rho_n v_{2n})&=&\int_0^T\int_{\Omega} \Big( (\mu_{1{\theta _n}}+\gamma_1 g_{1{\theta _n}})v_{1n},(\mu_{2{\theta _n}}+\gamma_2 g_{1{\theta _n}})v_{2n}\Big)dxdt
\nonumber\\ &&\to \int_0^T\int_{\Omega} \Big( (\bar\mu_{1}+\gamma_1 \bar g_{1})v_{1},(\bar\mu_{1}+\gamma_2 \bar g_{1})v_{2}\Big)dxdt=J'( \bar g_1,\bar g_2)(v_1,v_2). \nonumber
\end{eqnarray}
Subsequently (\ref{to 0}) gives  $J'(\bar g_{1},\bar g_{2})\le 0$, but multiplying $\bar\mu_{i}+\gamma_i \bar g_{i}$ with $v_{i}$ from (\ref{d_i(x,t)}), and the first two relations in (\ref{coneDef}) respectively we show that $(\bar\mu_{i}+\gamma_i \bar g_{i})v_{i}>0,i=1,2$ for a.a. $(t,x)\in [0,T]\times \Omega$ and then
 \begin{eqnarray*}
0&\magent{\le}& \int_0^T\int_{\Omega} \Big( (\bar\mu_{1}+\gamma_1 \bar g_{1})v_{1}+(\bar\mu_{2}+\gamma_2 \bar g_{1})v_{2}\Big)dxdt\\
&\magent{=}&\int_0^T\int_{\Omega}  (\bar\mu_{1}+\gamma_1 \bar g_{1}),\mu_{2}+\gamma_2 g_{1}))(v_{1},v_{2})dxdt\\
&\magent{=}&J'( \bar g_1,\bar g_2)(v_1,v_2)\le 0,
\end{eqnarray*}
and we have the third relation of (\ref{coneDef}).
%----------------------------------------------------------------------------------------------------

We will prove that  $J''(\bar g_1, \bar g_2)(v_1^2,v_2^2) \le 0$ and then due to (\ref{2nd cond}), it can be true only when $(v_1,v_2)=(0,0)$. From relations (\ref{1/n})-(\ref{vk}) and a Taylor expansion for a $0<\theta _n<1$
\begin{eqnarray*}
% \nonumber to remove numbering (before each equation)
% &
J(\bar g_{1},\bar g_2) +\rho_n  J'(\bar g_{1},\bar g_2)(v_{1n},v_{2n})%&\\
%  &
+\frac{\rho^2_n}{2}J''(\bar g_1+\theta_n\rho_n v_{1n},\bar g_2+\theta_n\rho _n v_{2n})(v_{1n}^2,v_{2n}^2)< J(\bar g_{1},\bar g_2)+\frac{\rho^2_n}{2n},
%&
\end{eqnarray*}
 and (\ref{1st cond}) gives
\begin{eqnarray*}
J''(\bar g_1+\theta_n\rho_n v_{1n},\bar g_2+\theta_n\rho _n v_{2n})(v_{1n}^2,v_{2n}^2)<\frac{1}{n},
\end{eqnarray*}
and
\begin{eqnarray}\label{326}
J''(\bar g_1,\bar g_2)(v_{1}^2,v_{2}^2)\le \mathop {\lim }\limits_{n \to \infty } \inf J''(\bar g_1+\theta_n\rho_n v_{1n},\bar g_2+\theta_n\rho _n v_{2n})(v_{1n}^2,v_{2n}^2)\le 0.\text{ }\text{ }
\end{eqnarray}
For the pass to the limit, we set the control $(g_{1{\theta_n}}, g_{2{\theta_n}})=(\bar g_1+\theta_n\rho_n v_{1n},\bar g_2  +\theta_n\rho_n v_{2n})$, the state $(y_{1{\theta _n}},y_{2{\theta _n}})=G(g_{1{\theta_n}},g_{2{\theta_n}})$ and  $( \mu_{1{\theta _n}},\mu_{2{\theta _n}})$ the associate adjoint state.
Let $(z_{1{\theta _n}},z_{2{\theta _n}})=G'(g_{1{\theta _n}},g_{2{\theta _n}})(v_{1n},v_{2n})$
then
\begin{eqnarray}
J''(g_{1{\theta_n}},g_{2{\theta_n}})({v^2_{1{\theta_n}}},{v^2_{2{\theta_n}}})=\int_0^T\int_{\Omega} \Big(z_{1{\theta _n}}^2+\gamma_1 {v^2_{1n}},z_{2{\theta _n}}^2+\gamma_2 {{v_{2n}}}\Big)dxdt.\label{327}
\end{eqnarray}
Now we can pass the limit using that $z_{i{\theta _n}}\rightharpoonup z_{v{i}}$ (weakly in $H^{2,1}(\Omega\times [0,T])$), and $\mu_{i{\theta_n}}\to \bar\mu_{{i}}$ (strongly in $H^{2,1}(\Omega\times [0,T])$), we also use for the $v_{in}$ integral the lower semi-continuity concerning the weak topology of ${L^2[0,T;L^2(\Omega)\times L^2(\Omega)]}$.
%-----------------------------------------------------------------------------------

Finally, since $(v_1,v_2)=0$, we take  $z_{i{\theta _n}}\rightharpoonup 0$ and due to (\ref{326}), (\ref{327}) and the identity $\|(v_{1n},v_{2n})\|_{{L^2[0,T;L^2(\Omega)\times L^2(\Omega)]}}=\|\frac{1}{\rho _{n}}(\bar g_1-g_{1n},\bar g_2-g_{2n})\|=\frac{1}{\|(\bar g_1-g_{1n},\bar g_2-g_{2n})\|
%_{L^2[0,T;L^2(\Omega)\times L^2(\Omega)]}
}\|(\bar g_1-g_{1n},\bar g_2-g_{2n})\| =1$ we take that
%
%&\rho _{n}=\|(\bar g_1-g_{1n},\bar g_2-g_{2n})\|_{L^2[0,T;L^2(\Omega)\times L^2(\Omega)]} \text{ and}& \\\label{vk}&(v_{1n},v_{2n})=\frac{1}{\rho _{n}}(\bar g_1-g_{1n},\bar g_2-g_{2n})&
%
\begin{eqnarray*}
% \nonumber to remove numbering (before each equation)
 0&\ge& \mathop {\lim }\limits_{n \to \infty } \inf J''( g_{1{\theta_n}}, g_{2{\theta_n}})({{v^2_{1{\theta_ n}}}},{{v^2_{2{\theta_n}}}})
   \\
   &=&\int_0^T\int_{\Omega} \Big(z_{1{\theta _n}}^2+\gamma_1 ,z_{2{\theta _n}}^2+\gamma_2 \Big)dxdt
\\&=& (\gamma _1,\gamma _2),
\end{eqnarray*}
which is a contradiction.
\end{proof}
\section{Robin boundary control}\label{sec:Robin1}
\subsection{The optimality system}
%\re{NA bEI KATI PIO APLO TO EBALA STO DISTRIBUTED Below, we state the optimality system which consists of the state equation given in the weak form, % (\ref{eqn:2.3})
%the adjoint, and the optimality condition applying Robin boundary control. Particularly, first-order necessary conditions (optimality system) of the above optimal control problems are demonstrated. %, we also refer to \cite{ChKa12} for a simpler . %, and the convergence rate are expected to be similar to the uncontrolled problem, so we omit the proves. %Also the second order necessary conditions are proved for the Fitzhugh Nagumo systems %. %but and we suggest the reader to the paragraph [Section 5] in this work to see the proof of the second order conditions in the quite  similar Fitzugh - Nagumo semilinear problem but
%with zero Dirichlet conditions.
%The aforementioned systems %, and the numerical experiments that follow
 %are introduced by employing a discontinuous in time Galerkin (dG) scheme and a conforming Galerkin method in space. The corresponding optimality system (first-order necessary conditions) consists of a primal (forward in time) equation system and an adjoint (backward in time) equation system which are coupled through an optimality condition, and non-linear terms  %, see e.g. \cite{ChKa12}
 %as it is described in  Lemma \ref{Blem:2.5}.
%The main aim is to show that the dG approximations of the optimality system exhibit good behavior and to examine the crucial matter of the distance of the solution from the desired target.}
%
{\magent{{%KATI PIO APLO  
Similarly to subsection \ref{subsec:3.1}, we again state the optimality system, namely the state equation, % given in the weak form, % (\ref{eqn:2.3})
the adjoint, the Robin optimality condition%applying Robin boundary control
,  the first-order necessary conditions, %which actually is the optimality system %) of the above optimal control problems are demonstrated. %, we also refer to \cite{ChKa12} for a simpler . %, and the convergence rate are expected to be similar to the uncontrolled problem, so we omit the proves. %Also the second order necessary conditions are proved for the Fitzhugh Nagumo systems %. %but and we suggest the reader to the paragraph [Section 5] in this work to see the proof of the second order conditions in the quite  similar Fitzugh - Nagumo semilinear problem but
%with zero Dirichlet conditions.
%The aforementioned systems %, and the numerical experiments that follow
 %are introduced by 
employing a discontinuous in time dG scheme and a conforming Galerkin method in space. The primal/forward in time system and the adjoint/backward in time system coupled through an optimality condition, and the non-linear terms  %, see e.g. \cite{ChKa12}
  are  as it is  described in  Lemmas % \ref{Blem:2.5} 
below.}}
%
%!!!!!!!!!!!!!!!!!!!!!!!!!!!!!!!!!!!!!!!!!!
\green{

We begin by stating the weak formulation of the state equation. Given $f_1$, $f_2\in L^2\left [0,T;H^{*}(\Omega)\right ]$, controls $g_1,g_2\in L^2\left[0,T;H^{-1/2}(\Gamma)\right ]$, and {{initial}}  states ${y_{1,0}}$, ${y_{2,0}} \in L^2(\Omega)$
we seek $y_1,y_2\in L^2[0,T;H^1(\Omega)] \cap L^2[0,T;H^{1}(\Omega)^*]$ such that for a.e. $t\in(0,T]$, and for all $v \in H^1(\Omega)$,
\begin{eqnarray}\label{eqnn:2.1}
 %\qquad
 \left\langle {{\magent{y_{1t}}} ,v} \right\rangle  + \epsilon_1(\nabla y_1,\nabla v )  - ((a-by_2)y_1,v) 
+ \lambda_1\left\langle y_1, v\right\rangle_\Gamma = \left\langle {f_1,v } \right\rangle
+ \left\langle {g_1,v } \right\rangle_\Gamma \quad\mbox{ and}\quad
  \left( {y_1(0),v } \right) &=& \left( {{y_{1,0}} ,v } \right), \nonumber \\
 %\qquad
   \left\langle {\magent{{y_{2t}}} ,v} \right\rangle  + \epsilon_2  (\nabla y_2, \nabla v )  -((cy_1-d)y_2,v) +\lambda_1\left\langle y_2,v\right\rangle_\Gamma = \left\langle {f_2,v }\right\rangle+\left\langle g_2,v \right\rangle_\Gamma  \quad\mbox{ and}\quad
  \left( {y_2(0),v } \right) &=& \left( {{y_{2,0}} ,v } \right).\nonumber\label{eqnn:2.2}
\end{eqnarray}
An equivalent \magent{well-posed} weak formulation which is more suitable for the analysis of dG schemes is to seek
$y_1,y_2\in\magent{W_R(0,T):= L^2 [0,T;H^1(\Omega)] \cap L^{\infty} [0,T;L^{2} (\Omega)]\times L^{2} [0,T;L^{2} (\Gamma)]}$ such that for all $v \in L^2[0,T;H^1(\Omega)] \cap H^1[0,T;H^1(\Omega)^*]$,
\begin{eqnarray}
&&  (y_1(T),v(T)) + \int_0^T {\left( { - \left\langle {y_1,v_t } \right\rangle  + \epsilon_1 \left( {\nabla y_1,\nabla v } \right)-\left( {(a-by_2)y_1,v }\right) + \left\langle {y_1,v }\right\rangle_\Gamma
}\right)}dt \nonumber \\
&& \qquad   = ({y_{1,0}} ,v(0)) + \int_0^T {  {\left\langle {f_1,v }  \right\rangle  }  } dt+ \int_0^T {\left\langle{g_1,v }\right\rangle_\Gamma} dt, \label{eqnnn:2.1}
\\
&&  (y_2(T),v(T)) + \int_0^T {\left( { - \left\langle {y_2,v_t } \right\rangle  + \epsilon_2 \left( {\nabla y_2,\nabla v } \right)
      -({(cy_1-d)y_2,v } ) +\left\langle {y_2,v }\right\rangle_\Gamma}    \right)} dt \nonumber \\
&& \qquad   = ({y_{2,0}} ,v(0)) + \int_0^T { { \left <f_2,v \right >}  } dt+ \int_0^T \left\langle {g_2,v }\right\rangle_\Gamma dt.\label{eqnnn:2.2}
\end{eqnarray}
The control to state mapping $G: L^2[0, T;{L^2(\Gamma)\times L^2(\Gamma)]} \to W_R(0, T)\times W_R(0, T)$, which associates to each control $g_1, g_2$ the corresponding state $G(g_1,g_2)$ $ = ({y}_{g_1},{y}_{g_2})$ $\equiv$ $(\magent{y_1(g_1,g_2)},\magent{y_2(g_1,g_2)})$ via (\ref{eqnnn:2.1})--(\ref{eqnnn:2.2}) is well defined, and continuous, so does the cost functional, frequently denoted to by its reduced form, $J(y_1,y_2,g_1,g_2) \equiv J(\magent{y_1(g_1,g_2)},\magent{y_2(g_1,g_2)}): L^2[0,T;{L^2(\Gamma)}] \to \mathbb R$. %is also well-defined and continuous.
\begin{defn}\label{def:2.2}
Let ${f_1, f_2}\in L^2 [0,T;H^1(\Omega)^*]$, ${y_{1,0}, y_{2,0}} \in L^2(\Omega)$, and ${y_{1,d}, y_{2,d}} \in L^2[0,T;L^2(\Omega)]$\magent{,%$g_{ia}$, $g_{ib} \in \mathbb{R}$
}
  \green{be given data. Then, the set of admissible controls denoted by ${\mathcal A}_{ad}$ for the corresponding Robin boundary control problem with unconstrained controls
takes the form:
%\begin{enumerate}
%\item {\it Unconstrained Controls}:
%${\mathcal A}_{ad} \equiv L^2[0,T;L^2(\Gamma)]$.
%\item {\it Constrained Controls}:
${\mathcal A}_{ad} = \{ {g_i} \in L^2[0,T;L^2(\Gamma)]$. %: {{g_{ia}}} \leq {g_i}(t,x) \leq {{g_{ib}}}, i=1,2 \text{ for a.e. } (t,x) \in (0,T) \times \Omega \}.$
%\end{enumerate}
The pairs $(\magent{y_i(g_1,g_2)},g_i) \in W_R(0,T) \times {\mathcal A}_{ad},i=1,2$,  is said to be an optimal solution if $J(\magent{y_1(g_1,g_2)},\magent{y_2(g_1,g_2)},g_1,g_2)\le J(\magent{w_1(h_1,h_2),w_2(h_1,h_2),}h_1,h_2)$\magent{, for all $(h_1 , h_2 ) \in \mathcal{A}_{ad }\times \mathcal{A}_{ad}$, $(w _1 (h _1\magent{,h_2}), w _2 (\magent{h_1,}h _2 )) \in W_ R (0, T ) \times W_ R (0, T )$ being the solution of problem (\ref{eqnnn:2.1})--(\ref{eqnnn:2.2}) with control functions $(h _1 , h _2 )$}%\forall (w_i(h_i),h_i) \in W_D(0,T) \times {\mathcal A}_{ad}, i=1,2
.}
\end{defn}%\newline
We will again  occasionally abbreviate the notation $y_i \equiv {y}_{g_i} \equiv {y_i}( \magent{g_1,g_2}),i=1,2$ and an optimality system of equations can be derived %by using standard techniques under minimal regularity assumptions on the given initial data and forcing term; see for instance 
\cite{ChKa14}% or \cite[Section 2]{ChGuHo06}
. }\green{Next, we first state the basic differentiability property of the cost functional.}
\green{\begin{lem} \label{lem:2.4}
The cost functional $J:L^2[0,T;{L^2(\Gamma)}] \to \mathbb R$ is of class $C^{\infty}$ and for every $g_1,g_2,u_1,u_2 \in L^2[0,T;{L^2(\Gamma)}]$,
\[J^{'}(g_1,g_2)(u_1,u_2) = \int_0^T \int_\Gamma \Big (( {\mu_1} (\magent{g_1,g_2}) + \gamma_1 g_1) u_1 , ( {\mu_2} (\magent{g_1,g_2}) + \gamma_2 g_2) u_2\Big ) dxdt, \]
where $\mu _i(\magent{g_1,g_2}) \equiv {\mu}_{g_i} \in W_R(0,T), i=1,2,$ is the unique solution of following problem: For all $v \in L^2[0,T;H^1(\Omega)] \cap H^1[0,T;H^1(\Omega)^*]$,
\begin{eqnarray} %\label{eqn:2.4}
&&   \int_0^T {\left( {  \left\langle {{\mu}_{g_1},v_t } \right\rangle  + \epsilon_1 \left( {\nabla{\mu}_{g_1},\nabla v } \right)
    - ((a-by_{g_2}){\mu}_{g_1},v )   }    +\lambda_1\left\langle {\mu}_{g_1},v \right\rangle_\Gamma\right) dt} \nonumber \\
&& \qquad \qquad\qquad\qquad\qquad\qquad\qquad\qquad
 = -({\mu}_{g_1}(0) ,v(0)) + \int_0^T \left( {y}_{g_1}-y_{1,d},v   \right)dt,\\
&&   \int_0^T {\left( {  \left\langle {{\mu}_{g_2},v_t } \right\rangle  + \epsilon_2\left( {\nabla{\mu}_{g_2},\nabla v } \right)
    -((c{y}_{g_1}-d){\mu}_{g_2},v ) + \lambda_2\left\langle {\mu}_{g_2},v \right\rangle_\Gamma  }    \right) dt} \nonumber \\
&& \qquad \qquad \qquad\qquad\qquad\qquad\qquad\qquad
 = -({\mu}_{g_2}(0) ,v(0)) + \int_0^T \left( {y}_{g_2}-{y_{2,d}},v   \right)dt.
\end{eqnarray}
where $\mu_{g_i}(T)=0$ and $(\mu_{g_i})_t \in L^2[0,T;H^1(\Omega)^*], i=1,2$.
\end{lem}
In the following Lemma, we state the optimality system which consists of the state equation% (given in the weak form (\ref{eqn:2.3}))
, the adjoint, and the optimality condition.
\begin{lem} \label{Blem:2.5} Let $y_{g_i} \equiv y_i \in W_R(0,T)% \times {\mathcal A}_{ad}
, \text{ and } g_i\in L^2(\Gamma)$ denote the unique optimal pairs, for $i=1,2$% of Definition \ref{Bdefn:2.2}
. Then, there exist adjoints $\mu_1,\mu_2 \in W_R(0,T)$ satisfying, $\mu_1(T) =\mu_2(T)=0$ such that for all $v  \in L^2[0,T;H^1(\Omega)] \cap H^1[0,T;H^{1}(\Omega)^*]$,
 \begin{eqnarray}%\label{eqn:2.5}
&&  (y_1(T),v(T)) + \int_0^T {\left( { - \left\langle {y_1,v_t } \right\rangle  + \epsilon_1  \left( {\nabla y_1,\nabla v } \right)-\left( {(a-by_2)y_1,v } \right)
       }    \right)} dt + \lambda _1 \int_0^T { { \left <y_1,v \right > }_{\Gamma}} dt\nonumber \\
&& \qquad \qquad \qquad \qquad \qquad    = ({y_{1,0}} ,v(0)) + \int_0^T {\left <f_1%-y_2
,v    \right >} dt+ \lambda _1 \int_0^T { { \left < g_1,v \right >}_{\Gamma}} dt,
\\
&&  (y_2(T),v(T)) + \int_0^T {\left( { - \left\langle {y_2,v_t } \right\rangle  + \epsilon_2 \left( {\nabla y_2, \nabla v } \right)
       }    \right)} dt - \int_0^T { ( (cy_1-d) y_2,v  )  } dt + \lambda _2 \int_0^T { { \left <y_2,v \right > }_{\Gamma}} dt \nonumber \\
&& \qquad \qquad \qquad \qquad \qquad    = ({y_{2,0}} ,v(0)) + \int_0^T   \left < f_2 ,v \right > dt
+ \lambda _2\int_0^T {\left <g_2,v \right >}_{\Gamma} dt,\\
%\end{eqnarray}
%\begin{eqnarray}
&&   \int_0^T {\left( {  \left\langle {{\mu}_{1},v_t } \right\rangle  + \epsilon_1 \left( {\nabla {\mu}_{1},\nabla v } \right)
     -({(a-by_2)\mu}_{1},v )   }    \right) dt} + \lambda _1 \int_0^T { { \left <\mu_1,v \right > }_{\Gamma}} dt\nonumber\\
&& \qquad \qquad \qquad \qquad \qquad   = -({\mu}_{1}(0) ,v(0)) + \int_0^T \left( {y}_{1}-{y_{1,d}},v   \right)dt, \label{Beqn:2.6mu1}\\
&&   \int_0^T {\left( {  \left\langle {{\mu}_{2},v_t } \right\rangle  + \epsilon_2\left( {\nabla{\mu}_{2},\nabla v } \right)
   -({(cy_1-d)}{\mu}_{2},v )   }    \right) dt} + \lambda _2 \int_0^T { { \left <\mu_2,v \right > }_{\Gamma}} dt\nonumber \\
&& \qquad \qquad \qquad \qquad \qquad    = -({\mu}_{2}(0) ,v(0)) + \int_0^T \left( {y}_{2}-{y_{2,d}},v   \right)dt,\label{Beqn:2.6mu2}
\end{eqnarray}
with control functions %also %constraints
 to satisfy:
%\begin{equation} \label{Beqn:2.8}
%\int_0^T \int_\Omega \Big (\left ( \alpha g_1+   \mu_1 \right) \left ( u_1-g_1 \right ),( \alpha g_2+   \mu_2 ) \left ( u_2-g_2 \right )
%\Big)
%dx dt \geq 0 \quad \forall u_1,u_2 \in {\mathcal A}_{ad}.
%\end{equation}
%In addition, ${y_{i,t}}, {\mu_i}_t \in L^2[0,T;H^{-1}(\Omega)]$, and note that (\ref{eqn:2.8}), is equivalent to
$$\int_0^T\left\langle{\gamma_1g_ 1(t,x)  %Proj_{[{g_{i,a}},{g_{i,b}}]} \left (
+\lambda_1 {}{} \mu_1 (t,x)},u\right\rangle_{\Gamma} dt=0, \,
\int_0^T\left\langle{\gamma_2g_2(t,x)  %Proj_{[{g_{i,a}},{g_{i,b}}]} \left (
+\lambda_2  \mu_2 (t,x)},u\right\rangle_{\Gamma} dt=0,
%\right ) %\equiv Proj_{[{g_i}_a,{g_i}_b]} \left ( \frac{1}{\alpha} \frac{\partial \mu_1 (t,x)}{\partial {\bf n}} \right )
$$ for a.e. $(t,x) \in (0,T]\times \Gamma%\Omega
$ and $\forall u\in L^2[0,T;L^2(\Gamma)]$. In addition, ${y_{1,t }}$, ${y_{2,t }}\in L^2[0,T;H^1(\Omega
)^*]$, ${\mu_{1}}$, ${\mu_{2}}\in L^2[0,T;H^2(\Omega
)] \cap H^1[0,T;L^2(\Omega)]$.
\end{lem}}
%\subsubsection{The fully-discrete {\magenta{distributed}} optimal  control problem}
}}
\subsection{The fully discrete Robin boundary type optimal control problem}\label{sec:Robin}
For the Robin boundary control, we employ again a discretization which allows the presence of
discontinuities in time, i.e., we define,
$
{\mathcal G}_{\blue{h,R}} = \{ g_{ih} \in L^2[0,T;L^2(\Gamma)] :
g_{ih}|_{(t^{n-1},t^n]} \in {\mathcal P}_k[t^{n-1},t^n; G_{\blue{h,R}}], i=1,2 \}.
$
Here, a conforming subspace $G_{\blue{h,R}} \subset L^2(\Gamma)$ is specified
at each time interval $(t^{n-1},t^n]$, which satisfy
standard approximation properties.  To summarize, for the choice of piecewise linear in space, we choose:
\begin{eqnarray*}
&&U_{\blue{h,R}} = \{ v_h \in C(\bar {\Omega}): v_h|_{T} \in {\mathcal P}_1, \mbox{ for all $T \in {\mathcal T}_h$} \},\\
&&G_{\blue{h,R}} = \{ u_h \in C(\Gamma): u_h |_{[x_{i,\Gamma}, x_{i+1,\Gamma}]} \in {\mathcal P}_1, \mbox{ for $i=1,...,N_h$} \}.
\end{eqnarray*}
For the control variable, %we have two separate choices depending for the constrained and the unconstrained case.In both cases
our discretization is motivated by the optimality condition. %We begin by the unconstrained case, where emphasis is placed on the regularity on the given data. %\\

The physical meaning of the optimization problem under consideration is to seek
states $y_{ih}$ and controls $g_{ih}$, $i=1,2$, such that $y_{ih}$ is as close as possible to a given target ${y_{i,d}}$ applying Laplace multipliers $\mu_{1h}$, $\mu_{2h}$. Here, %$\Omega$ denotes a bounded domain in $\mathbb R^2$, with Lipschitz boundary $\Gamma$,
$y_{i,0}$ denotes %, $f_i$ are
 the initial
data% and the forcing term respectively, and ${\gamma _i}$ are  penalty parameters which measure the size of the control
, $i=1,2$.
%The scope of this work is to study the "distance" between the solution and the desired target function.
After applying known techniques, \cite{ChKa14}, the fully-discrete optimality system may be described in Lemma \ref{lemmm:3.4Lo}.

%$$$$$$$$$$$$$$$$$$$$$$$$$$$$$$$$$$$4
\green{
The discontinuous time-stepping fully-discrete scheme for the control to state mapping $G_{{h,R}}:L^2[0,T;{L}^2(\Gamma)\times {L}^2(\Gamma)] \to {\mathcal U}_{{h,R}}\times {\mathcal U}_{{h,R}}$,
which associates to each control $(g_1, g_2)$ the corresponding state $G_{{h,R}}(g_1,g_2) = ({y_1}_{g_1,h},{y_2}_{g_2,h}) \equiv (y_{1h}({g_1}),y_{2h}({g_2}))$ is defined as follows: For any control data $g_1,g_2 \in L^2[0,T;L^2(\Gamma)]$, for given initial data ${y_{1,0}}, {y_{2,0}} \in L^2(\Omega)$, forces $f_1,f_2 \in L^2[0,T;H^1(\Omega)^*]$, and targets $y_{1,d}$, $y_{2,d} \in L^2[0,T;L^2(\Omega)]$ we seek $y_{1h}$, $y_{2h} \in \mathcal U_{h,R}$ such that for
$n=1,...,N$, and for all $v_h \in {\mathcal P}_k[t^{n-1},t^n;U_{{h,R}}]$,
\begin{eqnarray}\label{eqnnn:3.2}
&&  
(y_1^n, v^n) + \int_{t^{n-1}}^{t^n}{\left(-(y_{1h},v_{ht}) + \epsilon_1 \left( \nabla y_{1h},\nabla v_h \right) - \left( (a-by_{2h})y_{1h},v_h   \right)  
+\lambda_1 \left\langle y_{1h}, v_h\right\rangle
%}
_\Gamma \right) }
 dt \nonumber \\
&& \qquad   = ({y_1}^{n-1} ,v_+^{n-1}) + \int_{t^{n-1}}^{t^n1} \left\langle {f_1,v_h } \right\rangle dt 
+  \int_{t^{n-1}}^{t^n} {\lambda_1%\left ( 
 {\left\langle{g_{1h}, v_h}\right\rangle}_\Gamma  
%\right )
} dt,
\\
&&  (y_2^n,v^n) + \int_{t^{n-1}}^{t^n} {
\left( { - ( {y_{2h},v_{ht} } )  + \epsilon_2 \left( {\nabla y_{2h},\nabla v_h } \right)}    \right)} dt 
- \int_{t^{n-1}}^{t^n} \left
 ({\left( (cy_{1h}-d)y_{2h},v_h \right)
+ \lambda_2
\left\langle{y_{2h}, v_h}\right\rangle_\Gamma
}\right ) dt\nonumber \\
&& \qquad   = ({y_2}^{n-1} ,v_+^{n-1}) + \int_{t^{n-1}}^{t^n} { \left <f_2,v_h \right >} dt
+ \int_{t^{n-1}}^{t^n} {\lambda_2\left\langle {g_{2h}, v_h}\right\rangle}_\Gamma
 dt.  \label{eqnnn0:3.3}
\end{eqnarray}
%\begin{rem}We note that in the above definition only $g_1,g_2 \in L^2[0,T;L^2(\Omega)]$ regularity is needed to validate the weak fully-discrete formulation. As a matter of fact, the %\TODO
%{stability estimate at arbitrary time-points} as well as in $L^2[0,T;H_0^1(\Omega)]$ %and $L^2[0,T;L^2(\Gamma)]$
% norm easily follows by setting $v_h = y_{1h}$, $v_h = y_{2h}$ into (\ref{eqn:3.2})--(\ref{eqn0:3.3}) while for the arbitrary time-points stability estimate, we may apply the techniques %which were developed 
%as in \cite{ChKa14,ChKa12}. %[Section 2]{ChWa06}
%%for general linear parabolic PDEs, (see also \cite{ChKa12}
%%for stability estimate for semilinear parabolic PDEs with Dirichlet data).
%\end{rem}
Analogously to the distributed control case, 
\magent{we will occasionally abbreviate the notation $y_{ih} = y_{ih} (g_{ih} )$ and} we note that the control to fully-discrete state mapping $G_{{h,R}} :L^2[0,T;{{L}}^2(\Gamma)\times {{L}}^2(\Gamma)] \to {\mathcal U}_{{h,R}}\times {\mathcal U}_{{h,R}}$, is well defined, and continuous and let \magent{$J_h(y_{1h},y_{2h},g_{1h},g_{2h}) = \frac{1}{2}\int_0^T \int_{\Omega} (| y_{1h}-{y_{1,d}}|^2 +| y_{2h}-{y_{2,d}}|^2) dxdt  +
\frac{\gamma_1 }{2}\int_0^T \int_{\Gamma} (| g_{1h}|^2) dxdt+
\frac{\gamma_2 }{2}\int_0^T \int_{\Gamma} (| g_{2h}|^2) dxdt$}. 

The control problem definition now takes the form:
\begin{defn}\label{defn:3.1}
Let $\magent{f_1,f_2}\in L^2 [0,T;H^1(\Omega)^*]$, $\magent{y_{1,0}, y_{2,0}} \in L^2(\Omega)$, $\magent{y_{1,d}, y_{2,d}} \in L^2[0,T;L^2(\Omega)]$, be given data.
 Suppose that the set of discrete admissible controls is denoted by ${\mathcal A}^d_{ad} \equiv \mathcal{G}_h%{\mathcal U}_{\blue{h,D}}
 \cap {\mathcal A}_{ad}$\magent{. T}he pairs $(y_{ih},g_{ih}) \in {\mathcal U}_{{h,R}} \times  {\mathcal A}^d_{ad}%L^2[0,T; U_{\blue{h,D}}%^n
, i=1,2,$ satisfy (\ref{eqnnn:3.2})--(\ref{eqnnn0:3.3}) and the definition of the corresponding control problem takes the form: The pairs $(y_{ih},g_{ih})\in A^d_{ad},i=1,2$,  are said to be optimal solutions if $J_h(y_{1h},y_{2h},g_{1h},g_{2h})\le J_h(w_{1h},w_{2h},u_{1h},u_{2h})$ \magent{for all $(u _{1h} , u_{2h} ) \in \mathcal{A} ^d_{ad}\times \mathcal{A} ^d_{ad}, (w_{1h} , w_{2h} ) \in \mathcal{U} _{h,R} \times \mathcal{U} _{h,R}$ being the solution of problem (\ref{eqnnn:3.2})-(\ref{eqnnn0:3.3}) with control functions $(u_{h1} , u_{h2} )$}.
\end{defn}

The discrete optimal control problem solution existence can be proved by standard techniques while uniqueness follows from the structure of the functional, and the linearity of the equation. The basic stability estimates in terms of the optimal pair $(y_{ih},g_{ih}) \in W_R \times L^2[0,T;L^2(\Gamma)],i=1,2,$ can be easily obtained, see e.g.  \cite{ChKa14}. 
We note that the key difficulty of the discontinuous time-stepping formulation is the lack of any meaningful regularity for the time-derivative of $y_{ih}$ due to the presence of discontinuities. However,  it is also worth noting that dG in time is applicable even for higher order schemes, \cite{ChKa14}.
%\subsubsection{The discrete {\blue{distributed}} optimality system}

In the following Lemmas \ref{lemmm:3.3} and  \ref{lemmm:3.4Lo}  we state the discrete optimality systems in two appropriate forms.
\begin{lem} \label{lemmm:3.3}
The cost functional $J_h:L^2[0,T;{L^2(\Gamma)}] \to \mathbb R$ is well defined differentiable and for every $g_1,g_2,u_1,u_2 \in L^2[0,T;L^2(\Gamma)]$,
\[J_h^{'}(g_1,g_2)(u_1,u_2) = \int_0^T \int_\Gamma \Big (( {\mu_{1h}} ({g_1}) + \alpha g_1) u_1 , ( {\mu_{2h}} ({g_2}) + \alpha g_2) u_2\Big ) dxdt, \]
where $\mu _{ih}({g_i}) \equiv {\mu}_{{g_i},h} \in W_R(0,T), i=1,2,$ is the unique solution of following problem: For all $\upsilon _h \in P_k[t^{n-1},t^n;U_{{h,R}}], n=1,...N$,
\begin{eqnarray} \label{eqn:3.3}
&&  -(\mu_{{g_1},h+}^n,\upsilon ^n)+ \int_{t^{n-1}}^{t^n} {\left( {  ( {{\mu}_{{g_1},h},\upsilon_{ht} } )  + \epsilon_1 \left( \nabla \upsilon_h,{\nabla{\mu}_{g_1,h} } \right)
    -  ((a-by_{g_2,h}){\mu}_{g_1,h},\upsilon_h     })  + {\lambda_1\left\langle {{\mu}_{{g_1},h}, v_h}\right\rangle}_\Gamma \right) dt} 
\nonumber \\
&& \qquad   = -({\mu}_{g_1,h+}^{n-1} ,\upsilon_+^{n-1}) + \int_{t^{n-1}}^{t^n} \left({y}_{g_1,h}-{y_{1,d}},\upsilon_h   \right) dt,\\
&& -(\mu_{{g_2},h+}^n,\upsilon ^n)+  \int_{t^{n-1}}^{t^n} {\left( { ( {{\mu}_{g_2,h},\upsilon_{ht} })  + \epsilon_2\left( {\nabla{\mu}_{g_2,h},\nabla\upsilon_h } \right)
   -((c {y}_{g_1,h}-d){\mu}_{g_2,h},\upsilon_h )   }   + {\lambda_2\left\langle {{\mu}_{{g_2},h}, v_h}\right\rangle}_\Gamma  \right) dt} \nonumber \\
&& \qquad   = -({\mu}_{g_2,h+}^{n-1} ,\upsilon_+^{n-1}) + \int_{t^{n-1}}^{t^n} \left( {y}_{g_2,h}-{y_{2,d}},\upsilon_h   \right) dt,
\end{eqnarray}
where $\mu_{g_i,h+}^N=0$ and $g_i\in \mathcal A_{ad}^d$.
\end{lem}
}

%
%\begin{rem}
%\red
%TO EBALA PIO META SSTO ROBIN KEFALAIO 
\begin{lem} \label{lemmm:3.4Lo}
Let $(y_{ih}(%\magent{
g_{ih}%g_{1h},g_{2h}
)
%}
,g_{ih}) \equiv (y_{ih},g_{ih})
%\textcolor[rgb]{1.00,0.00,0.00}{\in {{\mathcal U_h}} \times L^2[0,T;G_h]}
$
denote the unique optimal pairs, $i=1,2$%of Definition \ref{defn:3.1}
. Then, there exist adjoints $\mu_{ih}
%\textcolor[rgb]{1.00,0.00,0.00}{\in {\mathcal U}_h}
$ satisfying, $\mu^N_{i+}=0$ such that for all polynomials with degree $k=0$ or $1$ in time in an appropriate polynomial space $U_h$ for the space discretization $v_h\in {{\mathcal{P}}_k[t^{n-1},t^n;U_{h,{\red{R}}}]}$,  and for all $n=1,...,N$
    \begin{eqnarray}%\label{eqn:2.5}
&&  (y_1^n,v^n) + \int_{t^{n-1}}^{t^{n}} {\left( { - \left\langle {y_{1h},v_{ht} } \right\rangle  + \epsilon_1  \left( \nabla{y_{1h},\nabla v_h } \right)-\left( {(a-by_{2h})y_{1h},v_h } \right)
       }    \right)} dt \nonumber \\
&& \qquad   + \lambda _1\int_{t^{n-1}}^{t^{n}} { { \left <y_{1h},v_h \right >} _{\Gamma}} dt= ({y_{1}^{n-1}} ,v^{n-1}_+) + \int_{t^{n-1}}^{t^{n}} {\left < { {f_1%-y_2
,v_h }   } \right >} dt+ \lambda _1 \int_{t^{n-1}}^{t^{n}} { { \left <g_{1h},v_h \right >}_{\Gamma}} dt,\label{full_discrete_y1}
\\
&&  (y_2^n,v^n) + \int_{t^{n-1}}^{t^{n}} {\left( { - \left\langle {y_{2h},v_{ht} } \right\rangle  + \epsilon_2 \left( \nabla{y_{2h},\nabla v_h } \right)
       }    \right)} dt - \int_{t^{n-1}}^{t^{n}} { {( {(cy_{1h}-d) y_{2h},v_h } ) } } dt\nonumber \\
&& \qquad   + \lambda _2\int_{t^{n-1}}^{t^{n}} { { \left <y_{2h},v_h \right >} _{\Gamma}} dt= ({y_{2}^{n-1}} ,v^{n-1}_+) + \int_{t^{n-1}}^{t^{n}} { { \left <f_2
,v_h \right >} } dt+ \lambda _2\int_{t^{n-1}}^{t^{n}} { { \left <g_{2h},v_h \right >} _{\Gamma}} dt,\label{full_discrete_y2}
\end{eqnarray}
\begin{eqnarray}
&&   -(\mu_{{1}+}^n,\upsilon ^n)+ \int_{t^{n-1}}^{t^{n}} {\left( {  \left\langle {{\mu}_{1h},v_{ht} } \right\rangle  + \epsilon_1 \left( {\nabla{\mu}_{1h},\nabla v_h } \right)
    - ({(a-by_{2h})\mu}_{1h},v_h )   }    \right) dt} + \lambda _1\int_{t^{n-1}}^{t^{n}} { { \left <\mu_{1h},v_h \right >} _{\Gamma}} dt\nonumber\\
&& \qquad\qquad    = -({\mu}_{1}^{n-1} ,v^{n-1}) + \int_{t^{n-1}}^{t^{n}} \left( {y}_{1h}-{y_{1,d}},v_h   \right)dt, \label{Beqn:2.6mu1bru}\\
&& -(\mu_{{2}+}^n,\upsilon ^n)+   \int_{t^{n-1}}^{t^{n}} {\left( {  \left\langle {{\mu}_{2h},v_{ht} } \right\rangle  + \epsilon_2\left( {\nabla{\mu}_{2h},\nabla v_h } \right)
    -((c{y_{1h}-d}){\mu}_{2h},v_h )   }    \right) dt} + \lambda _2\int_{t^{n-1}}^{t^{n}} { { \left <\magent{\mu}_{2h},v_h \right >} _{\Gamma}} dt\nonumber \\
&& \qquad \qquad    = -({\mu}_{2}^{n-1} ,v^{n-1}) + \int_{t^{n-1}}^{t^{n}} \left( {y}_{2h}-{y_{2,d}},v_h   \right)dt,\label{Beqn:2.6mu2bru}
\end{eqnarray}
%   %%%%%%%%%%%%%%%%%%%%%%%%%%%%%%%%%%%%%%%%%%%
%\begin{eqnarray}\label{eqn:3.5bru}
%   {-(\mu^n_{+},v^n)} &+&{ \int_{t^{n-1}}^{t^{n}} {\left(
%   {\left\langle {\mu_h,v_{ht} } \right\rangle  + \alpha \left( {\mu_h,v_h }
%    \right)
%    + \lambda\left\langle {\mu_h ,v_h } \right\rangle _\Gamma} \right)dt}   }\nonumber
% \hfill  \\\hfill
%  &=&  - (\mu^{n-1}_{+},v^{n-1}_{+})
%  + \int_{t^{n-1}}^{t^{n}} {\left( {y_h - y_d,v_h } \right)dt}
%  \end{eqnarray}
 and the following optimality condition holds: For all $u_h \in L^2[0,T;L^2(\Gamma)]%W_R%{\mathcal A}^d_{ad}
 $,
\begin{equation} \label{eqn:3.6bru}
 %\text{ 1)  Unconstrained Controls:}
 \int_0^T \left\langle\gamma _1 g_{1h}  + \lambda_1 \mu _{1h} ,u _h \right\rangle_{\Gamma} dt  = 0,\,\, \int_0^T \left\langle\gamma _2 g_{2h}  + \lambda_2 \mu _{2h} ,u _h \right\rangle_{\Gamma} dt  = 0.
\end{equation}
%\begin{equation} \label{eqn:3.7bru}
% \text{ 2) Constrained Controls:}
% \int_0^T \int_\Gamma \left ( \gamma _i g_{ih}  + \lambda _i \mu _{ih} \right ) \left (u _h-g_{ih} \right ) dx dt  \geq  0.
%\end{equation}
\end{lem}
In particular, given target functions ${y_{1,d}}(t,x_1, x_2)$ and
${y_{2,d}}(t,x_1, x_2)$,  we seek state variables $y_1(t,x_1, x_2)$, $y_2(t,x_1, x_2)$ and Robin boundary control variables $g_1(t,x_1, x_2)$, $g_2(t,x_1, x_2)$.
\section{Numerical Experiments}\label{sec:experiments}
The numerical examples are tested using the FreeFem++  software package Version 3.32, compiled in Linux OS, in a time interval $[0,T]$=$[0,0.1]$ and space $\Omega=[0,1]\times [0,1]$.
%
%Related to this model problem,  
The state $y_1(t,x_1,x_2)$ represents the prey population density at time $t$, and $y_2(t,x_1,x_2)$ represents the predator population density both at time $t$ and position $(x_1,x_2)$.
We focus onto driving the state solution to desired targets $({y_{1,d}},{y_{2,d}})$ minimizing the quadratic functional $J(y_1,y_2,g_1,g_2)$ as it is described in (\ref{functional_full}) for $S=\Omega$ or $\Gamma$
%\begin{eqnarray}\label{B1.1_lo}
%   J(y_1,y_2,g_1,g_2) &=& \frac{1}
%{2}\int_0^T {\left\| y_1-{y_{1,d}}\right\|_{L^2 (\Omega )}^2  dt
%}
%+\frac{\gamma _1}{2}\int_0^T {\left\| g_1 \right\|_{L^2 (\Gamma )}^2 dt
%}
%\nonumber\\&+&\frac{1}
%{2}\int_0^T {\left\| {{y_2-{y_{2,d}}}}\right\|_{L^2 (\Omega )}^2 dt
%}
%+\frac{\gamma _2}{2}\int_0^T {\left\| g_2 \right\|_{L^2 (\Gamma )}^2dt
%}
%\end{eqnarray}
subject to the constraints (\ref{LoD1})--(\ref{LoD2}) or (\ref{LoR1})--(\ref{LoR2}) respectively. Smooth initial conditions will be considered, 
\begin{figure}
   \centering
\scalebox{0.06}
    \centering
        \includegraphics[width=0.22\textwidth,trim={4cm 6cm 1.5cm  0cm},clip]{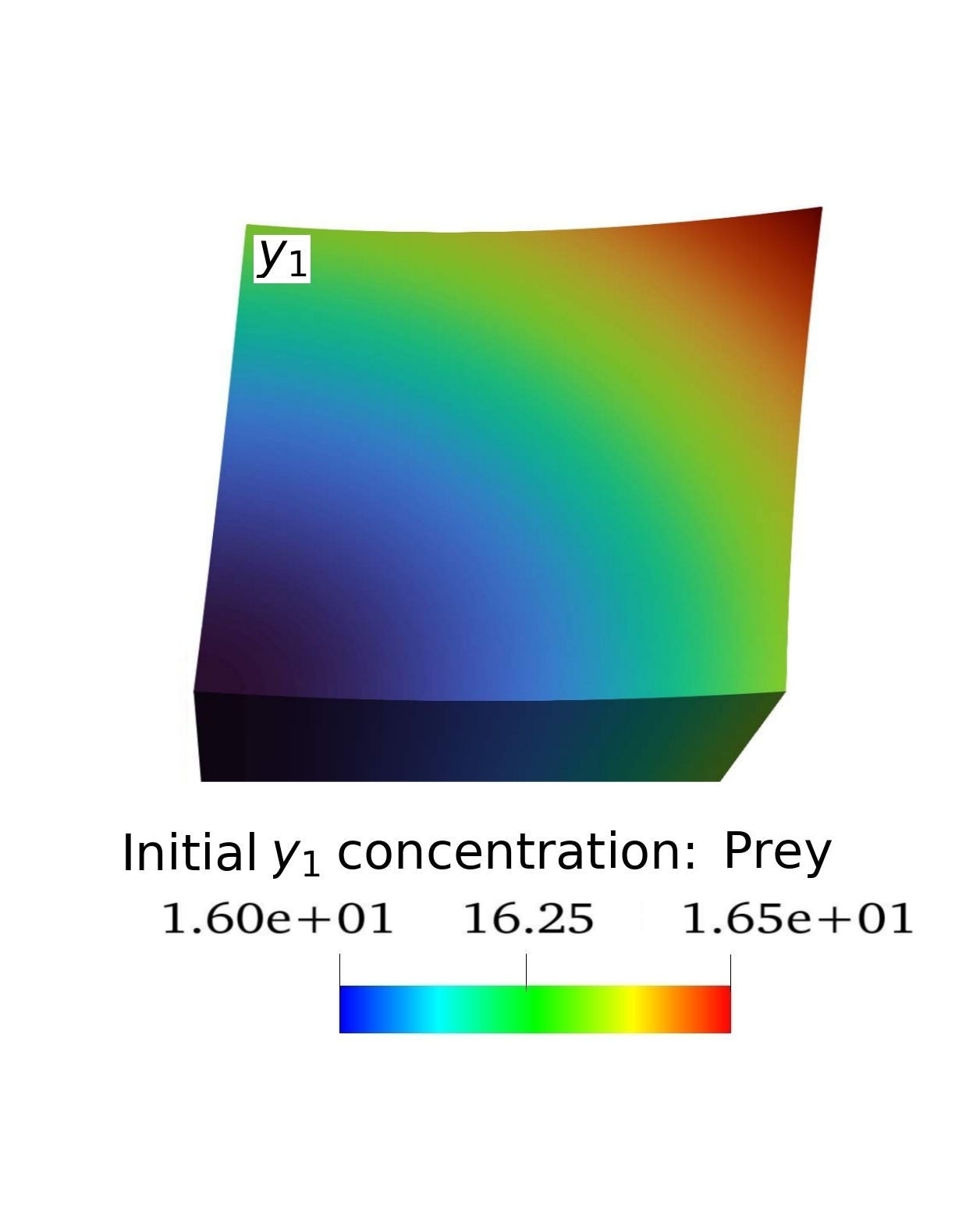}~
%\includegraphics[width=0.28\textwidth,trim={4cm 5cm 1.5cm 7cm},clip]{imgs/prey0.jpg}
   % \captionsetup{width=0.95\textwidth}
    \caption{\magent{The $y_1(0,x_1,x_2)$ smooth initial data, while $y_2(0,x_1,x_2)=25 $ everywhere in $\Omega$.}}\label{Figure_Lotka_smooth}
   % \label{S_continouity2}
\end{figure}
$
 y_1(0, x_1, x_2) =  {y_{1,0}} = 16+0.25(x_1^2+x_2^2)%{\text{ in }} \Omega
\, \text{ and }\, y_2(0,x) = {y_{2,0}} = 25 %-\sin(\pi x_1)\sin(\pi x_2) 
{\text{ in }} \Omega,
$
see  Figure \ref{Figure_Lotka_smooth} for the prey initial concentration visualization, as well as, for some cases the nonsmooth/discontinuous initial conditions
%
%\begin{eqnarray}\label{1.2}
%\nonumber
%\hskip-15pt  & &\frac{\partial{y_1}}{\partial t}  = \epsilon_1 \Delta y_1 +  (a-b y_2)y_1 {\text{ in }}(0,T] \times \Omega, \quad
 %y_1+\frac{\epsilon_1}{\lambda_2}\frac{\partial y_1 }{\partial {\mathbf n}}= g_1 {\text{ on }}(0,T] \times \Gamma, \quad
%\\%\nonumber
%\hskip-15pt  & &\frac{\partial{y_2}}{\partial t}  = \epsilon_2 \Delta y_2 + (cy_1-d)y_2 {\text{ in }}(0,T] \times \Omega, \quad
 %y_2+\frac{\epsilon_1}{\lambda_2}\frac{\partial y_2 }{\partial {\mathbf n}}= g_2 {\text{ on }}(0,T] \times \Gamma, \\\nonumber
%& &\quad\qquad\quad\qquad\quad\qquad y_1(0,x)=  {y_{1,0}}%{\text{ in }} \Omega
%\quad y_2(0,x)=  {y_{2,0}} {\text{ in }} \Omega .
%\end{eqnarray}
%
%\begin{figure}[ht]%[scale=0.65]
%\includegraphics[scale=0.55] {imgs/predator0_smooth5.jpg} 
%\caption{.}
%\label{mesh_geometry}
%\end{figure}
%
\begin{figure}
%\scalebox{0.12}
   \centering
\includegraphics[width=0.22\textwidth,trim={4cm 6cm 1.5cm 7cm},clip,=]{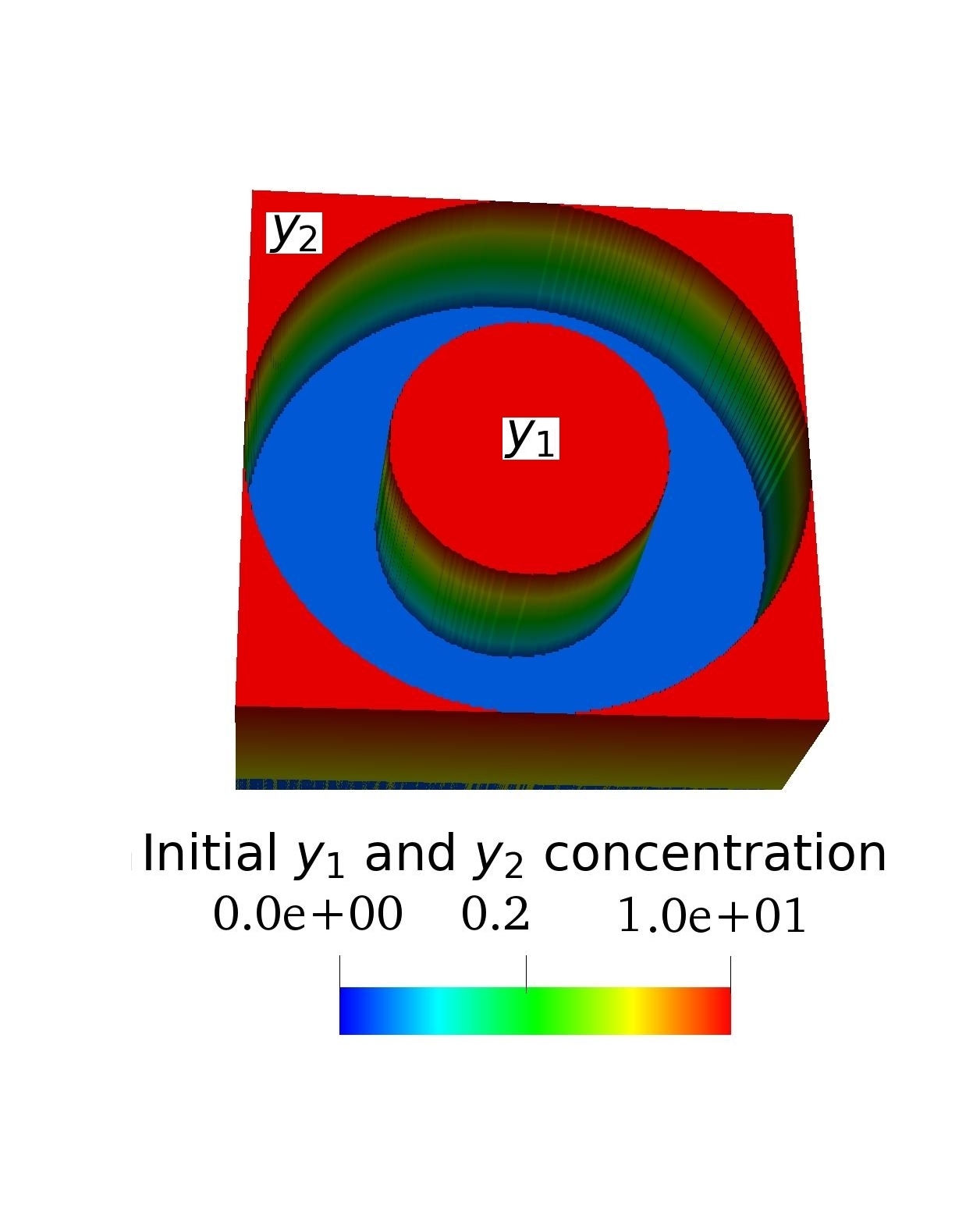}~
\includegraphics[width=0.22\textwidth,trim={4cm 6cm 1.5cm  7cm},clip]{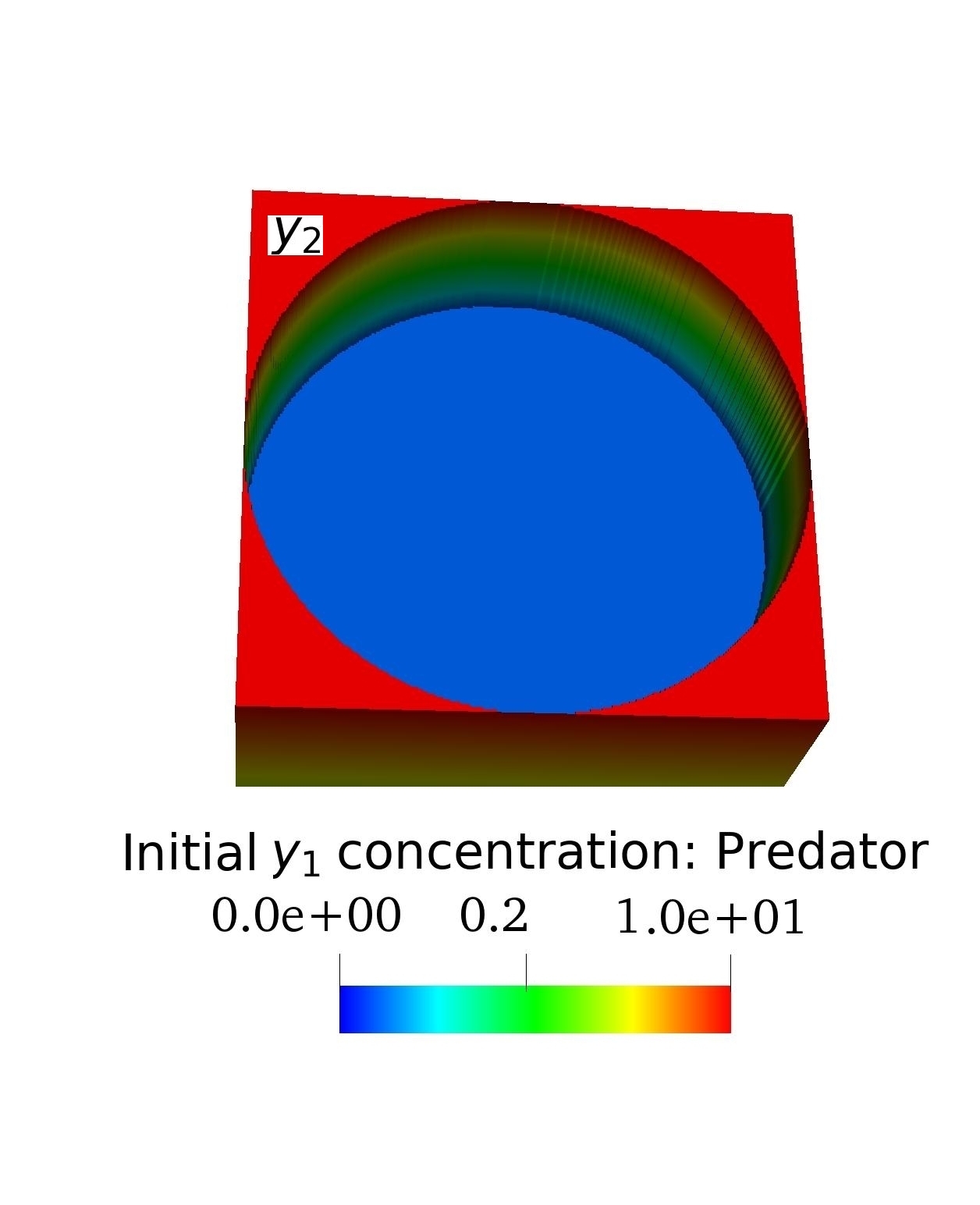}~
\includegraphics[width=0.22\textwidth,trim={4cm 6cm 1.5cm 7cm},clip]{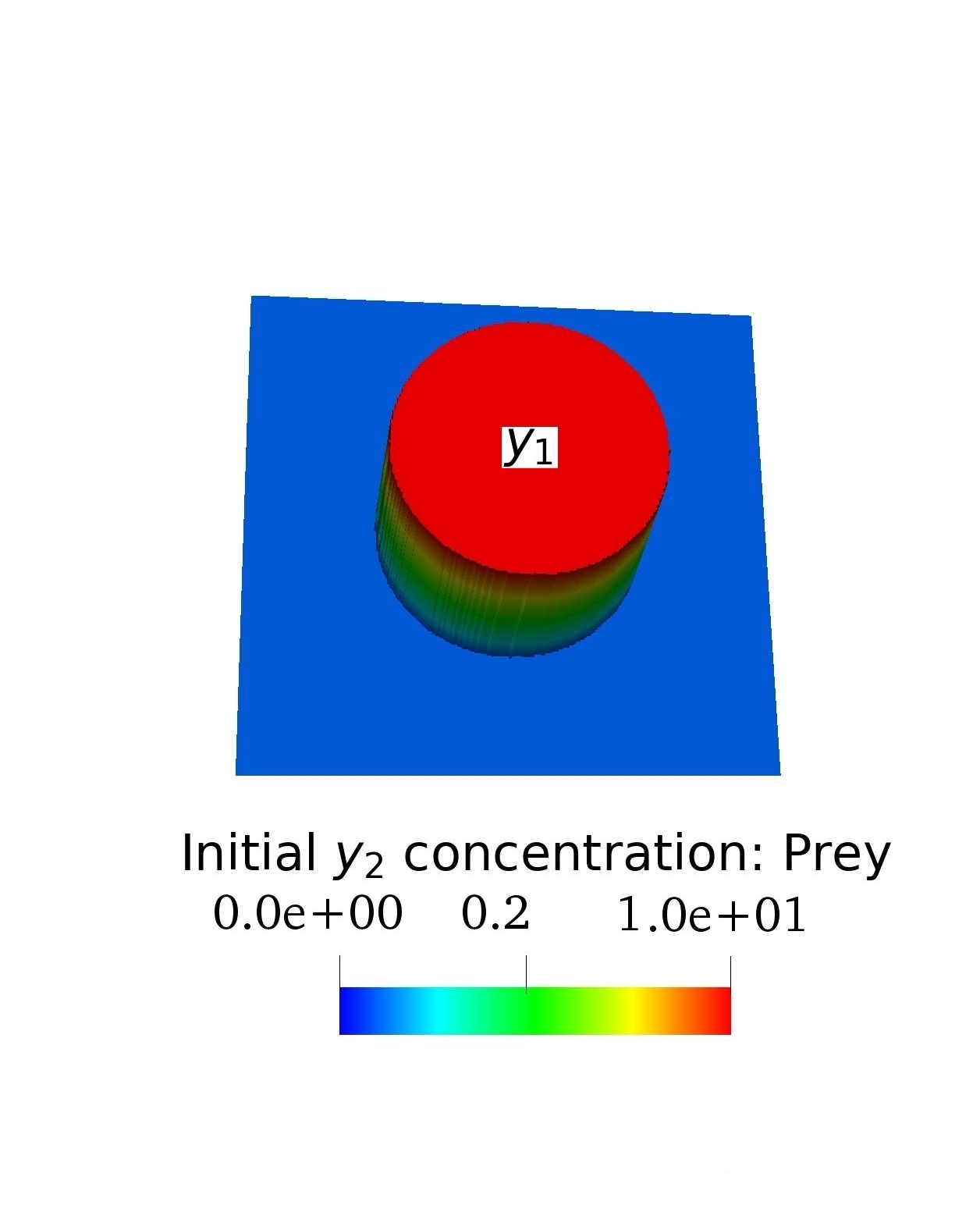}
% \captionsetup{width=0.95\textwidth}
\caption{\magent{The $y_1(0,x_1,x_2)$ and $y_2(0,x_1,x_2)$ rough initial data.}}\label{Figure_Lotka}
   \label{S_continouity2}
\end{figure}
%
%Deriving the appropriate discrete system, see Lemma  \ref{lem:3.4Lo}, we apply the discontinuous initial data
\begin{eqnarray}\label{initialData:nonsmouth}
\begin{array}{l}
y_{1,0} = \left\{ {\begin{array}{*{20}{c}}
{10,\quad{\rm{ if }}{{(x_1 - 0.5)}^2} + {{(x_2 - 0.5)}^2} \le \frac{1}{16}}\\
1, \quad {\rm{otherwise}}
\end{array}}\right.,%\\
%\end{array}
%\end{eqnarray}
%\begin{eqnarray}\begin{array}{l}
\,\, y_{2,0} = \left\{ {\begin{array}{*{20}{c}}
{10,\quad{\rm{ if }}{{(x_1 - 0.5)}^2} + {{(x_2 - 0.5)}^2} \ge \frac{1}{4}}\\
1, \quad {\rm{otherwise}}
\end{array}} \right.
\end{array}
\end{eqnarray}
as seen in Figure \ref{S_continouity2}, 
\magent{ and with right hand forces
\begin{eqnarray*}
%&& fromkTABLE 3
%\\
&&f_1=-(b\cos(t)\sin(t)\sin(\pi x)y^2sin(\pi y)+b \cos(t)\sin(t)x^2sin(\pi x) sin(\pi y)+64 b \sin(t)sin(\pi x) sin(\pi y)
\\
&&\qquad\, +\sin(t)y^2 -25 b\cos(t)y^2+64 b \sin(t)sin(\pi x) sin(\pi y)+\sin(t)y^2-25 b\cos(t)y^2 + a \cos(t) y^2
\\&&\qquad\,+\sin(t) x^2-25 b \cos(t) x^2+a \cos(t) x^2+4 \epsilon_1 \cos(t)-1600 b+64 a)/4,
\\
&&f_2=(c \cos(t) \sin(t) \sin(\pi x) y^2 \sin(\pi y)+c \cos(t) \sin(t) x^2 \sin(\pi x) \sin(\pi y) +64 c \sin(t) \sin(\pi x) \sin(\pi y)
\\
&&\qquad -8 \pi^2 \epsilon_1 \sin(t)  
\sin(\pi x) \sin(\pi y)-4 d \sin(t)\sin(pi*x)\sin(\pi y)-4 \cos(t)\sin(\pi x) \sin(\pi y)-25 c \cos(t) y^2
\\&&\qquad\,-25 c \cos(t) x^2-1600 c+100 d)/4.
%\\
%&&U1=0;//5;//2.5;//5;//2.;//8.;//3.;//2.\\
%&&U2=20;//10;//2.5;//10;//2;//13.;//3.;//2.;
%\\
%&& gg1= 0 ;
%\\
%&&gg2= 0 ;                             
%\\
%&& u1exact= 16+0.25*cos(t)*(y^2+x^2)  ;
%\\
%&& u2exact= 25-     sin(t)*sin(pi*x)*sin(pi*y);
\end{eqnarray*}}
%, for the prey and predator initial concentration visualization.
Further, we enforce the target functions $(y_{1,d},y_{2,d})=(0,20)$ which physically means that we want the prey to vanish and the predator to dominate with a specific number. This makes sense e.g. in the case where the growth rate of $y_1$
is sluggish, %very slow,
the rate of $y_2$ killing $y_1$ is %very slow
high, and the death rate of $y_1$ is high %prolonged %is very slow
too with a sluggish death rate of %$y_1$
$y_2$. %the growth rate of $y_2$ by chances of killing $y_1$
We highlight that the uncontrolled system indicates that both species survive with such a cycle of survival as we will see later in the first snapshot of Figure \ref{phaseplanes}, see also  \cite[p.~160]{Ha97}, confirming the correctness of the target functions choice.
\newline
We assume the parameters $\epsilon_1=0.1$ (prey diffusion), $\epsilon_2=0.01$ (predator diffusion), $a=0.47$ (growth rate of prey), $b=0.024$ (searching efficiency/attack rate), $c=0.023$ (predator's attack rate and efficiency at turning food into offspring-conversion efficiency), $d=0.76$ (predator mortality rate). The diffusion parameter values signify that the prey diffuses faster than the predator. % thing natural for the uncontrolled system so that the prey tries to survive. % \TODO{and finally dominates}. %, while our problem evolves in the time interval $[0,T]=[0,0.1]$ in a square $\Omega=[0,1]\times[0,1]$. 
The aforementioned parameter values were obtained from a related ordinary differential equation model for the classic Hudson Bay Company Hare-Lynx example.
%
 %We solve the linear parabolic equation over a period of time with given initial conditions for  $y_{1,0}$ and $y_{2,0}$  driving the state values of $y_1$ and $y_2$ at time $t$ to the desired target functions $y_{1,d}$, $y_{2,d}$ respectively.
 \newline
The predator or the prey enters or exits in the domain $\Omega$ or the boundary $\Gamma$
in a controlled way, indicated by the distributed control or Robin boundary control  functions $g_i$, $i=1,2$. Especially for the nonsmooth initial data case, the predator density at the starting point is concentrated near the corners of the domain, while the prey density is intensive around the center, see Figure \ref{S_continouity2}. We note that the initial and boundary conditions may be difficult to pinpoint. % for problems like this. %In general, initial and boundary conditions may be difficult to pin down for problems like this.  
Finally, we report that a  nonlinear gradient method with strong Wolfe-Powel conditions and Fletcher-Reeves directions is employed for determining the optimal control functions. % , with results reported in Table \ref{Exp_tau=h^2Lo,k=0}.
\magent{The values of $\gamma_1$, $\gamma_2$ have been tested in the interval $(10^{-1}, 10^{-5})$, and best selection approved to be $\gamma_1=\gamma_2=0.01$, see also \cite[p.~218%:Fig.4.5
]{Ka15} and references therein.}%\cite{GarcHiKah19,Ka15[page 218 (Fig 4.5)].}}
%\cite[p.~150]{kopka2003guide}
%
\subsection{Distributed control case.}
This paragraph is dedicated to validating numerically that even if we force strict control constraints employing low order time polynomials --higher-order schemes are not applicable due to the irregularity that is caused by the control constraints-- either unconstrained control for low and higher-order time polynomials, the results and the solutions' distance from the target functions become efficiently small. In particular, %we use the same data as in \TODO{paragraph \ref{R1}}, but 
we consider distributed control with and without control constraints, in cases with low and higher order time dG schemes respectively, and zero Dirichlet boundary conditions. Additionally, we report a qualitative study and the nullclines of the examined system. % as well as nonsmooth initial data cases.
%
%\TODO{In this example we demonstrate the case of some  discontinuous $y_{1}(0,x_1,x_2)$, $y_{2}(0,x_1,x_2)$ initial conditions as it is indicated in Figure \ref{Figure_Lotka}.
%\begin{figure}
%\scalebox{0.12}
  %  \centering
    %    \includegraphics[width=0.95\textwidth]{Lotka_y0.jpg}
   % \captionsetup{width=0.95\textwidth}
   % \caption{The $y_1(0,x_1,x_2)$ and $y_2(0,x_1,x_2)$ rough initial data.}\label{Figure_Lotka}
   % \label{S_continouity2}
%\end{figure}
%
 %For the fully discrete scheme, we use constant in time polynomials combined with linear in space. This is an appropriate choice in cases with nonsmooth data, see e.g. \cite{ChKa14}.

 %the control constraints
%$$
%{g_i}_a \leq g_i(t,x) \leq {g_i}_b \text{ for a.e.} (t,x) \in (0,T) \times \Omega \text{,  where } {g_i}_a,{g_i}_b \in {\mathbb R}, i=1,2 ,$$
%and the state constraints $$y_1(t,x), y_2(t,x)\ge 0 \quad\text{in} (0,T] \times \Omega.$$
%
\subsubsection{Experiment A: control constraints $g_{ia}$, $g_{ib}$ and $k=0$, $l=1$.} We employ constant polynomials in time and linear in space. It is known that higher-order schemes are not applicable due to the irregularity that the control constraints cause to our system. We apply the control constraints $g_{ia}=0$ and $g_{ib}=0.1$. This choice is very strict considering that using the same data but with no control constraints we noticed that the control takes values in the interval $(0.9, 3.55)$. The related results are demonstrated in Table \ref{tau=h^2LoConstrained_g} showing clearly that the state variables $y_1$, $y_2$ are driven efficiently to the desired targets $y_{1,d}$, $y_{2,d}$.
\begin{table}
\caption{Distributed control: Distance from the target  for the 2d solution with $k=0$, $l=1$ ($\tau=\mathcal{O} (h^2) $) with smooth data and control constraints.}
\begin{center}
%\scalebox{0.82}
{\begin{tabular}{|c|c|c|%c|c|
c|}
   \hline
   \multicolumn{1}{|c|}{Discretization} & \multicolumn{2%5
    }{|c|}{Distance from the target}&\multicolumn{1}{|c|}{Functional}    \\ \hline
     $\tau=h^2/8$  &${\left\| y_1-{y_{1,d}} \right\|}_{{L^2[0,T;{ L}^2 (\Omega )] }}$
     &$\left\| y_2-{y_{2,d}} \right\|_{{{L^2[0,T;{ L}^2  (\Omega )] }}}$
     &$J(y,g)$%&  { }      &  { }
     \\ \hline
    $h=0.23570$ &   {0.05672514712}&    {0.031267953860}&    {0.08799310282}%&  { }      &{ }
     \\ \hline
     %    $h=0.23570$ &{}&    { }&{ }%&  { }      &{ }
     %\\ \hline
    $h=0.13888$ &   {0.01417057646}&    {0.004847915446}&    {0.01901849221}%&  { }      &{ }
    \\ \hline
    %    $h=0.13888$&{0.003991792244}&    {0.006977773302	}&    {0.01431706681}%&  { }      &{ }
    %\\ \hline
    $h=0.05892$ &   {0.00354747615}&    {0.000908289532}&    {0.00445576574}%&  { }      &{ }
    \\ \hline
    $h=0.02946$ &   {0.00088619624}&    {0.000204350006}&    {0.00109054626}%&  { }      &{ }
   \\ \hline
    %   $h= $ &   { }&    { }&    { }%&  { }      &{ }
   %\\ \hline
 \end{tabular}}
\end{center}
\label{tau=h^2LoConstrained_g}
\normalsize
\end{table}
It is worth noting that we take similar results if we implement the code for the same data but with no control constraints and using $k=0$,  and $l=1$ polynomial degrees in time and in space respectively, although we omit these results for brevity. 
\subsubsection{Experiment B: unconstrained control and $k=1$,  $l=2$.}
In this experiment, in Table \ref{tau=h^2_g_uncon_k=1_l=2} the results for linear in time and quadratic in space polynomials have been reported.
\magent{We note that the
comparison of Table 1 ($k=0$, $l=1$) and Table 2 ($k=1$, $l=2$), indicate
%Table 1 
%
% {0.00088619624}-{0.000204350006}-{0.00109054626}
%
%Table 2 k=1 l=2
%
% {0.00087479208}-{0.000195581819}-{0.001070397159}
%
%Experiment E
%
%========
%
%k=0 l=1
%
%    $h=0.02946$    {0.00031866772} {0.00648286241}    {0.0068040695}
%    
%k=1 l=2
%
%    $h=0.02946$   {0.0045402412} {0.00056394557} {0.005106144}
%
%======
%
 that it is not easy to further reduce significantly the functional even if we increase the order of polynomials in time and/or in space. 
%\newline    
We highlight that these difficulties appear in our opinion since the specific system consideration has to be controlled through both state variables $y_1$ and $y_2$, and the integration errors are accumulating especially in the cases of  higher order dG in time,  since e.g. for $k=1$ where the already quadruple optimal control system actually becomes eightfold. In our opinion, it is important though that the minimization functional is reduced even slightly. Similar behavior has been also depicted  in the next Robin boundary related Experiment E
 after comparing Tables 3 ($k=0$, $l=1$), Table 4 ($k=1$, $l=1$) and Table 5 ($k=1$, $l=2$) in paqragraph 5.2.2, see also \cite{ChKa14}.
 %
%Such output was reported as well as
%in and it has been verified in the work, 
%in [14]%
%\cite{ChryKaRobin}.
}
%
%, noting the fact that the results are \magent{slightly} improved QAAAAAAAAAAAAAAAAAAAAAAAAAAaaa
%much  improved 
%due to the higher order polynomial used.
\begin{table}
\caption{Distributed control: Distance from the target  for the 2d solution with $k=1$, $l=2$ ($\tau=\mathcal{O} (h^2) $) with smooth data and without control constraints.}
\begin{center}
%\scalebox{0.82}
{\begin{tabular}{|c|c|c|%c|c|
c|}
   \hline
   \multicolumn{1}{|c|}{Discretization} & \multicolumn{2%5
    }{|c|}{Distance from the target}&\multicolumn{1}{|c|}{Functional}    \\ \hline
     $\tau=h^2/8$  &${\left\| y_1-{y_{1,d}} \right\|}_{{L^2[0,T;{ L}^2 (\Omega )] }}$
     &$\left\| y_2-{y_{2,d}} \right\|_{{{L^2[0,T;{ L}^2  (\Omega )] }}}$
     &$J(y,g)$%&  { }      &  { }
     \\ \hline
    $h=0.23570$ &   {0.05914777260}&    {0.017063455880}&    {0.076212590230}%&  { }      &{ }
     \\ \hline
     %    $h=0.23570$ &{}&    { }&{ }%&  { }      &{ }
     %\\ \hline
    $h=0.13888$ &   {0.01419612697}&    {0.003323664495}&    {0.017520155410}%&  { }      &{ }
    \\ \hline
    %    $h=0.13888$&{0.003991792244}&    {0.006977773302	}&    {0.01431706681}%&  { }      &{ }
    %\\ \hline
    $h=0.05892$ &   {0.00351236326}&    {0.000785876982}&    {0.004298333324}%&  { }      &{ }
    \\ \hline
    $h=0.02946$ &   {0.00087479208}&    {0.000195581819}&    {0.001070397159}%&  { }      &{ }
   \\ \hline
   %    $h= $ &   { }&    { }&    { }%&  { }      &{ }
   %\\ \hline
 \end{tabular}}
\end{center}
\label{tau=h^2_g_uncon_k=1_l=2}
\normalsize
\end{table}
%We notice that using higher order polynomial degrees, the quadratic functional is minimized more effectively comparing it with the case of low order polynomial degrees, see Table \ref{tau=h^2LoConstrained_g} and Table \ref{tau=h^2_g_uncon_k=1_l=2}. 
Concerning the distance between the state solution and the target function, we also notice that the predator concentration $y_2$ \red{is driven} closer to the target $y_{2,d}$ comparing it with the distance between $y_1$ and the target function $y_{1,d}$ and similar phenomenon we noticed for Experiment A.  Also, It is observed that the lower-density starting meshes have good results for the $y_1$ state. This is natural since the prey diffusion $\epsilon_1 = 0.1$ is much bigger than the predator diffusion $\epsilon_2 = 0.01$. In order to achieve % and whether someone was seeking 
better results, an idea could be to use a multi-scaling approach,  which would be investigated in future work. 
\subsubsection{Experiment C: Qualitative study and nullclines investigation.}
For a deeper understanding of the solutions' behavior, we demonstrate the system's nullclines and the phase planes, which also exhibit \magent{the way}  the control functions affect the state solutions. Furthermore, the latter indicates the appropriate choice of parameters for the numerical simulation of optimal solutions. It is informative we add the phase planes of the system to a graph and essential to find the confined set and draw the nullclines of the system, that is, the curves in the phase plane (the curves on which one of the variables is stationary), see e.g. \cite[p.~46-47,78]{Mu03} or \cite{GaTre07}. In the following, we demonstrate a 4-step study to find the fixed points, the stability of the fixed points, the nodes, spirals, vortices, or possible saddles. Particularly, we name the nonlinear terms
$$\phi_1 (y_1,y_2)=(a-by_2)y_1$$
 $$\phi_2 (y_1,y_2)=(cy_1-d)y_2,$$
 and we follow the procedure as described in the next steps: 

STEP 1: We solve the equations $\phi_1 (y_1,y_2)=0$ and $\phi_2 (y_1,y_2)=0$ and we see that the fixed points are $FP1=(0,0)$ and $FP2=(d/c,a/b)$.

STEP 2: For the fixed points stability,  we examine the matrix
 $$\Phi(y_1,y_2)=
\left[
  \begin{array}{cc}
    \frac{\partial \phi _1}{\partial y_1}  & \frac{\partial \phi _1}{\partial y_2} \\
    \frac{\partial \phi _2}{\partial y_1}  & \frac{\partial \phi _2}{\partial y_2}  \\
  \end{array}
\right]
=
\left[
  \begin{array}{cc}
    a-by_2  & -by_1 \\
    cy_2  &cy_1-d \\
  \end{array}
\right],
$$
with determinant ${\it\Delta}=(c y_1-d)(a-b y_2)+b c y_1 y_2$
and trace $T= -by_2+cy_1-d+a$.

For  $a=0.47$, $ b =0.024 $,
$c =0.023 $, $d =0.76 $,   we take for the $FP1$
$$\Phi(0,0)=
\left[
  \begin{array}{cc}
    0.47  & 0 \\
    0  &-0.76 \\
  \end{array}
\right],
$$
and for the $FP2$ 
$$\Phi(d/c,a/b)=\Phi(33.043,19.583)=
\left[
  \begin{array}{cc}
    0.0  & -0.793 \\
    0.450  &-1.110\times10^-16 \\
  \end{array}
\right].
$$
It is easy to see that the $FP2$ is a stable fixed point since
 the determinant of $\Phi(d/c,a/b)$ is positive and the trace of $\Phi(d/c,a/b)$ is negative, but  $FP1$ is unstable.

 STEP 3:  In this step, we study the nature of the fixed points. Let $\it \Delta$ be the determinant of $\Phi$ and
$T$ be the trace. The discriminant is
\begin{eqnarray*}
D &=& T^2 - 4{\it\Delta}\\
&=& b^2y_2^2+(-2bcy_1-2bd-2ab)y_2+c^2y_1^2\\
 & &+(-2cd-2ac)y_1+d^2+2ad+a^2.
\end{eqnarray*}
After some computations, we notice that the $FP2$ is a stable spiral fixed point since we observe two complex eigenvalues with negative real parts and occurs when $D < 0$. Solutions spiral into  that fixed
point are depicted in Figure \ref{phaseplanes}.
Moreover, the $FP1$ is a saddle point since $\it \Delta<0$ and
 consists of one positive and one negative eigenvalue thus it is unstable.

Hence, in Figure  \ref{phaseplanes},
%field_degree=2 causes the field to be composed of quadratic splines,
%based on the first and second derivatives at each grid point.
%soln_at() and solns_at() draw solution curves passing through the
%specified points, using a slightly enhanced 4th-order Runge Kutta
%numerical integrator.
%\begin{figure}
%%\scalebox{0.9}%0.12}
%    \centering
   %     \includegraphics[width=0.95\textwidth]{phaseplanes}%null_clines.jpg}
%\\
%%\includegraphics[width=0.5\textwidth]{1.pdf}
%%\includegraphics[width=0.5\textwidth]{2.pdf}
%%\\
%%\includegraphics[width=0.5\textwidth]{3.pdf}
%%\includegraphics[width=0.5\textwidth]{4.pdf}
%%\\
%%\includegraphics[width=0.5\textwidth]{5.pdf}
%%\includegraphics[width=0.5\textwidth]{6.pdf}
%%\\
%%\includegraphics[width=0.5\textwidth]{%68.pdf}
%%\includegraphics[width=0.5\textwidth]{2.pdf}
  % % \captionsetup{width=0.95\textwidth}
    %\caption{Phaseplanes for the $y_1(t,x,y)$ and $y_2(t,x,y)$ applying control $(g_1,g_2)=(0,0),(0.1,0.1),(1,1),(1.5,1.5)%,(2,2),(3,3)
  %  ,(2.5,2.5),(4,4)$ and parameters $\epsilon_1=0.1$ (prey diffusion), $\epsilon_2=0.01$ (predator diffusion), $a=0.47$ (growth rate of prey), $b=0.024$ (searching efficiency/attack rate), $c=0.023$ (predator's attack rate and efficiency at turning food into offspring -conversion efficiency), $d=0.76$ (predator mortality rate).}\label{phaseplanes}
   % \label{S_continouity2}
%\end{figure}
\begin{figure}
%\scalebox{0.9}%0.12}
    %\centering
      %  \includegraphics[width=0.95\textwidth]{phaseplanes}%null_clines.jpg}
%\\
\includegraphics[trim={0cm 0 1.15cm 0},clip,width=0.52\textwidth
]{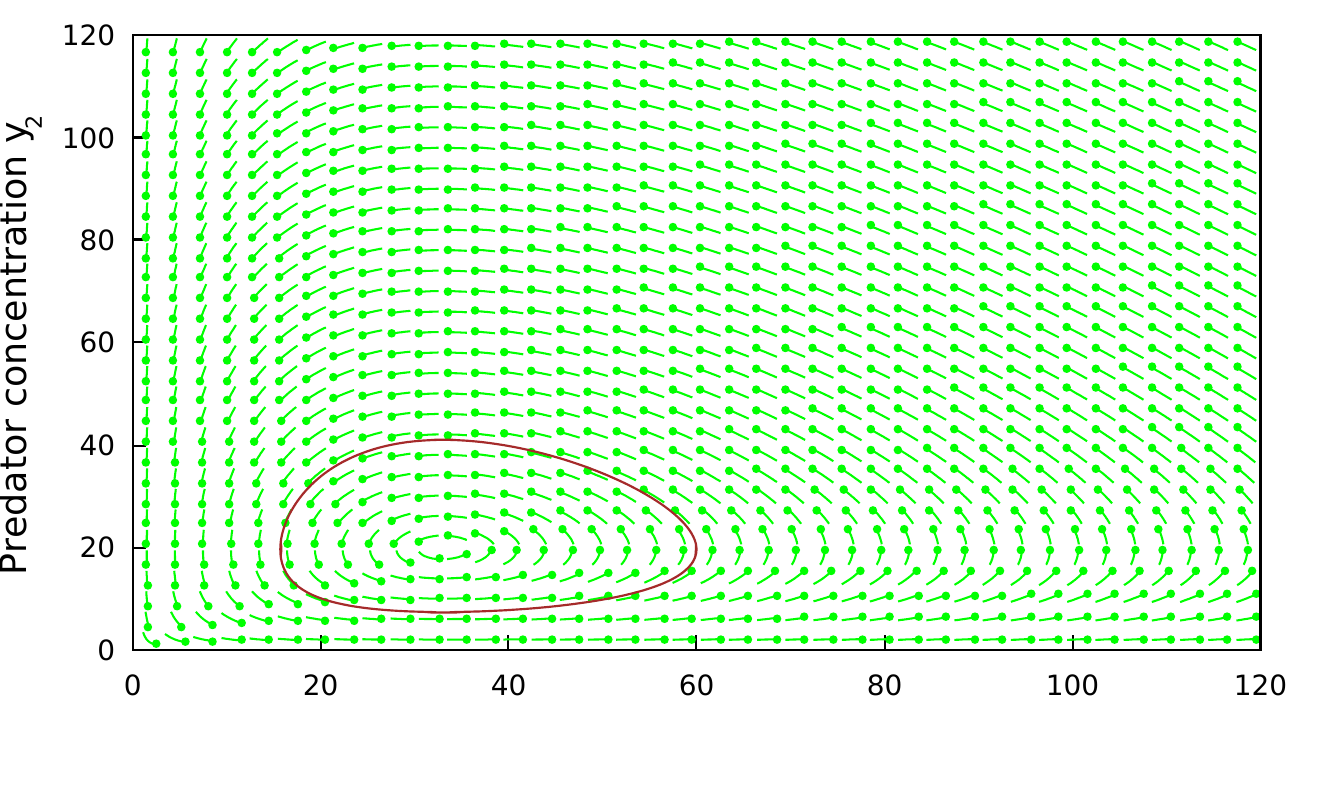}~\includegraphics[trim={0cm 0 1.15cm 0},clip,width=0.52\textwidth]{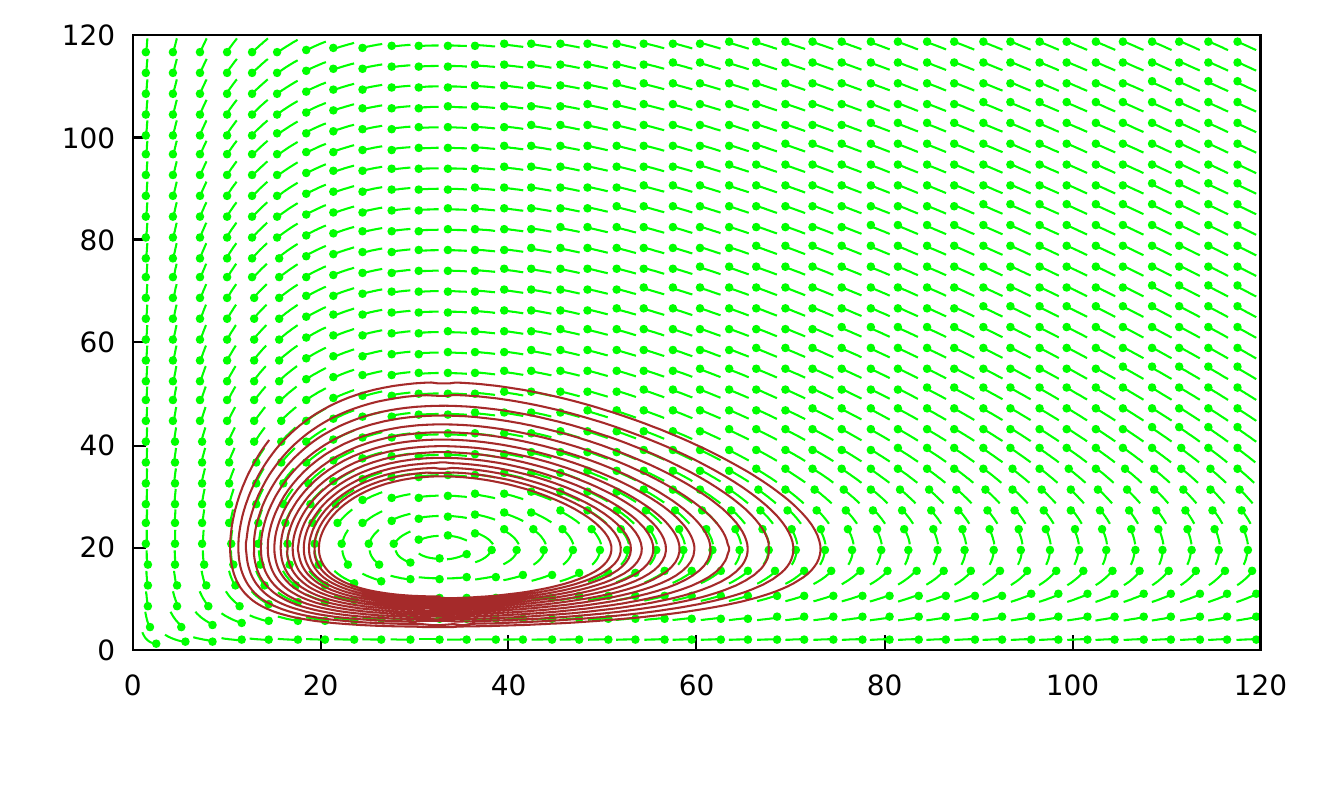}
\\
\includegraphics[trim={0cm 0 1.15cm 0},clip,width=0.52\textwidth]{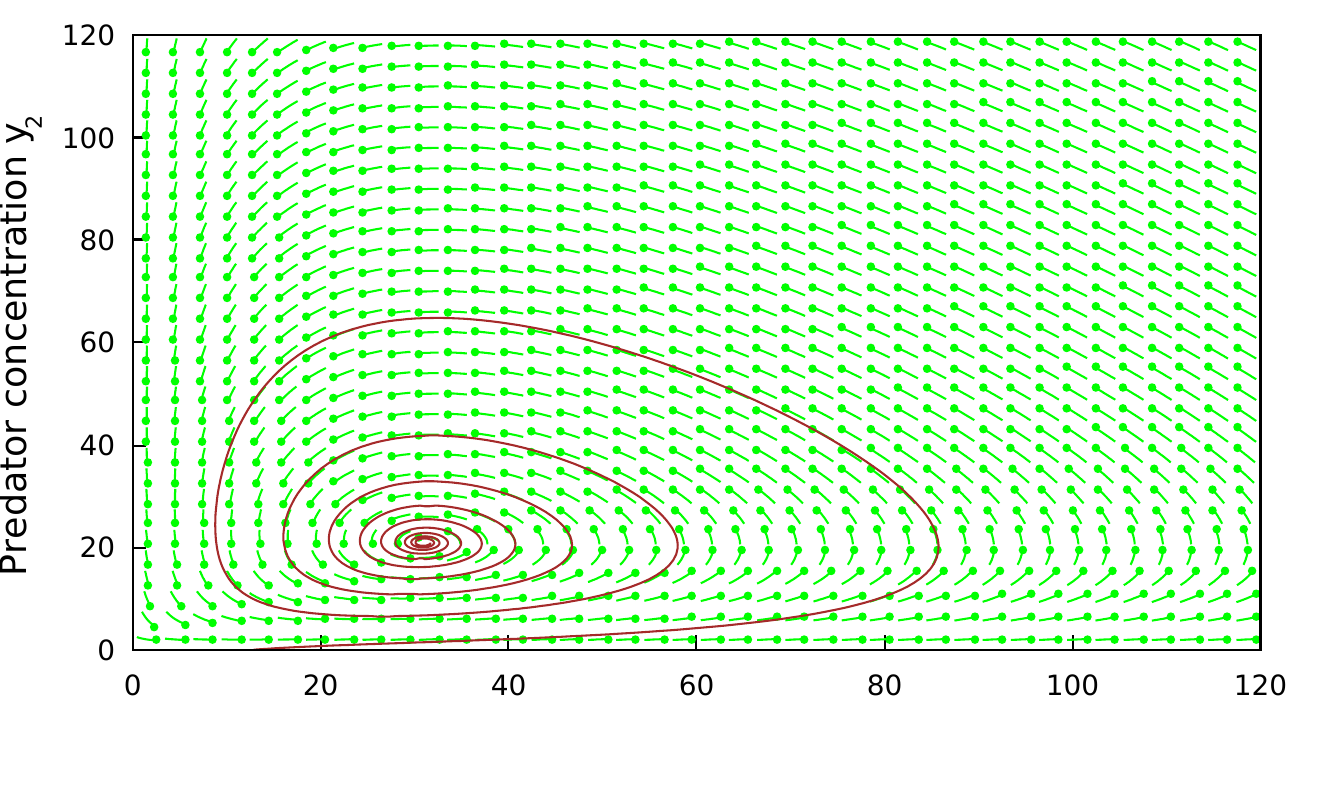}~\includegraphics[trim={0cm 0 1.15cm 0},clip,width=0.52\textwidth]{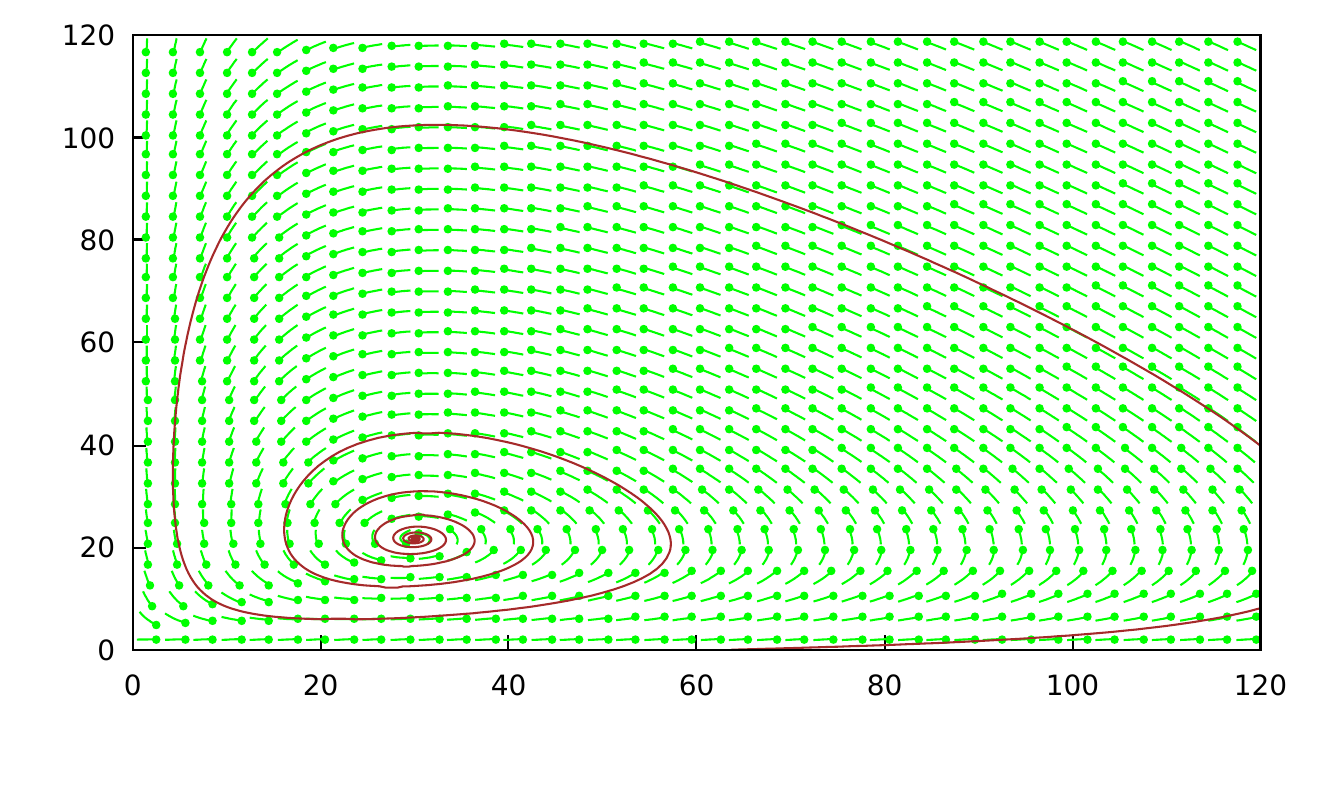}~
%\includegraphics[width=0.55\textwidth]{3.pdf}
%\\
%\includegraphics[width=0.55\textwidth]{5.pdf}~
%\includegraphics[width=0.55\textwidth]{6.pdf}
%\\
%\includegraphics[width=0.55\textwidth]{68.pdf}~
%\includegraphics[width=0.55\textwidth]{69.pdf}
\\
\includegraphics[trim={0cm 0 1.15cm 0},clip,width=0.52\textwidth]{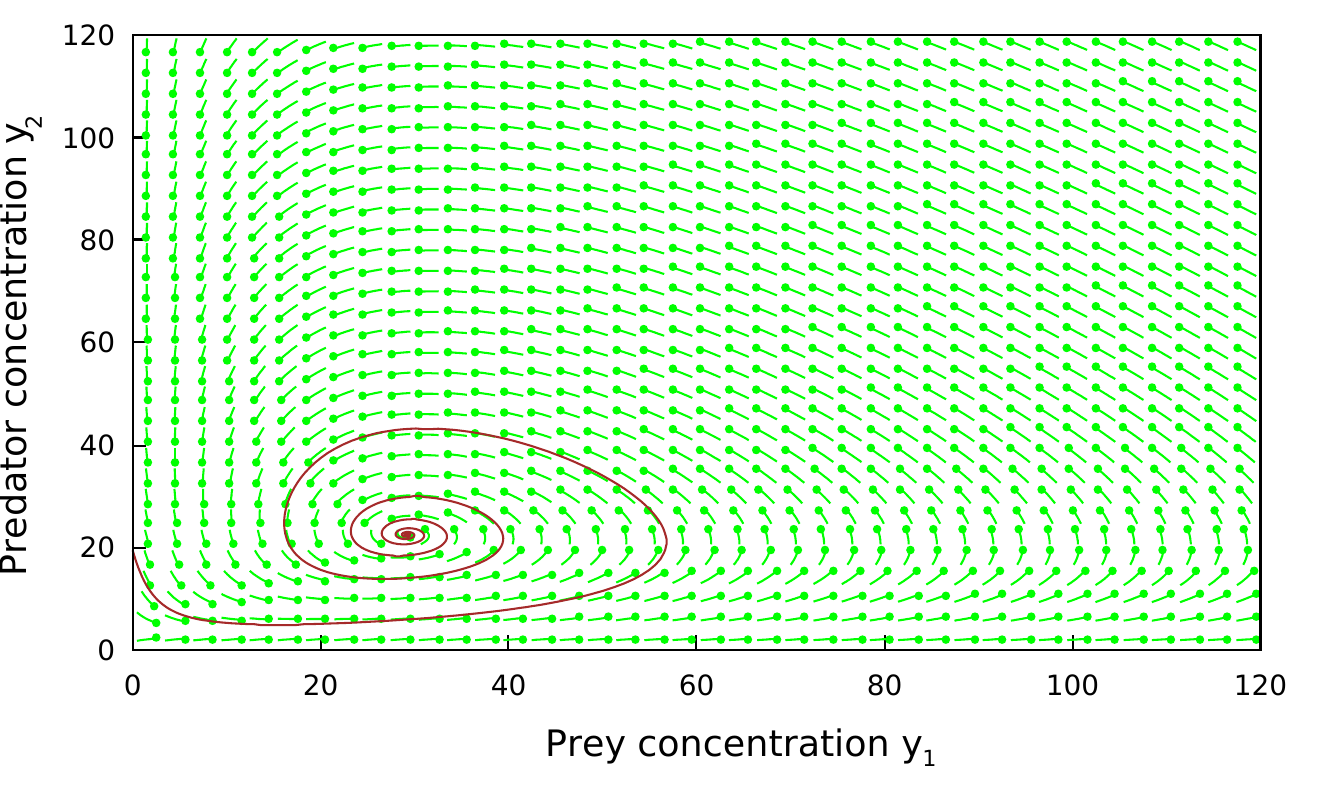}~
\includegraphics[trim={0cm 0 1.15cm 0},clip,width=0.52\textwidth]{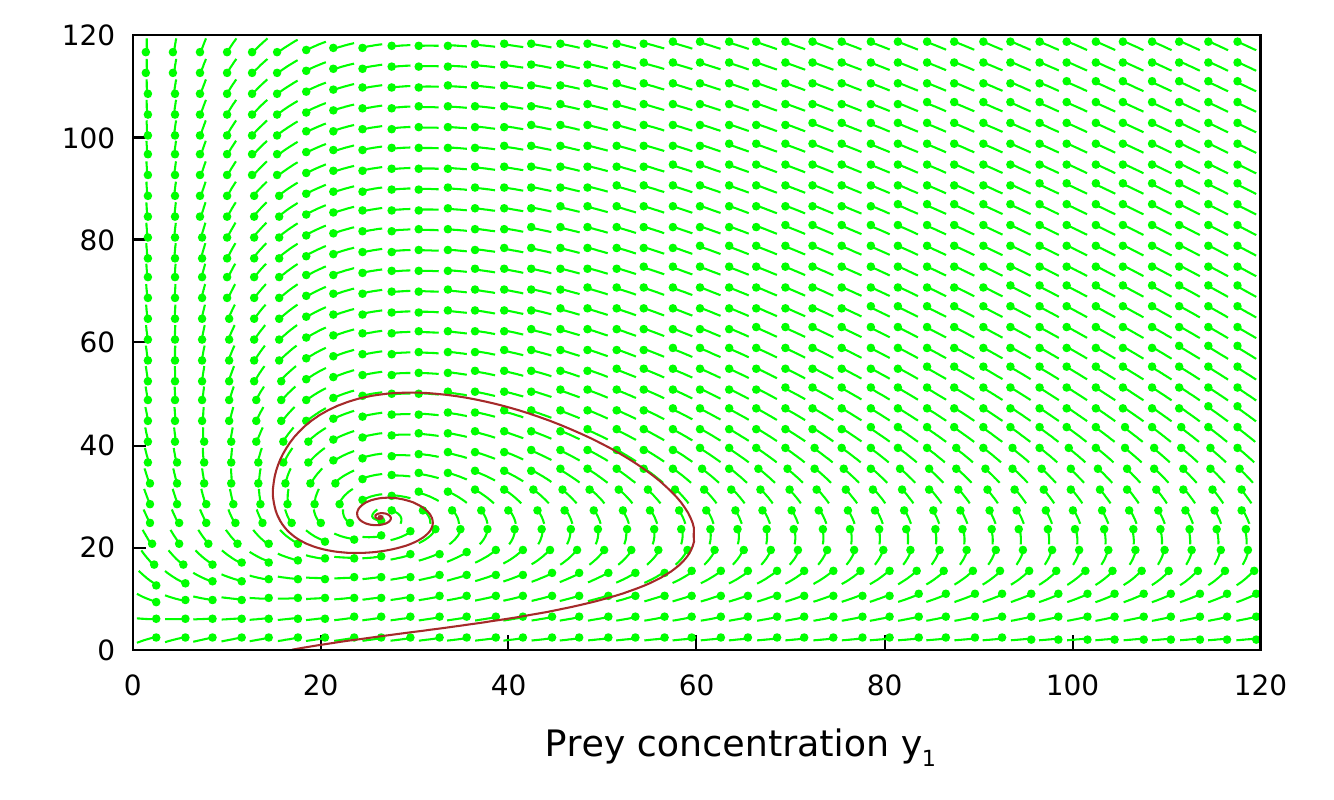}
%\includegraphics[width=0.5\textwidth]{2.pdf}
   % \captionsetup{width=0.95\textwidth}
    \caption{Phaseplanes for the $y_1(t,x,y)$ and $y_2(t,x,y)$ applying control $(g_1,g_2)=(0,0)$, $(0.1,0.1)$, $(1,1)$, $(1.5,1.5)%,(2,2),(3,3)
    $, $(2.5,2.5)$, $(4,4)$ and for parameters values: prey diffusion $\epsilon_1=0.1$, predator diffusion $\epsilon_2=0.01$, growth rate of prey $a=0.47$, searching efficiency/attack rate $b=0.024$, predator's attack rate and efficiency at turning food into offspring-conversion efficiency $c=0.023$, predator mortality rate $d=0.76$.}\label{phaseplanes}
   % \label{S_continouity2}
\end{figure}
we demonstrate the direction of the $(y_1, y_2)$ trajectory for 
 control variable example values, and we draw the dynamics of the local kinetics for the point $(y_1,y_2)= (16.125,24)$, associated with the initial data for
$(x_1, x_2)=%(0,0)     $, $
(\frac{1}{2},\frac{1}{2})$, %$(0,1)     $ and $(1,1)$.
It is worth noting that we had similar trajectories for the test points $(y_1,y_2)= (16,25)$, $(16.125,24)$, $(16.25,25)$ and $(16.5,25)$.

The first visualization in Figure \ref{phaseplanes} corresponds to the uncontrolled case and is in agreement also to that in \cite[p.~114]{Jo09}. Following the second to sixth visualization in the aforementioned figure, we demonstrate some examples of how the control affects the dynamics of the system.
 \begin{figure}
%\scalebox{0.9}%0.12}
    \centering
        \includegraphics[width=0.60\textwidth]{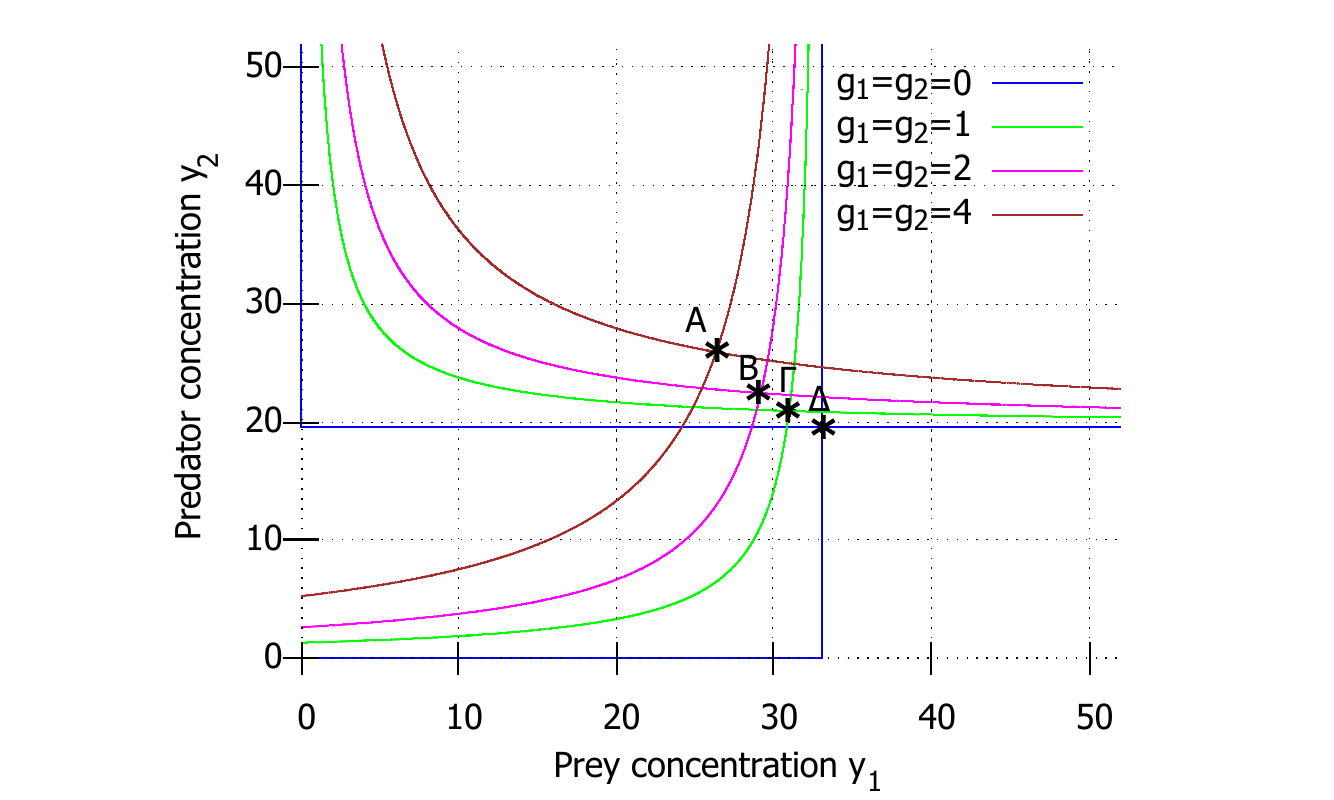}
   % \captionsetup{width=0.95\textwidth}
   \caption{Nullclines for the $y_1(t,x,y)$ and $y_2(t,x,y)$ applying controls $(g_1,g_2)=(0,0)$, $(1,1)$, $(2,2)$, $(4,4)$  and for parameters: prey diffusion $\epsilon_1=0.1$, predator diffusion $\epsilon_2=0.01$, the growth rate of prey $a=0.47$, searching efficiency/attack rate $b=0.024$, predator's attack rate and efficiency on turning food into offspring-conversion efficiency $c=0.023$, predator mortality rate $d=0.76$.}\label{null_cline_Zoom}
   % \label{S_continouity2}
\end{figure}
 Figure \ref{null_cline_Zoom} visualizes how the fixed point --an intersection of the nullclines for specific control values/intersection of lines with the same color-- varies in some cases where we apply control. It also depicts how the control function changes the system's dynamic. \magent{Moreover}, the FP2 fixed point, from the initial value (without control) $(33.043,19.583)$ changes \magent{positions}, namely to $(30.966,20.928)$, $(29.168,22.440)$, $(26.331,25.912)$ after applying some specific control values. Someone may also observe in the first column third snapshot, a typical example of how the control may drive the solution to the desired state $(0,20)$. % after some time.
\subsection{{{Boundary Robin control case.}}}\label{Experiments_Robin}
Related to this model problem,  we try to drive the solution to a desired target $({y_{1,d}},{y_{2,d}})$ minimizing the quadratic functional $J(y_1,y_2,g_1,g_2)$ described in (\ref{functional_full}) for $S=\Gamma$
%\begin{eqnarray}\label{B1.1_lo}
%   J(y_1,y_2,g_1,g_2) &=& \frac{1}
%{2}\int_0^T {\left\| y_1-{y_{1,d}}\right\|_{L^2 (\Omega )}^2  dt
%}
%+\frac{\gamma _1}{2}\int_0^T {\left\| g_1 \right\|_{L^2 (\Gamma )}^2 dt
%}
%\nonumber\\&+&\frac{1}
%{2}\int_0^T {\left\| {{y_2-{y_{2,d}}}}\right\|_{L^2 (\Omega )}^2 dt
%}
%+\frac{\gamma _2}{2}\int_0^T {\left\| g_2 \right\|_{L^2 (\Gamma )}^2dt
%}
%\end{eqnarray}
subject to the constrains (\ref{LoR1})-(\ref{LoR2}), and the initial conditions $
 y_1(0,x)=  {y_{1,0}}%{\text{ in }} \Omega
\quad y_2(0,x)=  {y_{2,0}} {\text{ in }} \Omega .
$
%\begin{eqnarray}\label{1.2}
%\nonumber
%\hskip-15pt  & &\frac{\partial{y_1}}{\partial t}  = \epsilon_1 \Delta y_1 +  (a-b y_2)y_1 {\text{ in }}(0,T] \times \Omega, \quad
 %y_1+\frac{\epsilon_1}{\lambda_2}\frac{\partial y_1 }{\partial {\mathbf n}}= g_1 {\text{ on }}(0,T] \times \Gamma, \quad
%\\%\nonumber
%\hskip-15pt  & &\frac{\partial{y_2}}{\partial t}  = \epsilon_2 \Delta y_2 + (cy_1-d)y_2 {\text{ in }}(0,T] \times \Omega, \quad
 %y_2+\frac{\epsilon_1}{\lambda_2}\frac{\partial y_2 }{\partial {\mathbf n}}= g_2 {\text{ on }}(0,T] \times \Gamma, \\\nonumber
%& &\quad\qquad\quad\qquad\quad\qquad y_1(0,x)=  {y_{1,0}}%{\text{ in }} \Omega
%\quad y_2(0,x)=  {y_{2,0}} {\text{ in }} \Omega .
%\end{eqnarray}
Again, the state  $y_1(t,x_1,x_2)$ represents the prey population density at time $t$ and $y_2(t,x_1,x_2)$ represents the
predator population density both at time $t$ and position $(x_1,x_2)$. 
In the case of the Robin boundary control of the Lotka-Volterra system, numerical tests indicate that it is more difficult to achieve reliable numerical experimental results compared with the zero Dirichlet boundary conditions and distributed control. Similar issues have been noticed in \cite{ChKa14, ChKa12}. 
%and the different behavior of each variable since they satisfy different equations which are also coupled
In our opinion, this is caused by the outwards normal derivative on the boundary,  and even worse whether we apply boundary control. Subsequently, this phenomenon in the boundary affects also the solutions inside the domain and depends on the different diffusion properties. 
For the above reason, we need to use special tools for solving the optimal control problem. We employ the non-linear conjugate gradient methods for better results and best convergence properties and the Fletcher-Reeves directions together with the strong Wolfe-Powel conditions. These techniques have advantages such as being fast and with strong global convergence properties,  \cite{Dai-Yuan99}, \cite{DuaZhangMa16}, but they also have some disadvantages that we will be concerned about in the following.

\subsubsection{Experiment D: Constant in time ($k=0$), linear in space polynomials ($l=1$) and rough initial data.}\label{R1}
%\section{The Lotka Volterra problem}
In this example, we demonstrate the test case of the discontinuous initial conditions $y_{1}(0,x_1,x_2)$, $y_{2}(0,x_1,x_2)$ as they have been visualized in Figure \ref{Figure_Lotka}.
%\begin{figure}
%\scalebox{0.12}
  %  \centering
    %    \includegraphics[width=0.95\textwidth]{Lotka_y0.jpg}
   % \captionsetup{width=0.95\textwidth}
   % \caption{The $y_1(0,x_1,x_2)$ and $y_2(0,x_1,x_2)$ rough initial data.}\label{Figure_Lotka}
   % \label{S_continouity2}
%\end{figure}
%
%\begin{figure}
%\scalebox{0.12}
 %   \centering
   %     \includegraphics[width=0.28\textwidth,trim={4cm 5cm 1.5cm 7cm},clip]{imgs/predator_prey0_.jpg}~
    %    %\includegraphics[width=0.35\textwidth]{imgs/prey0.jpg}
     %   \includegraphics[width=0.28\textwidth,trim={4cm 5cm 1.5cm  7cm},clip]{imgs/predator0.jpg}~
%\includegraphics[width=0.28\textwidth,trim={4cm 5cm 1.5cm 7cm},clip]{imgs/prey0.jpg}
 %  % \captionsetup{width=0.95\textwidth}
  %  \caption{The $y_1(0,x_1,x_2)$ and $y_2(0,x_1,x_2)$ rough initial data.}\label{Figure_Lotka}
  % % \label{S_continouity2}
%\end{figure}
%
For the fully discrete scheme, we use constant in time polynomials combined with linear in space. This is an appropriate choice in cases with nonsmooth data, see e.g. \cite{ChKa14}.
%
 %the control constraints
%$$
%{g_i}_a \leq g_i(t,x) \leq {g_i}_b \text{ for a.e.} (t,x) \in (0,T) \times \Omega \text{,  where } {g_i}_a,{g_i}_b \in {\mathbb R}, i=1,2 ,$$
%and the state constraints $$y_1(t,x), y_2(t,x)\ge 0 \quad\text{in} (0,T] \times \Omega.$$
%
%Deriving the appropriate discrete system, see Lemma  \ref{lem:3.4Lo}, 
\magent{We apply discontinuous initial data as in (\ref{initialData:nonsmouth}),}  
%\begin{eqnarray*}\begin{array}{l}
%y_{1,0} = \left\{ {\begin{array}{*{20}{c}}
%{10,\quad{\rm{ if }}{{(x_1 - 0.5)}^2} + {{(x_2 - 0.5)}^2} \le \frac{1}{16}}\\
%1, \quad {\rm{otherwise}}
%\end{array}}\right.,\,\,%\\
%%\end{array}
%%\end{eqnarray}
%%\begin{eqnarray}\begin{array}{l}
%y_{2,0} = \left\{ {\begin{array}{*{20}{c}}
%{10,\quad{\rm{ if }}{{(x_1 - 0.5)}^2} + {{(x_2 - 0.5)}^2} \ge \frac{1}{4}}\\
%1, \quad {\rm{otherwise}}
%\end{array}} \right.,
%\end{array}
%\end{eqnarray*}
 for  target functions $(y_{1,d},y_{2,d})=(0,20),$
that is the predator tends to dominate and the prey to vanish. %This makes sense e.g. in the case where the growth rate of $y_1$ is very slow,  the rate of $y_1$ is killing $y_2$ is very slow, and the death rate of $y_2$ is very slow too with a fast death rate of $y_1$. %the growth rate of $y_2$ by chances of killing $y_1$
We assume the parameters $\epsilon_1=0.1$, %(prey diffusion), 
$\epsilon_2=0.01$, %(predator diffusion), 
$a=0.47$, 
%(growth rate of prey), 
$b=0.024$, %(searching efficiency/attack rate), 
$c=0.023$, $d=0.76$ as in Experiments A, B, \magent{and for parameter values $\lambda_1=\lambda_2=1$}, while our problem evolves in the time interval $[0,T]=[0,0.1]$ in a square $\Omega=[0,1]\times[0,1]$. %Again, the diffusion parameter values signify that the prey diffuses faster than the predator.
%
% Beginning with given initial conditions for  $y_1$ and $y_2$, we solve the linear parabolic equation over a short period of time trying to drive the values of $y_1$ and $y_2$ at time $t_1$ to the desired target functions $y_{1,d}$, $y_{2,d}$ respectively.
%In general, initial and boundary conditions can be difficult to pin down for problems like this. We assume the domain $\Omega$  and that the predator or the prey enters or exits the domain in a controlled way indicated by the control functions $g_1$, $g_2$. Initially, the predator density is concentrated at the edges of the domain and the prey density is concentrated at the center.
%
%A nonlinear gradient method with strong Wolfe-Powel conditions and Fletcher-Reeves directions is employed, with 
%
The results reported in Table \ref{Exp_tau=h^2Lo,k=0},
\begin{table}[ht]
%\small
\caption{%Example  2.
Robin control: Distance from the target applying $k=0$, $l=1$ ($\tau=\mathcal{O} (h^2) $) and rough data.}
\begin{center}
%\scalebox{0.82}
{\begin{tabular}{|c|c|c|%c|c|
c|}
   \hline
   \multicolumn{1}{|c|}{Discretization} & \multicolumn{2%5
    }{|c|}{Distance from the target}&\multicolumn{1}{|c|}{Functional}    \\ \hline
     $\tau=h^2/2$  &${\left\| y_1-{y_{1,d}} \right\|}_{{L^2[0,T;{ L}^2 (\Omega )] }}$
     &$\left\| y_2-{y_{2,d}} \right\|_{{{L^2[0,T;{ L}^2  (\Omega )] }}}$
     &$J(y,g)$%&  { }      &  { }
     \\ \hline
    $h=0.23570$ &   {0.04265266201}&    {0.57806942650}&    {0.6335395815}%&  { }      &{ }
     \\ \hline
     %    $h=0.23570$ &{}&    { }&{ }%&  { }      &{ }
     %\\ \hline
    $h=0.13888$ &   {0.00534571082}&    {0.10386568510}&    {0.1093306993}%&  { }      &{ }
    \\ \hline
    %$h=0.13888^*$ &   { }&    { }&    { }%&  { }      &{ }
    %\\ \hline
    %    $h=0.13888$&{0.003991792244}&    {0.006977773302	}&    {0.01431706681}%&  { }      &{ }
    %\\ \hline
    $h=0.05892$ &   {0.00130612821}&    {0.02562171016}&    {0.0269393861}%&  { }      &{ }
    \\ \hline
    $h=0.02946$ &   {0.00031866772}&    {0.00648286241}&    {0.0068040695}%&  { }      &{ }
   \\ \hline
 \end{tabular}}
\end{center}
\label{Exp_tau=h^2Lo,k=0}
\normalsize
\end{table}
show that the manipulation is effective and the related distance size from the target may be reduced spectacularly. %In a quite dense mesh, we observe much smaller distances, e.g. for the $h=0.02946$ distance computations, in contrast with the related big distance values for an $h=0.23570$ discretization. Thus, 
As expected, the control functions make the system more efficient as the time-space mesh becomes denser.
 %  Firstly we demonstrate the algorithm for our case in a general form. Afterwards we introduce the basic aspects for the implementation in FreeFem++ software package.
\subsubsection{Experiment E: Linear in time ($k=1$), linear in space polynomials ($l=1$),  smooth data.} \label{R2}
 We switch, and we introduce higher order polynomials in time, enforcing the smooth initial conditions    $y_{1,0}=16+0.25(x_1^2+x_2^2),\,y_{2,0} = 25 %-     \sin(\pi x_1)\sin(\pi x_2)
$ and the same as the previous experiment's target functions, considering Robin boundary control.
   The results are shown in i) Table \ref{Exp_tau=h^2Lo,k=1_l=1} using linear in time and space polynomials and ii) for linear in time and quadratic in space in Table \ref{Exp_tau=h^2Lo,k=1_l=2}  \magent{for parameter values $\lambda_1=\lambda_2=1$}. %, and the  code we used is demonstrated in the Appendix \ref{k=1_Code}.
   \begin{table}[ht]
%\small
\caption{%Example  2.
Robin control:  Distance from the target  applying  $k=1$, $l=1$ ($\tau=\mathcal{O} (h^2) $) and smooth data.}
\begin{center}
%\scalebox{0.82}
{\begin{tabular}{|c|c|c|%c|c|
c|}
   \hline
   \multicolumn{1}{|c|}{Discretization} & \multicolumn{2%5
    }{|c|}{Distance from the target}&\multicolumn{1}{|c|}{Functional}    \\ \hline
     $\tau=h^2/2$  &${\left\| y_1-{y_{1,d}} \right\|}_{{L^2[0,T;{ L}^2 (\Omega )] }}$
     &$\left\| y_2-{y_{2,d}} \right\|_{{{L^2[0,T;{ L}^2  (\Omega )] }}}$
     &$J(y,g)$%&  { }      &  { }
     \\ \hline
    $h=0.23570$ &   {0.4691621648}&    {0.06976235262}&    {0.5392965415}%&  { }      &{ }
\\ \hline
    $h=0.13888$ &   {0.0770592135}&    {0.01100635051}&    {0.1112083441}%&  { }      &{ }
    % \\ %\hline
    %{\red{Re-computation:}}$%(h=0.13888^*)
        %        $ & {\red{{0.0748028268}}}&    {{\red{0.01084693009}}}&   {{\red{0.0857365135}}}%&  { }      &{ }
    \\ \hline
    $h=0.05892$ &   {0.0182574659}&    {0.00242353097}&    {0.0206912036}%&  { }      &{ }
%    \\Re-computation:% \hline
 %    %$(h=0.05892^*)$
%     &   { }&    { }&    { }%&  { }      &{ }
    \\ \hline
    $h=0.02946$ &   {0.0045402776}&    {0.00057613379}&    {0.0051184802} %&  { }      &{ }
   \\ \hline
 \end{tabular}}
\end{center}
\label{Exp_tau=h^2Lo,k=1_l=1}
\normalsize
\end{table}
%-----------------------------
   \begin{table}[ht]
%\small
\caption{%Example  2.
Robin control: Distance from the target  applying  $k=1$, $l=2$ ($\tau=\mathcal{O} (h^2) $) and smooth data.}
\begin{center}
%\scalebox{0.82}
{\begin{tabular}{|c|c|c|%c|c|
c|}
   \hline
   \multicolumn{1}{|c|}{Discretization} & \multicolumn{2%5
    }{|c|}{Distance from the target}&\multicolumn{1}{|c|}{Functional}    \\ \hline
     $\tau=h^2/2$  &${\left\| y_1-{y_{1,d}} \right\|}_{{L^2[0,T;{ L}^2 (\Omega )] }}$
     &$\left\| y_2-{y_{2,d}} \right\|_{{{L^2[0,T;{ L}^2  (\Omega )] }}}$
     &$J(y,g)$%&  { }      &  { }
     \\ \hline
    $h=0.23570$ &   {0.7186089436}&    {0.06769868694}&    {1.367070152}%&  { }      &{ }
\\ \hline
    $h=0.13888$ &   {0.0753016087}&    {0.00951411931}&    {0.084852089}%&  { }      &{ }
    \\ \hline
    $h=0.05892$ &   {0.0182633483}&    {0.00225891391}&    {0.020530010}%&  { }      &{ }
%    \\Re-computation:% \hline
 %    %$(h=0.05892^*)$
%     &   { }&    { }&    { }%&  { }      &{ }
    \\ \hline
    $h=0.02946$ &   {0.0045402412}&    {0.00056394557}&   {0.005106144}%&  { }      &{ }
   \\ \hline
 \end{tabular}}
\end{center}
\label{Exp_tau=h^2Lo,k=1_l=2}
\normalsize
\end{table}
This approach allows \magent{slightly} %faster convergence and 
better results regarding the distance from the target functions. %, even with the Robin boundary conditions. 
\magent{We note that t}hese boundary conditions restrict the regularity in the boundary, so it is strongly not recommended to employ higher-order polynomials in space even if we have a smooth initial \magent{data}  due to the limited regularity near the boundary caused by the Robin boundary conditions. The associated results are reported in Table \ref{Exp_tau=h^2Lo,k=1_l=2}. One also should expect that the experiment with $k=1$, $l=2$ would fail, and almost does in the coarse stepping mesh. Furthermore, the functional obviously has almost three times larger values than in the case with $k=1$, $l=1$ in Table \ref{Exp_tau=h^2Lo,k=1_l=1} focusing on $h=0.23570$. Nevertheless, the parabolic regularity appears as the algorithm execution progresses and thus it is revealed beneficial. %, giving good results.
%
%We recall that a debugging method is used as mentioned in Paragraph \ref{Experiments_Robin}, and the proposed debugging method works well in all experiments, as we can see in the Table \ref{Exp_tau=h^2Lo,k=0}, \ref{Exp_tau=h^2Lo,k=1_l=1}, \ref{Exp_tau=h^2Lo,k=1_l=2}. Particularly, someone could see some recomputations, e.g. for $h=0.13888$, noticing that we have a stabilized behavior with small deviations after computing with the debugging method.   Slight different values are observed for the  two computations in the same mesh. We also have noticed similar behavior for other experiments and again the proposed debugging method works efficiently giving us desirable results.
%

%Next, we introduce some basic aspects of the \magent{algorithm  we} used.
\magent{In \ref{Algo:Opt_Procedure}, we introduce some basic aspects of the implementation and the algorithm we used.}
\subsection{Conclusion}
We examined the Robin boundary and distributed control for a non-linear population model and we introduced the basic concepts of the first and second-order necessary conditions also in cases with control constraints. We demonstrated the related forward in time state and backward in time adjoint system, and a non-linear conjugate gradient method numerically verified that in each case we managed to manipulate the system to desired targets even in cases with rough initial data. We additionally demonstrated non-linear terms of the system qualitative analysis, extending and enriching the literature with the way that a control function changes the system's quality, as well as, indicating the proper physical parameter value choices. We conclude that both the distributed and Robin boundary control for Lotka-Volterra system cases drive efficiently the state variables to desired targets. Nevertheless, the first case seems more stable and with a more robust minimization of the functional, fact justified from the lower regularity in the Robin boundary consideration. These minimization convergence issues in the second case motivated us to employ a non-linear gradient method while its implementation confirmed the aforementioned aspects. As future work, we are planning to extend and investigate in an optimal control framework state-of-the-art methods like the arbitrarily shaped finite element dG method \cite{Ka23}, and/or unfitted mesh finite element methods possibly evolutionary in time and proper time discretization, \cite{ChKa14, KaKaTra23, ArKa22, ArKaKa22}. %\TODO{KaKaTra23,ArKa22, ArKaKa22}. %\TODO{KaKaTra23,ArKa22, ArKaKa22}.
\section*{Acknowledgments}
This project has received funding from %the Hellenic Foundation for Research and Innovation (HFRI) and  the  General  Secretariat  for  Research  and  Technology (GSRT), under  grant agreement No[1115] and  the 
``First Call for H.F.R.I. Research Projects to support Faculty members and Researchers and the procurement of high-cost research equipment'' grant 3270 and was supported by computational time granted from the National Infrastructures for Research and Technology S.A. (GRNET S.A.) in the National HPC facility - ARIS - under project ID pa190902. %The author would like also to thank  Prof. Emmanuil Georgoulis %\EK{and}  Prof. K. Chrysafinos  from NTUA for valuable comments and inspiring ideas. 
We would like also to thank the contributors of the FreeFem++, \cite{He12}. %ngsolve-ngsxfem, \cite{Scho14,LeHePreWa21}.
\appendix
%
%  \section{FreeFem++ code for $k=1$}\label{k=1_Code}
%\begin{appendices}
\magent{
\section{{Optimization procedure}}\label{Algo:Opt_Procedure}
{The algorithm is based on the steepest--descent/projected gradient method combined with Strong  Wolfe-Powel conditions and Fletcher-Reeves conjugate directions to ensure good convergence properties. Specifically, we consider the  Fletcher-Reeves conjugate direction as the search direction to compute the step for this direction. The step  $\epsilon _n$ is derived from a suitable linear search procedure. 
This approach makes us waste more computational resources in memory, but we can reduce significantly the number of iterations of the double iteration loop of the gradient method.
We highlight that we noticed similar behavior with the results of a simple gradient method, but we gain faster convergence. Next, we introduce  the strong Wolfe-Powel conditions with  $0<\rho\le \sigma <1$ and $d_{k+1}=-J'_{k+1}+\beta _{k+1} d_k$,   $d_0=-J'_k$:
%while we choose the Fletcher-Reeves conjugate directions
 \begin{enumerate}
   \item $J(y_{k+1},g_{k+1})\le J(y_{k},g_{k})+\sigma \epsilon _k {J'^T}_k d_k$ ({Armijo rule}),
   \item $|{J'}_{k+1} d_k|\le -\rho {J'}_{k}d_k$,
 \end{enumerate}
with  $0<\rho\le \sigma <1$ and $d_{k+1}=-J'_{k+1}+\beta _{k+1} d_k$,   $d_0=-J'_k$
and choosing the Fletcher-Reeves conjugate directions%Polak-Ribiere}%\footnote{Επίσης δοκιμάστηκαν   και οι επιλογές για τις συζυγείς κατευθύνσεις  \tl{Hestens-Steifel, Fletcher-Reeves, Hager-Zhang.}
 %}
 : $\beta_k=\frac{J'^T_k J'_k} {\|J'_{k-1}\|^2}$. 
\newline
% $\beta_k=\frac{J'^T_k(J'_k-J'_{k-1})} {\|J'_{k-1}\|^2}$.
%\end{rmk}
The pseudocode we used is:
\newline
  %\scalebox{.65}{                        %new code
    \begin{algorithm}[H]         \label{BruAlgorithm}     %new code
{\DontPrintSemicolon
\floatname{algorithm}{\magent{Procedure}}
\caption*{\magent{Optimization approach with preconditioned CG/nonlinear gradient method with strong Wolfe-Powel conditions and Fletcher-Reeves directions.}}
       \magent{ \KwIn{$n=0$,  $\epsilon=1$,  $tol=10^{-5}$,  ${{g^0_i}}_\Gamma=1$, $y_i|_{t=0}$, $y_{id}$, $\sigma =0.1$, $\rho =0.9$, }
        \KwOut{Optimal triple: $(y_i, \mu_i, {g_i}_\Gamma)$}
         %\textcolor[rgb]{0.00,1.00,0.00}{{\% Initialization}} \\
        Solve: State problem for  ${{{g^0_i}}_\Gamma}$\\
        Compute $J^0$\\
        %\textcolor[rgb]{0.00,1.00,0.00}
        {{{\%}Main iteration}} \\
        \Repeat{$J(y^{n+1}_i,{{g^{n+1}_{i\,\,\Gamma}}})> J(y^{n}_i,{g^{n}_i}_\Gamma)+\sigma \epsilon  {{J'}^n}^T d^n_i$  %\textcolor[rgb]{0.00,1.00,0.00}
        {{{\%}{Armijo rule}}} }{
          Solve: Conjugate problem for  $y^{n-1}_i$\\
          Compute  $d^0_i=-J'^0=-\gamma_i {g^0_i}_\Gamma+{\mu^0_i}_\Gamma $\\
          %\textcolor[rgb]{0.00,1.00,0.00}
          {{{\%}Optimization iteration}} \\
          \Repeat{$|{{J'}^{n+1}}^T d^n|> -\rho {{J'}^{n}}^Td^n_i$} {
          %\textcolor[rgb]{0.00,1.00,0.00}
          {{{{\%}}{Fletcher--Reeves	conjugate directions}}}\\
            $\beta^n=\frac{{{J'}^n}^T J'^n} {\|J'^{n-1}\|^2}$  %\textcolor[rgb]{0.00,1.00,0.00}
            {{\%}$%{J'}^n=
{J'({g_i}_\Gamma)}^n=\gamma _i{g^n_i}_\Gamma+{\mu ^n_i}_\Gamma$}\\
            $d^{n+1}_i=-J'^{n+1}+\beta ^{n+1}_i d^n_i$  \\
            ${g^n_i}_\Gamma={g^{n-1}_{i\,\,\Gamma}}+\epsilon d^{n+1}_i$\\
            Solve: State problem for ${g^n_i}_\Gamma$\\
            $\epsilon=0.5*\epsilon$
        }
        \If{$|J^n - J^{n-1}|/J^n \le tol$}{
        Break
            %$C \gets C \cup \{x_i\}$ \;
            }
   %     $SHF_{row} \gets$ Lat.ChooseMin(CandidateSet) \;
    %    Output $X_{row}$, $T_{row}$, $SHF_{row}$ \;
         $\epsilon=1.5*\epsilon$
        }
}}
    \end{algorithm}
Where $i=1,2$ refers to the first and second species respectively and the projection step $\epsilon _n\equiv \epsilon$ is necessary since the direction term, it might not be admissible, see e.g. \cite{HeKu10}. %(\tl{Hinze Pironeu Ulrich Ulrich}).
%\begin{cor}
%
As we can see from the results in Tables \ref{Exp_tau=h^2Lo,k=0}, \ref{Exp_tau=h^2Lo,k=1_l=1}, \ref{Exp_tau=h^2Lo,k=1_l=2} the algorithm is efficient and we can notice how the state solution distance from the target function for each variable and the quadratic functional reduces. Looking carefully at the results, the case with rough initial data is more demanding, while the quadratic functional values are larger in this case. %} %The better results and behavior in Table \ref{Exp_tau=h^2Lo,k=1} are also justified because of using higher order polynomials in space.%, and for completeness, we present the related code in the Appendix \ref{k=1_Code}.
%\end{cor}
%
 % \section{FreeFem++ code for $k=1$}\label{k=1_Code}
%TEST
%\small
%TEST 
%\beginff
%TEST
}
}
%\end{appendices}
%
%
%\newpage
%
\bibliographystyle{amsplain}%\begin{thebibliography}{00}
%\bibliography{}

\begin{thebibliography}{100}
\bibitem{AkrCr2004}
{\sc G. Akrivis,  and  M. Crouzeix}, Linearly implicit methods for nonlinear parabolic equations. {\em Math. Comp.}, {\bf  73},  613-635, 2004.

\bibitem{AkrCrMakr1998}
{\sc G. Akrivis, M. Crouzeix, and C. Makridakis}, Implicit-explicit multistep finite element methods
for nonlinear parabolic problems. {\em Math. Comp.}, {\bf  67}, 457-477, 1998.

 \bibitem{AkrMakr2022}
{\sc G. Akrivis, and C. G. Makridakis}, On maximal regularity estimates for discontinuous Galerkin time-discrete methods. {\em SIAM. J. Numer. Anal.}, {\bf 60}, 180–194, 2022.

\bibitem{AkrMakr2014}
{\sc G. Akrivis, and C. G. Makridakis}, Galerkin time-stepping methods for nonlinear parabolic equations. {\em ESAIM: Math. Model. and Numer. Anal.}, {\bf 38},  261–289, 2004.

\bibitem{ArKa22}
{\sc A. Aretaki, and E. N. Karatzas}, Random geometries for optimal control PDE problems based on fictitious domain FEMS and cut elements, {\em Journal of Computational and Applied Mathematics}, {\bf 412}, 114-286, 2022. 

\bibitem{ArKaKa22}
{\sc A. Aretaki, E. N. Karatzas, and G. Katsouleas}, Equal higher order analysis on an unfitted dG method for Stokes flow systems, {\em Journal of Scientific Computing}, {\bf 91(48)}, 2022. 

\bibitem{Bo96}
{\sc J. F. Bonnans}, Second Order Analysis for Control Constrained Optimal Control Problems
of Semilinear Elliptic Systems. [Research Report] RR-3014 (Inria-00073680), INRIA. 1996.

\bibitem{CaGeJe2022}
{\sc A. Cagniani, E. H. Georgoulis, and M. Jensen}, Discontinuous Galerkin methods for mass transfer through semi-permeable membranes. {\em SIAM. J. Numer. Anal.}, {\bf 51}, 2911– 2934, 2013.

\bibitem{CaVa22}
{\sc H. Carreon, and F. Valdez}, A New Optimization Method Based on the Lotka-Volterra System Equations. In: Castillo, O., Melin, P. (eds) New Perspectives on Hybrid Intelligent System Design based on Fuzzy Logic, Neural Networks and Metaheuristics. Studies in Computational Intelligence, vol 1050. Springer, Cham. 2022.

\bibitem{CaCh12}
{\sc E. Casas, and K. Chrysafinos}, A discontinuous Galerkin time stepping scheme for the velocity tracking problem, {\em SIAM. J. Numer. Anal.}, {\bf 50} (2012), pp. 2281-2306.% (extended version at http://www.math.ntua.gr/~chrysafinos).

\bibitem{Ci}
{\sc P. Ciarlet}, The finite element method for elliptic problems,
SIAM Classics, Philadelphia, 2002.

\bibitem{ChKa15}
{\sc K. Chrysafinos, and E. Karatzas}, Symmetric error estimates for discontinuous Galerkin time-stepping schemes for optimal control problems constrained to evolutionary Stokes equations, {\em Computational Optimization and Applications}, {\bf 60(3)},  719-751, 2015.

\bibitem{ChKa12}
{\sc K. Chrysafinos, and E. Karatzas}, Symmetric error estimates for discontinuous Galerkin approximations for an optimal control problem associated to semilinear parabolic PDEs, {\em Disc. and Contin. Dynam. Syst. - Ser. B}, {\bf 17},  1473 - 1506, 2012.

\bibitem{ChKa14}
{\sc K. Chrysafinos, and E. Karatzas}, Error Estimates for Discontinuous Galerkin Time-Stepping Schemes for Robin Boundary Control Problems Constrained to Parabolic PDEs, {\em SIAM J. Numer. Anal.}, 52(6), 2837–2862, 2014.

\bibitem{ChryKaKo19}
{\sc K. Chrysafinos, E. Karatzas, and D. Kostas}, Stability and error estimates of fully-discrete
schemes for the Brusselator system, {\em SIAM J. Numer. Anal.}, 57(2), 828–853, 2019.

\bibitem{ChryKo22}
{\sc K. Chrysafinos, and D. Kostas}, Numerical Analysis of High Order Time Stepping Schemes for a Predator-Prey System, {\em International Journal of Numerical Analysis and Modeling}, 19(2-3), 404--423, 2022.

\bibitem{ChPla23}
{\sc K. Chrysafinos, and D. Plaka},  Analysis and approximations of an optimal control problem for the Allen–Cahn equation., {\em  Numer. Math.},  155, 35-82, 2023.

\bibitem{Dai-Yuan99}
{\sc Y. H. Dai, and Y. Yuan}, A Nonlinear Conjugate Gradient Method with a Strong Global Convergence Property, {\em SIAM J. Optim.}, 10(1), 177–182, 1999.

\bibitem{DelfHagTro81}
{\sc M. Delfour, W. Hager and F. Trochu}, Discontinuous Galerkin Methods for Ordinary Differential Equations, {\em Math. Comp},  36, 455-473, 1981.

\bibitem{DieGaTre17}
{\sc F. Diele, M. Garvie, and C. Trenchea}, Numerical analysis of a first-order in time implicit symplectic scheme for predator-prey systems, {\em Comput. Math. Appl.}, 74, 948–961, 2017.

\bibitem{DuaZhangMa16}
{\sc X. Dua, P. Zhang, and W. Ma}, Some modified conjugate gradient methods for unconstrained optimization, {\em Journal of Computational and Applied Mathematics}, Volume 305, 92–114, 2016.

\bibitem{E98}
{\sc L. Evans}, Partial Differential Equations, AMS, Providence RI, 1998.

\bibitem{ErickssonJohnson95}
{\sc K. Ericksson and C. Johnson}, Adaptive finite element methods for parabolic problems IV: Nonlinear problems, {\em SIAM J. Numer. Anal.},  32, 1729-1749, 1995.

\bibitem{EstepLarsson93}
{\sc D. Estep and S. Larsson}, The discontinuous Galerkin method for semilinear parabolic equations, {\em RAIRO Model. Math. Anal. Numer.},  27, 35-54, 1993.

\bibitem{FriMadSguVen17}
{\sc M. Frittelli, A. Madzvamuse, I. Sgura, and C. Venkataraman}, Lumped finite elements for
reaction-cross-diffusion systems on stationary surfaces, {\em Comput. Math. Appl.}, {\bf 74(12)},  3008–3023, 2017.

\bibitem{GaMo97}
{\sc J. L. Gámez, and J. A. C. Montero}, Uniqueness of the optimal control for a Lotka-Volterra control problem with a large crowding effect, {\em ESAIM: Control, Optimisation and Calculus of Variations}, {\bf 2},  1-12, 1997.

\bibitem{GarcHiKah19}
{\sc H. Garcke, M. Hinze, and C. Kahle}, Optimal control of time-discrete two-phase flow driven by a diffuse-interface model, {\em ESAIM, Control Optim. Calc. Var.}, {\bf 25(13)}, 1-31, 2019. 

\bibitem{GaBlo05}
{\sc M.R Garvie, and J.M. Blowey}, A reaction-diffusion system of $\lambda$–$\omega$ type II: Numerical Analysis, {\em Euro Jnl of Applied Mathematics}, {\bf 16},  621–646, 2005.

\bibitem{GaTre07}
{\sc M. Garvie, and C. Trenchea}, Optimal Control of a Nutrient-Phytoplankton-Zooplankton-Fish System, {\em SIAM J. Control Optim.}, {\bf 46(3)},  775 - 791, 2007.

\bibitem{GaTre08}
{\sc M. Garvie, and C. Trenchea}, Finite element approximation of spatially extended predator interactions with the Holling type II functional response, {\em  Numer. Math.}, {\bf 107},  641–667, 2008.

\bibitem{GuMa06}
{\sc  M.D. Gunzburger and S. Manservisi},  The velocity tracking problem for Navier–Stokes ﬂow
with boundary control, {\em  SIAM J. Control Optim.}, {\bf 39 },  594–634., 2000.

\bibitem{GuHoZhu06}
{\sc  M.D. Gunzburger, L.S. Hou, and W. Zhu},  Fully-discrete finite element approximation of the
forced Fisher equation, {\em  J. Math. Anal. Appl.}, {\bf 313}, 419-440, 2006.

\bibitem{Ha97}
{\sc A. Hastings}, Population Biology, Springer, Berlin, Germany, 1997.

\bibitem{He12}
{\sc F.  Hecht}, New development in FreeFem++, {\em Journal of numerical mathematics}, {\bf 20(3-4)},  251--265, 2012.   

\bibitem{HeKu10}
{\sc R. Herzog, and K. Kunisch}, Algorithms for PDE-constrained optimization, {\em GAMM-Mitteilungen}, {\bf 33(2)},  163–176, 2010.   

\bibitem{Hol59}
{\sc R. C.S. Holling}, Some characteristics of simple types of predation and parasitism, {\em Can. Entomol.}, {\bf 91},  163–176, 385- 398. 

\bibitem{Hol63}
{\sc C.S. Holling}, The functional response of predators to prey density and its role in mimicry and population regulation, {\em Mem. Entomol. Soc. Can.}, {\bf 45}, 1-60, 1965.

\bibitem{HoRo21}
{\sc M. Holtmannspötter and A. Rösch}, A priori error estimates for the finite element approximation of a nonsmooth optimal control problem governed by a coupled semilinear PDE-ODE system, {\em SIAM J. Control Optim.}, {\bf 5}, 3329-3358, 2021. 

\bibitem{Ib17}
{\sc A. Ibanez}, Optimal control of the Lotka–Volterra system: turnpike property and numerical simulations, {\em Journal of Biological Dynamics}, {\bf 11:1},  25-41, 2017. 

\bibitem{Ja92}
{\sc D. Jackson}, Error estimates for the semidiscrete Galerkin approximations of the FitzHugh-Nagumo equations, {\em Appl. Math. Comput.}, {\bf 50},  93–114, 1992.

\bibitem{Jo09}
{\sc D.S. Jones, Michael Plank, and B.D. Sleeman}, Differential Equations and Mathematical Biology, CHAPMAN \& HALL/CRC, Second Edition, 2009.

\bibitem{JuLiQiaoZha18}
{\sc L. Ju, X. Li, Z. Qiao, and H. Zhang}, Energy stability and error estimates of exponential time differencing schemes for the epitaxial growth model without slope selection, {\em Math. of Comput.}, {\bf 87},  1859–1885, 2018.

\bibitem{Ka15}
{\sc {E. N. Karatzas}}, {Optimal Control and Parabolic Problems, Numerical Analysis and Applications}, {\em http://dx.doi.org/10.26240/heal.ntua.1496}, 2015.

\bibitem{Ka23}
{\sc E. N. Karatzas}, hp-version analysis for arbitrarily shaped elements on the boundary discontinuous Galerkin method for Stokes systems, \magent{{\em under revision}, arXiv:2301.12577, 2024.} 

\bibitem{KaRo21}
{\sc E. N. Karatzas, and G. Rozza}, A Reduced Order Model for a Stable Embedded Boundary Parametrized Cahn–Hilliard Phase-Field System Based on Cut Finite Elements, {\em J Sci Comput}, {\bf 89(9)}, 2021.

\bibitem{KaKaTra23}
{\sc G. Katsouleas, E. N. Karatzas, and F. Travlopanos}, Discrete Empirical Interpolation and unfitted mesh FEMs: application in PDE–constrained optimization, {\em Optimization Journal}, {\bf 72(6)},  1609–1642, 2023.

\bibitem{LinRuizTian14}
{\sc Z. Lin, R. Ruiz-Baier, and C. Tian}, Finite volume element approximation of an inhomogeneous
Brusselator model with cross-diffusion, {\em Journal of Comput. Phys.}, {\bf 256},  806–823, 2014.

\bibitem{Mu03}
{\sc J.D. Murray}, Mathematical Biology II: Spatial Models and Biomedical Applications, Third Edition, Springer, 2003.

\bibitem{NeVe12}
{\sc I. Neitzel, and B. Vexler}, A priori error estimates for
space-time finite element discretization of semilinear parabolic
optimal control problems, {\em Numer. Math.}, {\bf 120},  345-386, 2012.

\bibitem{RoLi13}
{\sc J. F. Rosenblueth, and G. S. Licea}, Cones of critical directions in optimal control, {\em Int. J. Med. Eng. Informat.}, {\bf 7},  55-67, 2013.

\bibitem{Ru95}
{\sc S. Ruuth}, Implicit-explicit methods for reaction-diffusion problems in pattern-formation, {\em J. Math. Biol.}, {\bf 34},  148-176, 1995.

\bibitem{Tr10}
{\sc F. Tr\"oltzsch}, Optimal control of partial differential
equations: Theory, methods and applications, {\em Graduate Studies
in Mathematics, AMS}, {\bf 112}, Providence 2010.

\bibitem{CaTrReyes2008}
{\sc F. Tr\"oltzsch, E. Casas, and	J. C. de los Reyes}, Sufficient Second-Order Optimality Conditions for Semilinear Control Problems with Pointwise State Constraints, {\em SIAM Journal on Optimization}, {\bf 19(2)}, 2008,  616-643.

\bibitem{Volterra78}
{\sc V. Volterra}, Variations and fluctuations in the numbers of coexist Theoretical Ecology: 1923-1940, (Eds. F.M. Scudo and J.R. Ziegler), Lecture Notes in Biomathematics, Vol. 22, 65-236, Springer,
Berlin, Heidelberg, 1978.
%
\end{thebibliography}

%============================================================================================================================

\end{document}